\documentclass[11pt]{article}
\usepackage{amsmath}
\usepackage{amsthm}
\usepackage[xetex]{graphicx}
\usepackage{epstopdf}
\usepackage{amssymb,enumerate}
\usepackage{graphicx}
\usepackage{latexsym}
\newtheorem{thm}{Theorem}[section]

\newtheorem{lem}[thm]{Lemma}
\newtheorem{prop}[thm]{Proposition}

\newtheorem{defn}[thm]{Definition}

\newtheorem{rem}[thm]{Remark}

\makeatletter

\newcommand{\Rmnum}[1]
{expandafter\@slowromancap\romannumeral#1@}
\makeatother
\numberwithin{equation}{section}
\numberwithin{figure}{section}
\allowdisplaybreaks[4]
\begin{document}
	\title{Large deviation principles for 2D stochastic $g$-Navier--Stokes equations on unbounded channel-like domains}
	\author{Gang He$^{1}$, \,\,\,Ji Shu$^{1}$\thanks{Corresponding author. \textit{E-mail addresses}:  hegang2025@stu.sicnu.edu.cn(G.He), shuji@sicnu.edu.cn(J.Shu), wenhuixia@cafuc.edu.cn(H.Wen),  wulongyul@163.com(L.Wu).},\,\,\,Huixia Wen$^{1,2}$,\,\,\,Longyu Wu$^{3}$\small\\
		\textit{1.School of Mathematical Sciences, Sichuan Normal University},\\ \textit{Chengdu 610068, PR China}	\small\\
		\textit{2.Faculty of Science, Civil Aviation Flight University of China},\\ \textit{Chengdu 618307, PR China}    \small\\
		\textit{3.College of Applied Mathematics, Chengdu University of Information Technology},\\ \textit{Chengdu 618307, PR China}}
		
	\date{}
	
	\maketitle
	
	{\begin{center}
			\begin{minipage}{6.0in}
				{\textbf{Abstract:} This paper investigates the large deviation principle for 2D stochastic $g$-Navier--Stokes equations driven by multiplicative noise and defined on an unbounded channel-like domain. We first prove the existence and uniqueness of probabilistically strong solutions for  the stochastic equation under small noise perturbations by means of the local monotonicity method. Subsequently, we show the well-posedness of solutions to the corresponding control equation and the convergence in distribution of the solutions for the stochastic equation as the noise intensity tends to zero. Finally, we establish the Freidlin-Wentzell type large deviation principle for the solutions of the stochastic equation by using the weak convergence method. To overcome the lack of compactness of Sobolev embedding on unbounded domains, we employ a scalar stream function together with a truncation technique to derive uniform tail estimates for the solutions of the control equation. }\\
				
		\end{minipage}
	\end{center}}
	\section{Introduction}
	\indent
	\par
	Let $\mathcal{O}$ be an unbounded domain in $\mathbb{R}^{2}$ with smooth boundary $\partial \mathcal{O}$. In this paper, we study the large deviation principle (LDP) for 2D stochastic $g$-Navier--Stokes equations driven by multiplicative noise:\\
	\begin{equation}
		\label{eq:1.1}
		\begin{aligned}
			\begin{split}
				\left \{
				\begin{array}{ll}
					du^{\varepsilon}-(\nu\Delta u^{\varepsilon}- (u^{\varepsilon}\cdot \nabla )u^{\varepsilon} - \nabla p)dt=f(t)dt+\sqrt{\varepsilon}\sigma(t,u^{\varepsilon})dW(t), & x\in\mathcal{O}, t>s,\\
					\nabla\cdot(gu^{\varepsilon})=0,                                             & x\in\mathcal{O}, t>s, \\
					u^{\varepsilon}=0,                                                           & x\in\partial\mathcal{O}, \\
					u^{\varepsilon}(s, x)=u_{0}(x),                                                   & x\in\mathcal{O},
				\end{array}
				\right.
			\end{split}
		\end{aligned}
	\end{equation}
	where $\nu > 0$ is the viscosity coefficient, $\varepsilon \in (0,1)$ is the noise-intensity
	parameter, $p$ denotes the pressure, $u$ is the velocity field, $\Delta$ is the Laplacian operator with Dirichlet boundary conditions, $\sigma$ is a Lipschitz continuous mapping, and $W(t)$ is the driving Wiener process. 
	\par 
	We assume that the domain supports the weighted Poincar\'e inequality
	\begin{equation}
		\label{eq:2.2}
		\int_{\mathcal{O}}^{}|u|^{2}gdx \le \frac{1}{\lambda_{1}}\int_{\mathcal{O}}^{}|\nabla u|^{2}gdx ,\quad  \forall u \in H^{1}_{0}(\mathcal{O}),
	\end{equation}
	for some $\lambda_{1}>0$. We further assume that the weight function $g = g(x_{1}, x_{2}) \in W^{1, \infty}(\mathcal{O}) $ is real-valued and that there exist positive constants $m_{0}$ , $M_{0}$ and $\lambda_{1}$ such that
	\begin{equation}
		\label{eq:1.2}
		0 < m_{0} \le g(x_{1}, x_{2}) \le M_{0},\; for\,\,all (x_{1},  x_{2}) \in \mathcal{O}, \quad \|\nabla g\|_{\infty} \le \dfrac{1}{16}m_{0}\sqrt{\lambda_{1}}.
	\end{equation} 
\par
The Navier--Stokes equations, as the classical fluid dynamics equations, are extensively employed to model the motion of viscous incompressible fluids. Their applications span various scientific and engineering fields, including meteorology, oceanography, and engineering, as well as they have considerable theoretical and practical relevance (see, e.g., \cite{ref R. Temam1, ref R. Temam2}). In the past two decades, substantial progress has been made in the study of 2D stochastic Navier--Stokes equations. The well-posedness, invariant measures, and attractors associated with 2D stochastic Navier--Stokes equations have been widely investigated in the literature (see, e.g., \cite{ref Brzezniak. Z1, ref Brzezniak. Z2, ref Da Prato1,ref Sell. G.R,  ref Liu. R1, ref Zhao. C}).
\par 
Roh \cite{ref Roh. J1} first introduced the $g$-Navier--Stokes equations in 1999. In this model, the varying vertical thickness of a three-dimensional domain is represented by a strictly positive spatial function $g=g(x_{1}, x_{2})$. Under this framework, the classical incompressibility condition is modified into the weighted form $\nabla \cdot (gu) = 0$. This formulation allows the fluid dynamics to be described on a 2D domain $\mathcal{O} \subset \mathbb{R}^2$, while retaining the profound impact of the topographic variations of a 3D thin domain on the dynamic evolution of the fluid. Subsequently, the well-posedness, attractors, and asymptotic behavior of the $g$-Navier--Stokes equations have been explored in \cite{ref Bae. H, ref Li. Y1, ref Li. Y2, ref Yan. T, ref Quyet. D.T}. To the best of the authors' knowledge, the LDP for the 2D stochastic $g$-Navier--Stokes equations on unbounded channel-like domains has not yet been investigated.
\par
The main purpose of this paper is to investigate the LDP for strong solutions of 2D stochastic $g$-Navier--Stokes equations driven by multiplicative noise on channel-like domains. We use the 
weak convergence method, which is based on the variational representation of certain functionals of Brownian motion and the equivalence between the LDP and Laplace principle (see \cite{ref A. Budhiraja1, ref A. Budhiraja2}). General treatments of large deviations and their applications are given in \cite{ref A. Dembo, ref P. Dupuis, ref Da Prato3}. Related LDPs have been obtained for stochastic ordinary differential equations \cite{ref Bao. J}, stochastic lattice systems \cite{ref Chen. Z1, ref Wang. B1}, and stochastic partial differential equations \cite{ref Brzezniak. Z2, ref Cerrai. S, ref Duan. J, ref Hong. W}.
\par 
It is noteworthy that on bounded domains, the LDP for stochastic differential equations can be effectively established through the weak convergence method, which relies on the variational representation developed by Budhiraja and Dupuis. For example, Chang \cite{ref M.-H. Chang} established a large deviation result of Freidlin--Wentzell-type for the 2D stochastic Navier--Stokes equations driven by additive noise via the Girsanov transformation. Caraballo et al. \cite{ref Caraballo. T} employed the weak convergence method to derive a large deviation principle of Freidlin--Wentzell-type for the 3D globally modified stochastic Navier--Stokes equations. This is due to the fact that on bounded domains one can apply the compact embedding theorems for Sobolev spaces.  As the noise intensity $\varepsilon \to 0$, one can employ Galerkin approximations and a priori estimates to extract weakly convergent subsequences, as well as thereby establish the tightness of the distributions of the solution family.
\par 
However, when attempting to establish the LDP on unbounded domains, a major difficulty arises from 
the failure of the Sobolev compact embedding theorem. To overcome this obstacle, Wang \cite{ref Wang. B3} introduced uniform tail estimates for solutions and successfully established the LDP for fractional stochastic reaction-diffusion equations on $\mathbb{R}^n$ with polynomial drift of arbitrary order using the weak convergence method. Similarly, Chen \cite{ref Chen. Z2} extended the polynomial drift condition to more general superlinear drift conditions and obtained LDP results for stochastic nonclassical reaction-diffusion equations. On the other hand, for a class of stochastic Navier-Stokes equations with Lipschitz drift satisfying linear growth conditions, the weak convergence approach remains applicable even in unbounded domains. For example, Sritharan \cite{ref S. S. Sritharan} established LDPs for the 2D stochastic Navier-Stokes equations driven by multiplicative noise, both in bounded and unbounded settings.
\par 
Note that since the stochastic $g$-Navier--Stokes equations form a vector system subject to the incompressibility constraint $\nabla \cdot (gu) = 0$, the technique of uniform tail estimates used for stochastic reaction-diffusion equations does not carry over directly to the stochastic $g$-Navier--Stokes equations. Specifically, both the stochastic Navier--Stokes equations and the stochastic $g$-Navier--Stokes equations contain the pressure term $\nabla p$. Compared with the classical Navier--Stokes setting, the dissipative terms involving the $g$-Laplace operator are handled via a generalized Poincar\'e inequality adapted to the channel-like geometry. In order to apply the truncation technique to obtain uniform tail estimates for the solutions, uniform estimate of the pressure term on the unbounded domain $\mathcal{O}$ is required, which is generally difficult to obtain on an unbounded domain. Alternatively, the projection method may be used to eliminate the pressure term. Note that although the equation itself no longer contains the pressure term and the truncation technique can be applied to obtain uniform tail estimates for the solutions, the incompressibility condition no longer hold.  Specifically, even if the original velocity field satisfies the weighted incompressibility condition (i.e., $\nabla \cdot (gu) = 0$), applying a cut-off  function to it generally results in a field that the truncated field generally fails to satisfy the same weighted incompressibility condition.
\par 
Following Wang \cite{ref Wang. B4}, we introduce a scalar stream function that eliminates the pressure and preserves the divergence structure. This transformation reduces the stochastic $g$-Navier--Stokes system to a scalar equation; see \eqref{eq:3.70}. Then, the truncation technique is applied to derive uniform tail estimates for the solutions of the associated control equation. The compactness of the solution family for the controlled equation plays a key role in establishing the convergence in distribution of the solutions of the stochastic equation as $\varepsilon \to 0$ (see Lemma \ref{lem:1.20}), and this convergence is then essential for deriving the LDP for the stochastic equation (see Theorem \ref{thm:1.12}). 
\par 
To establish the LDP for the 2D stochastic $g$-Navier--Stokes equations, we employ the weak convergence method, which relies on the equivalence between the Laplace principle (LP) and LDP. Next, we use a scalar stream function to transform the control equation corresponding to \eqref{eq:1.1} into a scalar equation (see \eqref{eq:3.70}), and show that the solutions of the control equation are uniformly small outside a bounded domain on a finite time interval (see Lemma \ref{lem:3.6}). Based on the uniform tail estimates of the solutions and the compactness of embeddings in bounded domains, we then prove the compactness of a family of solutions of the control equation (see Lemma \ref{lem:3.38}). Finally, we establish the required convergence as the noise intensity tends to zero (see Lemma \ref{lem:1.20}) and prove the LDP for \eqref{eq:1.1} in $C([0,T]; H_{g}) \cap L^{2}(0,T; V_{g})$ using the Skorokhod representation theorem.
\par 
This paper is organized as follows. In Section 2, we discuss the existence and uniqueness of strong solutions to equation \eqref{eq:1.1} and introduce the weak convergence method for the LDP. In Section 3, we establish the LDP for \eqref{eq:1.1} in $C([0,T]; H_{g}) \cap L^{2}(0,T; V_{g})$. 
\par 
\section{Preliminaries}
\indent
\par
In this section, we first discuss some assumptions regarding the nonlinear term in \eqref{eq:1.1}, which are necessary for establishing the LDP for \eqref{eq:1.1} defined on $\mathcal{O}$. Then, we will give a proof of the existence and uniqueness of the strong solution of \eqref{eq:1.1}. In addition, we recall the weak convergence theory of the LDP for stochastic systems as introduced in references \cite{ref A. Budhiraja1, ref P. Dupuis}.
\par
\subsection{Well-posedness of  2D stochastic $g$-Navier-Stokes equations}
\par 

We first introduce the functional setting for \eqref{eq:1.1}. Let $H_{g}(\mathcal{O})$ be the closure of $\mathcal{V}(\mathcal{O})$ in $\mathbb{L}^{2}_{g}(\mathcal{O}):=L^{2}_{g}(\mathcal{O})^{2}$, where
\[
\mathcal{V}(\mathcal{O})=\{u \in C^{\infty}_{0}:\nabla \cdot (gu)=0\}.
\]
Then $H_{g}(\mathcal{O})$ is a Hilbert space with the weighted inner product
\[
(u,v)_{g}= \int_{\mathcal{O}}u \cdot vg dx = \sum_{i=1}^{2}(u_{i},v_{i})_{g}, 
\]
and the norm
\[
\|u\|_{g}=(u,u)^{\frac{1}{2}}_{g}, \; \forall u \in H_{g}(\mathcal{O}).
\]
Let $V_{g}(\mathcal{O})$ be the closure of $\mathcal{V}(\mathcal{O})$ in $\mathbb{H}^{1}_{0,g}(\mathcal{O})^{2}:= H^{1}_{0,g}(\mathcal{O})^{2}$. The inner product is given by
\[
((u,v))_{g}:=(Du,Dv)_{g} = \sum_{i,j=1}^{2} \int_{\mathcal{O}}^{} \frac{\partial{u_{j}}}{\partial{x_{i}}}\frac{\partial{v_{j}}}{\partial{x_{i}}}gdx, 
\]
with associated norm
\[
\|u\|_{V_{g}} = (Du, Du)^{\frac{1}{2}}_{g}, \; \forall \; u \in V_{g}(\mathcal{O}), 
\]
where $Du =(\dfrac{\partial u_{1}}{\partial_{ x_{1}}}, \dfrac{\partial u_{1}}{\partial_{ x_{2}}}, \dfrac{\partial u_{2}}{\partial_{ x_{1}}}, \dfrac{\partial u_{2}}{\partial_{ x_{2}}})$. We now define the $g$-Laplace operator by
\[
- \Delta_{g}u= -\frac{1}{g}(\nabla \cdot g\nabla)u=-\Delta u-\frac{1}{g}(\nabla g \cdot \nabla)u.
\]
Then, the first equation in \eqref{eq:1.1} can be rewritten as
\begin{equation}
	\label{eq:1.100}
		du^{\varepsilon}-(\nu\Delta_{g}u^{\varepsilon}-\frac{\nu}{g}(\nabla g \cdot \nabla)u^{\varepsilon}- 
		(u^{\varepsilon}\cdot \nabla )u^{\varepsilon} - \nabla p)dt=f(t)dt+\sqrt{\varepsilon}\sigma(t,u^{\varepsilon})dW(t),  x\in\mathcal{O}, t>s.
\end{equation}
By the $g$-orthogonal projection $P_{g}: \mathbb{L}^{2}_{g}(\mathcal{O}) \to H_{g}(\mathcal{O})$ , we define the $g$-Stokes operator \cite{ref Roh. J2} as
\[
A_{g}u=-P_{g}\Delta u,\; \langle A_{g}u,v \rangle_{g} = \langle A_{g}u,gv\rangle = ((u,v))_g,\quad \forall u,v \in V_{g}(\mathcal{O}).
\]
\par 
Let $V_{g}^{*}(\mathcal{O})$ be the dual space of $V_{g}(\mathcal{O})$. We have the dense, continuous embedding
\[
V_{g} \hookrightarrow H_{g} = H^{*}_{g} \hookrightarrow V^{*}_{g}.
\]
The operator $A_{g} : V_{g}(\mathcal{O}) \to V_{g}^{*}(\mathcal{O})$ is bounded with $\| A_{g} \| \le \frac{1}{m_{0}}$.
Indeed, for $\|v \|_{V_{g}} = 1$,
\begin{equation}
	\label{eq:2.57}
	\begin{aligned}
	| \langle A_{g}u, v \rangle | \le \frac{\|u\|_{V_{g}(\mathcal{O})}\|v\|_{V_{g}(\mathcal{O})}}{g(x)} \le \frac{1}{m_{0}}\|u\|_{V_{g}(\mathcal{O})}\|v\|_{V_{g}(\mathcal{O})},	
	\end{aligned}
\end{equation}
which means $\| A_{g}u \|_{V_{g}^{*}(\mathcal{O})}  \le \frac{1}{m_{0}}\|u\|_{V_{g}} $ and so $\| A_{g}u \| \le \frac{1}{m_{0}}\|u\|_{g}$. 
\par
We define the bilinear operator $B_{g} : V_{g}(\mathcal{O}) \times V_{g}(\mathcal{O}) \to  V^{*}_{g}(\mathcal{O})$ and the trilinear form $b_{g}: V_{g}(\mathcal{O}) \times V_{g}(\mathcal{O}) \times V_{g}(\mathcal{O}) \to \mathbb{R}$ by 
\[
\langle B_{g}(u, v), w \rangle_{g} = b_{g}(u, v, w) = \sum_{i,j=1}^{2} \int_{\mathcal{O}}u_{i} \frac{\partial{v_{j}}}{\partial{x_{i}}}w_{j} g dx,\quad \forall u,v,w \in V_{g}(\mathcal{O}).
\]
Now, we define the linear operator $R_{g} : V_{g}(\mathcal{O}) \to  V^{*}_{g}(\mathcal{O})$ as follows:
\[
R_{g}u =\left(\dfrac{1}{g}(\nabla g \cdot \nabla)u \right),\quad \forall u \in V_{g}(\mathcal{O}).
\]
By \cite{ref Quyet. D.T}, for $u, v, w$ in the corresponding spaces, we have
\begin{equation}
	\label{eq:2.8}
	\|B_{g}(u,u)\|_{V^{*}}  \le \|u \|_{g} \|u\|_{V_{g}},
\end{equation}
\begin{equation}
	\label{eq:2.3}
	b_{g}(u, v, v) = 0, \quad b_{g}(u, v, w) =- b_{g}(u,w, v),
\end{equation}
\begin{equation}
	\label{eq:2.4}
	|b_{g}(u, v, w)| \le c\|u\|^{\frac{1}{2}}_{g}\|u\|^{\frac{1}{2}}_{V_{g}}\|v\|^{\frac{1}{2}}_{V_{g}}\|A_{g}v\|^{\frac{1}{2}}_{g}\|w\|_{{g}},
\end{equation}
\begin{equation}
	\label{eq:2.5}
	|b_{g}(u, v, w)|\; \le\; c\|u\|^{\frac{1}{2}}_{g}\|A_{g}u\|^{\frac{1}{2}}_{g}\|v\|_{V_{g}}\|w\|_{g}.
\end{equation}
\par
Assume that the external body force  $f(t)$  takes values in $V^{*}_{g}$  for all $t$. Let $Q$ be a positive, symmetric, trace class operator on $H_{g}$.
\begin{defn}
	\label{defn:1.1}
	A stochastic process $\{W (t) : 0 \le t \le T \}$ is said to be an $H_{g}$-valued $\{\mathcal{F}_{t}\}$-adapted Wiener process with covariance operator $Q$ if \\
	(1) for each non-zero $h \in H_{g}$, $|Q^{1/2}h|^{-1}(W(t), h)$ is a standard one-dimensional Wiener process, and \\
	(2) for any $h \in H_{g}$, $(W (t), h)$ is a martingale adapted to $\{\mathcal{F}_{t}\}$.
	
\end{defn}
Let $H_{0,g} = Q^{1/2} H_{g}$. Then $H_{0,g}$ is a Hilbert space with the inner product
\begin{equation}
	\label{eq:1.3}
(u, v)_{0} = (Q^{-1/2}u, Q^{-1/2}v) \quad \forall u, v \in H_{0,g}.
\end{equation}
Let $\| \cdot\|_{0}$ denote the norm in $H_{0,g}$. Because $Q$ is trace class, the embedding $H_{0,g}  \hookrightarrow H_{g}$ is Hilbert--Schmidt.
\par
In the integral formulation of equation \eqref{eq:1.1}, the stochastic forcing term is represented by $\int_{0}^{t}\sigma(r,u^{\varepsilon}(r))dW(r)$. The assumptions on the noise coefficient $\sigma$ are specified below. Let $L_{Q}$ denote the space of bounded linear operators $S:H_g\to H_g$ such that $SQ^{1/2}$ is a Hilbert--Schmidt operator on $H_g$. We equip $L_Q$ with the norm $\|S\|_{L_Q}^2 = \operatorname{tr}(SQS^*)$,
where $S^{*}$ is the adjoint operator of $S$. 
\par
The noise coefficient $\sigma: [0, T] \times V_{g} \to  L_{Q}(H_{0,g}; H_{g})$ is assumed to satisfy the following hypotheses:\\
(\textbf{A1}). The function $\sigma$ belongs to $([0, T ] \times V_{g} ; L_{Q}(H_{0,g}; H_{g}))$.\\
(\textbf{A2}). For all $t \in (0, T)$, there exists a positive constant $K$ such that 
\[
\|\sigma(t, u)\|^{2}_{L_{Q}} \le K (1 +\|u\|_{V_{g}}^{2}).
\]
(\textbf{A3}). For all $t \in (0, T)$, there exists a positive constant $L$ such that for all $u, v \in V_{g}$,
\[
 \| \sigma(t, u) - \sigma(t, v) \|^{2}_{L_{Q}} \le L\|u-v\|^{2}_{V_{g}}.
\]
By applying the $g$-orthogonal projection $P_{g}$ to each term of the $g$-Navier-Stokes system and combining it with the above definitions, we can rewrite equation \eqref{eq:1.100} as follows:
\begin{equation}
	\label{eq:2.6}
	du^{\varepsilon}+(\nu A_{g}u^{\varepsilon} + B_{g}(u^{\varepsilon}, u^{\varepsilon}) + \nu R_{g}u^{\varepsilon})dt = f(t)dt + \sqrt{\varepsilon}\sigma(t,u^{\varepsilon})dW(t),\quad for \,\,  t > s , 
\end{equation}
with initial condition
\begin{equation}
	\label{eq:2.7}
	u(s)=u_{0}.
\end{equation}
\begin{lem}\cite{ref Bae. H}
	\label{lem:3.65}
	Let $u \in L^{2}(0, T;V_{g})$,  then the function $R_{g}u$ defined by
	\begin{equation}
		\label{eq:2.9}
		\begin{aligned}
		(R_{g}u(t), v)_{g} = \left(P_{g}\left(\dfrac{\nabla g}{g} \cdot \nabla \right)u, v\right)_{g} = \left(\left(\dfrac{\nabla g}{g} \cdot \nabla \right)u, v\right)_{g}, \forall \, v\in V_{g}, 
		\end{aligned}
	\end{equation}
	belongs to $L^{2}(0,T; H_{g})$, and therefore also belongs to $L^{2}(0, T; V^{*}_{g})$. Moreover,
	\begin{equation}
		\label{eq:2.10}
		\begin{aligned}
			\|R_{g}u(t)\|_{g} \le  \dfrac{\|\nabla g\|_{\infty}}{m_{0}}\|u\|_{V_{g}},\,\, for\,\, a.e. \,\,t \in (0, T),
		\end{aligned}
	\end{equation}
	and
	\begin{equation}
		\label{eq:2.11}
		\begin{aligned}
			\|R_{g}u(t)\|_{V^{*}} \le  \dfrac{\|\nabla g\|_{\infty}}{m_{0}\lambda^{1/2}_{1}}\|u\|_{V_{g}},\,\, for\,\, a.e. \,\,t \in (0, T).
		\end{aligned}
	\end{equation}
\end{lem}
\begin{proof}
	To prove \eqref{eq:2.10}, by \eqref{eq:2.9} and duality, for any $v \in H_{g}$, 
	\begin{equation}
		\label{eq:3.67}
		\begin{aligned}
		|(R_{g}u,v)_{g}| &= \left|\int_{\mathcal{O}}^{}\sum_{i,j=1}^{2}(\frac{\partial_{i}g }{g}\partial_{i}u_{j}) v_{j}gdx \right|\le \int_{\mathcal{O}}^{}|\frac{\nabla g}{g}||\nabla u| |v| g dx
		\\& \le \frac{\|\nabla g\|_{\infty}}{m_{0}} \int_{\mathcal{O}}^{}|\nabla u| |v| g dx
	 \le \frac{\|\nabla g\|_{\infty}}{m_{0}} \|u\|_{V_{g}} \|v\|_{g}.
		\end{aligned}
	\end{equation}
	By the definition of the norm, $\| R_{g}u(t)\|_{g} = \underset{|v| \le 1}{\sup} |(R_{g}u,v)_{g}| \le \frac{\|\nabla g\|_{\infty}}{m_{0}} \|u\|_{V_{g}}$. Since $u \in L^{2}(0,T; V_{g})$, this estimate implies that $R_{g}u \in L^{2}(0,T; H_{g})$. Similarly, 
	\begin{equation}
		\label{eq:3.68}
		\begin{aligned}
			\| R_{g}u(t)\|_{V^{*}} = \underset{v\in V_{g}, |v| \le 1}{sup} |(R_{g}u,v)_{g}|
			\le \frac{\|\nabla g\|_{\infty}}{m_{0}} \|u\|_{V_{g}} \|v\|_{g}
			\le \frac{\|\nabla g\|_{\infty}}{m_{0} \lambda^{1/2}_{1}} \|u\|_{V_{g}} \|v\|_{V_{g}}. 
		\end{aligned}
	\end{equation}
	The last step of the inequality employs the Poincar\'e inequality. 
	\end{proof}
\par
We first recall the Ladyzhenskaya interpolation inequality on two-dimensional unbounded domains. It will be used to derive a priori estimates for the solution. 
\begin{lem}
	\label{lem:1.2}
	For any real-valued smooth functions $\mathbf{f}$ and $\mathbf{g}$ with compact support in $\mathbb{R}^{2}$, the following hold:\\
	1. $\| \mathbf{f}^{2}\mathbf{g}^{2}\|_{L^{1}} \le \| \mathbf{f}\|_{L^{2}}\| \mathbf{g}\|_{L^{2}}\| \nabla\mathbf{f}\|_{L^{2}}\| \nabla\mathbf{g}\|_{L^{2}}$,\\
	2. $\|\mathbf{f}\|^{4}_{L^{4}} \le 2 \|\mathbf{f}\|^{2}_{L^{2}}\| \nabla\mathbf{f}\|^{2}_{L^{2}}$.
\end{lem}
\begin{proof}
	For any $(x, y)$, by the fundamental theorem of calculus, the function can be written as
	\begin{equation}
		\label{eq:1.4}
		\mathbf{f}(x,y) = \int_{-\infty}^{x} \partial_{1}\mathbf{f}(s,y)gds = - \int_{x}^{-\infty} \partial_{1}\mathbf{f}(s,y)gds.
	\end{equation}
	\begin{equation}
		\label{eq:1.5}
		\mathbf{g}(x,y) = \int_{-\infty}^{y} \partial_{2}\mathbf{g}(x,t)gdt = - \int_{y}^{-\infty} \partial_{2}\mathbf{g}(x,t)gdt.
	\end{equation}
	By \eqref{eq:1.4} we have
	\begin{equation}
		\label{eq:1.6}
		|\mathbf{f}(x,y)| = \frac{1}{2} \left|\left[ \int_{-\infty}^{x}\partial_{1}\mathbf{f}(s,y)gds +\int_{x}^{-\infty} -\partial_{1}\mathbf{f}(s,y)gds\right]\right| \le \frac{1}{2} \int_{-\infty}^{\infty} |\partial_{1}\mathbf{f}(s,y)g|ds. 
	\end{equation}
	In the same way 
	\begin{equation}
		\label{eq:1.7}
		|\mathbf{g}(x,y)| \le \frac{1}{2} \int_{-\infty}^{\infty} |\partial_{2}\mathbf{g}(x,t)g|dt.
	\end{equation}
	Using \eqref{eq:1.6} and \eqref{eq:1.7}, and integrating with respect to $x$ and $y$,
	\begin{equation}
		\label{eq:1.8}
		\left|\int_{-\infty}^{\infty}\mathbf{f}(x,y)\mathbf{g}(x,y)dxdy \right|\le \frac{1}{4} \int_{-\infty}^{\infty} |\partial_{1}\mathbf{f}(s,y)g|ds\int_{-\infty}^{\infty} |\partial_{2}\mathbf{g}(x,t)g|dt.
	\end{equation}
	\eqref{eq:1.8} is equivalent to
	\begin{equation}
		\label{eq:1.9}
		\| \mathbf{f}\mathbf{g}\|_{L^{1}} \le \frac{1}{4} \|\partial_{1}\mathbf{f}\|_{L^{1}}\|\partial_{2}\mathbf{g}\|_{L^{1}}.
	\end{equation} 
	Replacing $\mathbf{f}$ and $\mathbf{g}$ by $\mathbf{f}^{2}$  and $\mathbf{g}^{2}$ , respectively, it yields 
	\begin{equation}
		\label{eq:1.10}
		\| \mathbf{f}^{2}\mathbf{g}^{2}\|_{L^{1}} \le  \|\mathbf{f}\partial_{1}\mathbf{f}\|_{L^{1}}\|\mathbf{g}\partial_{2}\mathbf{g}\|_{L^{1}}.
	\end{equation}
	Using the Cauchy-Schwarz inequality twice gives
	\begin{equation}
		\label{eq:1.11}
		\begin{aligned}
		\| \mathbf{f}^{2}\mathbf{g}^{2}\|_{L^{1}} &\le  \|\mathbf{f}\|_{L^{2}}\|\mathbf{g}\|_{L^{2}}\|\partial_{1}\mathbf{f}\|_{L^{2}}\|\partial_{2}\mathbf{g}\|_{L^{2}}
		\\&\le \|\mathbf{f}\|_{L^{2}}\|\mathbf{g}\|_{L^{2}}\|\nabla\mathbf{f}\|_{L^{2}}\|\nabla\mathbf{g}\|_{L^{2}}.
	\end{aligned}
	\end{equation}
	Therefore, the proof of the first inequality is completed. Setting $\mathbf{f} =\mathbf{g}$ in \eqref{eq:1.11}, and using the Cauchy-Schwarz inequality, it yields
\begin{equation}
	\label{eq:1.12}
	\|\mathbf{f}\|^{4}_{L^{4}} \le 2 \|\mathbf{f}\|^{2}_{L^{2}}\| \nabla\mathbf{f}\|^{2}_{L^{2}}.
\end{equation}
\end{proof}
\begin{rem}
	\label{rem:1.3}
	From Lemma \ref{lem:1.2}, it follows that for any $u \in H_{g}$,
	\begin{equation}
		\label{eq:1.16}
	\|u \|^{4}_{L^{4}} \, \le \,  \frac{2M_{0}}{m^{2}_{0}}\, \|u\|^{2}_{{g}} \,\|u \|^{2}_{V_{g}}.
	\end{equation}
If $u = (u_{1}, u_{2})$, then $\| u\|^{2}_{g} $ $= \sum_{i=1}^{2}(u_{i},u_{i})_{g}$ and $\| u\|^{2}_{V_{g}}$ $= \sum_{i,j=1}^{2} \int_{\mathcal{O}}^{}\frac{\partial{u_{j}}}{\partial{x_{i}}}\frac{\partial{v_{j}}}{\partial{x_{i}}}gdx$. The bound given by \eqref{eq:1.16} holds for both bounded and  unbounded domains.
\end{rem}

\begin{prop}
	\label{prop:1.4}
	For a given $r > 0$, let $B_{r}$ denote the $L^{4}(\mathcal{O})$ ball in $V_{g} : B_{r} = \{v \in V_{g} :\|v\|_{L^{4}(\mathcal{O})} \le r \}$. Define the nonlinear operator $F$ on $V_{g}$ by $F(u) :=-\nu A_{g}u - B_{g}(u,u)-\nu R_{g}u$. Then for  any $0 < \varepsilon \le \frac{9 \nu }{16L}$, where $L$ is the constant that appears in the condition (\textbf{A3}), the pair $(F, \sqrt{\varepsilon}\sigma)$  is monotone in $B_{r}$ : i.e. for any $u \in  V_{g}$ and $v \in B_{r}$, then
	\begin{equation}
		\label{eq:1.13}
		(F(u)-F(v),u-v)_{g} - \frac{r^{4}}{\nu^{3}}\|u-v\|^{2}_{g} + \varepsilon\|\sigma(t,u) -\sigma(t,v)\|^{2}_{L_{Q}} \le 0. 
	\end{equation}
\end{prop}
\begin{proof}
	We first note that
	\begin{equation}
		\label{eq:1.14}
		\begin{aligned}
		(F(u)-F(v),u-v)_{g} &=-\nu(A_{g}(u-v),u-v)_{g} - (B_{g}(u,u)-B_{g}(v,v), u-v)_{g} 
		\\& - \nu(R_{g}(u-v), u-v)_{g}.
		\end{aligned}
	\end{equation}
	By the definition of operator $A_{g}$, we have $(A_{g}u,u)_{g} = \|u\|^{2}_{V_{g}}$. Using \eqref{eq:2.3} and the bilinearity of the operator $B_{g}$, it follows that $(B_{g}(u,u), u-v )_{g} = -(B(u,u-v),v)_{g}$ and $(B_{g}(v,v), u-v )_{g} = -(B(v,u-v),v)_{g}$. Hence,
	\begin{equation}
		\label{eq:1.15}
		(B_{g}(u,u)-B_{g}(v,v), u-v )_{g} = -(B_{g}(u-v,u-v),v)_{g}.
	\end{equation}
	By \eqref{eq:1.16} and H\"older inequality, we obtain
	\begin{equation}
		\label{eq:1.17}
		\begin{aligned}
			|(B_{g}(u,u)-B_{g}(v,v), u-v )_{g}| &\le \|u-v\|_{L^{4}(\mathcal{O})}\|u-v\|_{g}\|v\|_{L^{4}(\mathcal{O})}
			\\& \le  \|u-v\|^{1/2}_{g} \|u-v\|^{3/2}_{V_{g}} \|v\|_{L^{4}(\mathcal{O})}
			\\& \le \dfrac{\nu}{2} \| u-v\|^{2}_{V_{g}} + \dfrac{27}{32\nu^{3}}\| u-v\|^{2}_{g} \|v\|^{4}_{L^{4}(\mathcal{O})}.
		\end{aligned}
	\end{equation}
	For the last term on the right-hand side of equation \eqref{eq:1.14}, by \eqref{eq:1.2} and Poincar\'e inequality, we obtain
	\begin{equation}
		\label{eq:1.18}
		\begin{aligned}
			|(R_{g}u-R_{g}v, u-v)_{g}| &= |(P_{g} (\frac{1}{g}(\nabla g \cdot \nabla)(u-v), u-v )_{g}| 
			\\& \le \dfrac{\| \nabla g\|_{\infty}}{m_{0}\lambda^{1/2}_{1}}\|u-v\|^{2}_{V_{g}}  
			 \le \dfrac{1}{16}\|u-v\|^{2}_{V_{g}}.
		\end{aligned}
	\end{equation}
	By \eqref{eq:1.17}, \eqref{eq:1.18}, and the definition of the operator $F$,  we get
	\begin{equation}
		\label{eq:1.19}
		(F(u)-F(v),u-v)_{g} \le -\frac{9 \nu}{16} \|u-v\|^{2}_{V_{g}} + \dfrac{27r^{4}}{32\nu^{3}}\| u-v\|^{2}_{g},
	\end{equation}
	since $\| v \|_{L^{4}(\mathcal{O})} \le r$ . Combining this estimate with condition (\textbf{A3}) and $\varepsilon \le \frac{9 \nu}{16L}$ completes the proof.
\end{proof}
Let $\{e_j\}_{j\geq 1}\subset D(A_g)$ be a fixed orthonormal basis of $H_g$, where
\[
D(A_g)=\{u\in H_g:-P_g\Delta u\in H_g\}.
\]
For $n\in\mathbb{N}$, set
\[
H_n:=\operatorname{span}\{e_1,e_2,\ldots,e_n\},
\]
and let $P_n:H_g\to H_n$ denote the orthogonal projection onto $H_n$.
Define
\[
W_n:=P_nW,
\qquad
\sigma_n:=P_n\sigma.
\]
We define the Galerkin approximation $u_n^\varepsilon$ as the solution of the following finite-dimensional stochastic differential equation: for every $v\in H_n$,
	\begin{equation}
		\label{eq:1.20}
		\begin{aligned}
		d(u^{\varepsilon}_{n}(t),v)_{g} = ((f(t),v)_{g} + (F(u^{\varepsilon}_{n}(t)),v)_{g} )dt +\sqrt{\varepsilon}(\sigma(t,u^{\varepsilon}_{n}(t))dW_{n}(t), v)_{g}, 
		\end{aligned}
	\end{equation}
    with initial condition
     \begin{equation}
     		u_n^\varepsilon(0)=P_{n}u_{0}.
     \end{equation}
	Next, we need to provide some priori estimates regarding the solutions of the equation.
	\begin{prop}
	\label{prop:1.22} 
	Suppose (\textbf{A1})--(\textbf{A3}) hold, let $u^{\varepsilon}_{n}$ denote the unique strong solution of finite-dimensional system \eqref{eq:1.20} in $C([0, T]; H_{n})$. Then, for every $T > 0$, $N > 0$, $ 0 < \varepsilon \le \varepsilon_{0}$ and $\mathbf{E} \|u(0)\|^{2}_{g} < \infty$, the following estimates hold:
	\\
	\\
	1. If $f \in L^{2}(0,T; V^{*})$ and  $\varepsilon_{0} = \min\{\frac{3\nu}{4K}, \frac{7\nu}{8K(1+2c^{2}_{2})} \}$, then, for all $t \in [0,T]$,
	\begin{equation}
		\label{eq:4.38}
		\begin{aligned}
		\mathbf{E}(\| u^{\varepsilon}_{n}(t) \|^{2}_{g}) + \nu\int_{0}^{t}\mathbf{E}\| u^{\varepsilon}_{n}(s) \|^{2}_{V_{g}}ds \le C_{1}\left( \mathbf{E}\| u(0)\|^{2}_{g} + \mathbf{E}\int_{0}^{t}\| f(s)\|^{2}_{V^{*}}ds + 1 \right),
		\end{aligned}
	\end{equation}
	and 
	\begin{equation}
		\label{eq:4.39}
		\begin{aligned}
		\mathbf{E}\left[\underset{0 \le t \le T}{\sup}\left(\|u^{\varepsilon}_{n}(t)\|^{2}_{g} + \nu\int_{0}^{t} \| u^{\varepsilon}_{n}(s)\|^{2}_{V_{g}}ds \right)\right] \le C_{1} \left(\mathbf{E}\|u(0\|^{2}_{g} + \mathbf{E}\int_{0}^{T}\| f(s)\|^{2}_{V^{*}}ds +1 \right),
		\end{aligned}
	\end{equation}
	where $C_{1} = C_{1}(\varepsilon_{0}, \nu , T)$ is a constant independent of $\varepsilon$, $u_{0}$ and $N$. 
	\\
	\\
	2. If $f \in L^{2}(0,T; V^{*})$, let $\delta > 0$ and $\varepsilon_{0} = \frac{7\nu}{8K}$, then for all $t \in [0,T]$,
	\begin{equation}
		\label{eq:1.70}
		\begin{aligned}
			\mathbf{E}e^{-\delta t}\| u^{\varepsilon}_{n}(t)\|^{2}_{g}  + \nu\mathbf{E}\int_{0}^{t}e^{-\delta s}\|u^{\varepsilon}_{n}(s)\|^{2}_{V_{g}} ds \le C_{2}\left( \mathbf{E}\|u(0) \|^{2}_{g} + \mathbf{E}\int_{0}^{t}e^{-\delta s} \| f(s)\|^{2}_{V^{*}}ds + 1 \right),
		\end{aligned}
	\end{equation}
	where $C_{2} = C_{2}(\varepsilon_{0}, \nu , T) $ is a constant independent of $\varepsilon$, $u_{0}$ and $N$. 
	\\
	\\
	3. If $f \in L^{4}([0,T]; V^{*})$, $\mathbf{E}\|u(0)\|^{4}_{g} < \infty$, let $\delta > 0$ and $\varepsilon_{0} = \min\{\frac{\nu}{8K}, \frac{11\nu}{8K(3+c^{2}_{2})}\}$, then for all $t \in [0,T]$,
	\begin{equation}
		\label{eq:1.75}
		\begin{aligned}
			\mathbf{E}e^{-2\delta t}\| u^{\varepsilon}_{n}(t)\|^{4}_{g}  + 3\nu \mathbf{E} \int_{0}^{t}e^{-2\delta s} \| u^{\varepsilon}_{n}(s) \|^{2}_{V_{g}}\| u^{\varepsilon}_{n}(s) \|^{2}_{g} 	ds 
			\le C_{3}\left( \mathbf{E}\| u(0) \|^{4}_{g} + \mathbf{E} \int_{0}^{t} e^{-2\delta s}\| f(s) \|^{4}_{V^{*}}  ds + 1 \right),
		\end{aligned}
	\end{equation}
	and
	\begin{equation}	
		\label{eq:1.79}
		\begin{aligned}
			&\textbf{E}\left[ \underset{0 \le t \le T}{\sup}\left(  e^{-2\delta t} \|u^{\varepsilon}_{n}(s)\|^{4}_{g}  +  \nu \textbf{E}\int_{0}^{t}e^{-2\delta t}\| u^{\varepsilon}_{n}(t)\|^{2}_{V_{g}}\|u^{\varepsilon}_{n}(t)\|^{2}_{g} dt \right) \right] 
			\\& \le C_{4}\left( \textbf{E}\|u^{\varepsilon}_{n}(0)\|^{4}_{g} +  \textbf{E}\int_{0}^{T} e^{-2\delta t}\|f(s)\|^{4}_{V^{*}} ds + 1 \right), 
		\end{aligned}
	\end{equation}
	where $C_{3}$ and $C_{4}$ are positive constants independent of $\varepsilon$, $u_{0}$ and $N$.
	\end{prop}
	\begin{proof} Proof of \eqref{eq:4.38} and \eqref{eq:4.39}. By \eqref{eq:2.6} and It\^o formula, we can get 
	\begin{equation}
		\label{eq:4.50}
		\begin{aligned}
			d\| u^{\varepsilon}_{n}(t) \|^{2}_{g} &+ 2(\nu A_{g}u^{\varepsilon}_{n}(s) + B_{g}(u^{\varepsilon}_{n}(s), u^{\varepsilon}_{n}(s)) + \nu R_{g}u^{\varepsilon}_{n}(s) , u^{\varepsilon}_{n}(s) )_{g}dt = 2(f(s), u^{\varepsilon}_{n}(t))_{g}dt 
			\\& + \varepsilon tr(\sigma(u^{\varepsilon}_{n}(t))Q\sigma^{*}(u^{\varepsilon}_{n}(t)))dt + 2\sqrt{\varepsilon}(\sigma(u^{\varepsilon}_{n}(t))dW_{n}(t), u^{\varepsilon}_{n}(t) )_{g}.
		\end{aligned}
	\end{equation}
	By the definitions of the $A_{g}$ and $B_{g}$, we have
	\begin{equation}
		\label{eq:4.51}
		\begin{aligned}
			& d\| u^{\varepsilon}_{n}(t) \|^{2}_{g} + 2\nu\|u^{\varepsilon}_{n}(t)\|^{2}_{V_{g}} dt +  2(\nu R_{g}u^{\varepsilon}_{n}(t) , u^{\varepsilon}_{n}(t) )_{g}dt \\& = 2(f(s), u^{\varepsilon}_{n}(t))_{g}dt 
		+ \varepsilon tr(\sigma(u^{\varepsilon}_{n}(t))Q\sigma^{*}(u^{\varepsilon}_{n}(t)))dt + 2\sqrt{\varepsilon}(\sigma(u^{\varepsilon}_{n}(t))dW_{n}(t), u^{\varepsilon}_{n}(t) )_{g}.
		\end{aligned}
	\end{equation}
	For the last term on the left-hand side of \eqref{eq:4.51}, by Lemma \ref{lem:3.65} and \eqref{eq:1.2}, we have  
	\begin{equation}
		\label{eq:4.52}
		\begin{aligned}
		2\nu|( R_{g}u^{\varepsilon}_{n}(s) , u^{\varepsilon}_{n}(s) )_{g}dt	| \le 2\nu  \frac{\|\nabla g\|_{\infty}}{m_{0}\lambda^{1/2}_{1}} \|u^{\varepsilon}_{n}(t)\|^{2}_{V_{g}}dt \le \frac{\nu }{8} \|u^{\varepsilon}_{n}(t)\|^{2}_{V_{g}}dt. 
		\end{aligned}
	\end{equation}
		For the first term on the right-hand side of \eqref{eq:4.51}, by Young's inequality, we have
	\begin{equation}
		\label{eq:4.53}
		\begin{aligned}
	2|(f(s), u^{\varepsilon}_{n}(t))_{g}dt| \le  \frac{8}{\nu}\|f(t)\|^{2}_{V^{*}}dt + \frac{ \nu}{8} \|u^{\varepsilon}_{n}(t)\|^{2}_{V_{g}}dt. 
		\end{aligned}
	\end{equation}
		Define a sequence of stopping times as follows
	\[
	\tau_{N} = \inf \left\{t: \|u^{\varepsilon}_{n}(t)\|^{2}_{g} + \int_{0}^{t}\| u^{\varepsilon}_{n}(s)\|^{2}_{V_{g}} \ge N \right\}. 
	\]
	Combining \eqref{eq:4.52} - \eqref{eq:4.53} with  condition (\textbf{A2}), we get that for any $t \in [0,T]$,
	\begin{equation}
		\label{eq:4.54}
		\begin{aligned}
			\| u^{\varepsilon}_{n}(t \wedge \tau_{N}) \|^{2}_{g} + \nu \int_{0}^{t \wedge \tau_{N}}\|u^{\varepsilon}_{n}(s)\|^{2}_{V_{g}} ds 
			& \le \|u^{\varepsilon}_{n}(0)\|^{2}_{g} + \frac{8}{\nu}  \int_{0}^{t \wedge \tau_{N}} \|f(s)\|^{2}_{V^{*}}ds 
			 \\& + \varepsilon K \int_{0}^{t \wedge \tau_{N}} \left(1+ \|u^{\varepsilon}_{n}(s)\|^{2}_{V_{g}} \right)ds.
		\end{aligned}
	\end{equation}
	Taking the expectation for \eqref{eq:4.54}, and $\varepsilon < \frac{3\nu}{4K} $, we have
	\begin{equation}
		\label{eq:4.56}
		\begin{aligned}
			\mathbf{E}\| u^{\varepsilon}_{n}(t \wedge \tau_{N}) \|^{2}_{g} + \nu\int_{0}^{t \wedge \tau_{N}}\mathbf{E}\| u^{\varepsilon}_{n}(s) \|^{2}_{V_{g}}ds 
			\le C_{1}\left(\mathbf{E}\| u(0)\|^{2}_{g} + \mathbf{E}\int_{0}^{t \wedge \tau_{N}}\| f(s)\|^{2}_{V^{*}}ds + 1 \right), 
		\end{aligned}
	\end{equation}
	where $C_{1} > 0$  is independent of $\varepsilon$, $u_{0}$ and $\|f(t)\|_{V^{*}}$. 
	\par 
	Taking the supremum and expectation in \eqref{eq:4.51} and using \eqref{eq:4.54}, we have that for $t\in [0, T\wedge \tau_{N}]$, 
	\begin{equation}
		\label{eq:1.63}
		\begin{aligned}
			&\textbf{E}\underset{0 \le t \le T \wedge \tau_{N}}{\sup} \left(\|u^{\varepsilon}_{n}(t) \|^{2}_{g}  + \nu \int_{0}^{t} \|u^{\varepsilon}_{n}(s)\|^{2}_{V_{g}}ds \right) 
			\\& \le \textbf{E}\|u(0)\|^{2}_{g} + \varepsilon K \textbf{E}\int_{0}^{t \wedge \tau_{N}} \left(1+ \|u^{\varepsilon}_{n}(s)\|^{2}_{V_{g}} \right)ds 
			\\& + \frac{8}{\nu}\textbf{E}\int_{0}^{T \wedge \tau_{N}} \| f(s)\|^{2}_{V^{*}}ds + 2\sqrt{\varepsilon} \,\textbf{E} \underset{0\le t \le T \wedge \tau_{N}}{\sup} \left|\int_{ 0}^{t } (\sigma_{n}(s, u^{\varepsilon}_{n}(s) )dW_{n}(s), u^{\varepsilon}_{n}(s))_{g} \right|.
		\end{aligned}
	\end{equation}
	For the last term on the right-hand side of \eqref{eq:1.63}, by using condition (\textbf{A2}) and Burkholder-Davis-Gundy inequality, we have 
	\begin{equation}
		\label{eq:1.64}
		\begin{aligned}
			&2\sqrt{\varepsilon} \,\textbf{E}\left[ \underset{0\le t \le T \wedge \tau_{N}}{\sup} \left|\int_{ 0}^{t } (\sigma_{n}(s, u^{\varepsilon}_{n}(s) )dW_{n}(s), u^{\varepsilon}_{n}(s))_{g} \right|  \right]
			\\& \le 2\sqrt{\varepsilon} c_{2} \textbf{E} \left[  \left(\int_{0}^{T \wedge \tau_{N}} \|\sigma_{n}(s, u^{\varepsilon}_{n}(s)\|^{2}_{L_{Q}}\|u^{\varepsilon}_{n}(s)\|^{2}_{g}ds \right)^{1/2} \right]
			\\& \le  \frac{1}{2}   \textbf{E}\left[ \underset{0 \le t \le T \wedge \tau_{N} }{\sup} \|u^{\varepsilon}_{n}(t) \|^{2}_{g} \right]  + 2 \varepsilon c^{2}_{2}  \textbf{E} \left[ \int_{0}^{T \wedge \tau_{N}} \|\sigma_{n}(s, u^{\varepsilon}_{n}(s)\|^{2}_{L_{Q}} ds \right]
			\\& \le \frac{1}{2}   \textbf{E}\left[ \underset{0 \le t \le T \wedge \tau_{N} }{\sup} \|u^{\varepsilon}_{n}(t) \|^{2}_{g} \right]  + 2 \varepsilon c^{2}_{2} K \textbf{E} \left[  \int_{0}^{T \wedge \tau_{N}} \| u^{\varepsilon}_{n}(s)\|^{2}_{V_{g}} ds \right] +2 \varepsilon c^{2}_{2} KT.
		\end{aligned}
	\end{equation}
	Substituting \eqref{eq:1.64} into \eqref{eq:1.63}, we obtain that for any $t \in [0,T \wedge \tau_{N}]$,
	\begin{equation}
		\label{eq:1.65}
		\begin{aligned}
			&\textbf{E}\left[ \underset{0\le t \le T \wedge \tau_{N}}{\sup} \left( \frac{1}{2} \|u^{\varepsilon}_{n}(t) \|^{2}_{g}  + (\frac{15\nu }{8} -  \varepsilon K - 2 \varepsilon c^{2}_{2} K) \int_{0}^{t} \|u^{\varepsilon}_{n}(s)\|^{2}_{V_{g}}ds \right) \right] 
			\\&\le \textbf{E}\|u(0)\|^{2}_{g} + \frac{8}{\nu} \textbf{E} \int_{0}^{T \wedge \tau_{N}} \| f(s)\|^{2}_{V^{*}} ds 
			+ \varepsilon KT +2 \varepsilon c^{2}_{2} KT.
		\end{aligned}
	\end{equation}
	By \eqref{eq:4.56}, we have 
	\begin{equation}
		\label{eq:1.66}
		\begin{aligned}
			\nu\int_{0}^{T \wedge \tau_{N}}\mathbf{E}\| u^{\varepsilon}_{n}(s) \|^{2}_{V_{g}}ds \le C_{1}\left( \mathbf{E}\| u(0)\|^{2}_{g} + \mathbf{E}\int_{0}^{T \wedge \tau_{N}}\| f(s)\|^{2}_{V^{*}}ds + 1 \right).
		\end{aligned}
	\end{equation}
	Combining \eqref{eq:1.65}-\eqref{eq:1.66}, we obtain
	\begin{equation}
		\label{eq:1.67}
		\begin{aligned}
			\textbf{E}\underset{0\le t \le T \wedge \tau_{N}}{\sup}\|u^{\varepsilon}_{n}(t) \|^{2}_{g}\le C_{1}\left(\mathbf{E}\| u(0)\|^{2}_{g} + \textbf{E} \int_{0}^{T}\|f(s)\|^{2}_{V^{*}}ds+ 1 \right), 
		\end{aligned}
	\end{equation}
	where $ \varepsilon < \frac{7\nu }{8K(1+2c^{2}_{2})}$. It can be inferred from \eqref{eq:1.63}-\eqref{eq:1.66} that
	\begin{equation}
		\label{eq:1.68}
		\begin{aligned}
			\mathbf{E}\left[\underset{0 \le t \le T}{\sup}\left(\|u^{\varepsilon}_{n}(t)\|^{2}_{g} + \nu\int_{0}^{t} \| u^{\varepsilon}_{n}(s)\|^{2}_{V_{g}}ds \right)\right]  \le C_{1}\left(\mathbf{E}\| u(0)\|^{2}_{g} + \int_{0}^{T}\textbf{E}\|f(s)\|^{2}_{V^{*}}ds +1 \right), 
		\end{aligned}
	\end{equation}
	where $C_{1} >0 $ is independent of $\varepsilon$, $u_{0}$ and $N$. Since $\tau_N\to T$ almost surely as $N\to\infty$, Fatou's lemma applied to \eqref{eq:1.68} yields \eqref{eq:4.39}. This completes the proof.
	\par 
	Proof of \eqref{eq:1.70}. Let $\delta > 0$ and the function $h(t, u^{\varepsilon}_{n}) = e^{- \delta t}\| u^{\varepsilon}_{n}\|^{2}_{g}$. By \eqref{eq:1.20} and  It\^o formula, we obtain 
	\begin{equation}
		\label{eq:1.71}
		\begin{aligned}
			& de^{- \delta t}\|u^{\varepsilon}_{n}(t)\|^{2}_{g} + \delta e^{- \delta t} \|u^{\varepsilon}_{n}(t)\|^{2}_{g}dt + 2\nu e^{- \delta t}\|u^{\varepsilon}_{n}(t)\|^{2}_{V_{g}}dt + 2\nu (R_{g}u^{\varepsilon}_{n}(t), e^{- \delta t}u^{\varepsilon}_{n}(t))_{g}dt
			\\& = (f(t), 2e^{- \delta t}u^{\varepsilon}_{n}(t))_{g}dt   
			+ e^{- \delta t}\varepsilon tr(\sigma_{n}(t, u^{\varepsilon}_{n}(t)) Q\sigma^{*}_{n}(t, u^{\varepsilon}_{n}(t)))dt 
			\\& + 2\sqrt{\varepsilon}(\sigma_{n}(t, u^{\varepsilon}_{n}(t))dW_{n}(t), e^{- \delta t}u^{\varepsilon}_{n}(t))_{g}.
		\end{aligned}
	\end{equation}
	For the last term on the left-hand side of \eqref{eq:1.71}, by Lemma \ref{lem:3.65} and \eqref{eq:1.2}, we have  
	\begin{equation}
		\label{eq:4.55}
		\begin{aligned}
			2\nu|( R_{g}u^{\varepsilon}_{n}(s) , u^{\varepsilon}_{n}(s) )_{g}	| \le 2\nu e^{- \delta t}  \frac{\|\nabla g\|_{\infty}}{m_{0}\lambda^{1/2}_{1}} \|u^{\varepsilon}_{n}(t)\|^{2}_{V_{g}} \le \frac{\nu }{8}e^{- \delta t} \|u^{\varepsilon}_{n}(t)\|^{2}_{V_{g}}. 
		\end{aligned}
	\end{equation}
		For the first term on the right-hand side of \eqref{eq:4.51}, by Young's inequality, we get
	\begin{equation}
		\label{eq:4.57}
		\begin{aligned}
			2e^{- \delta t}|(f(s), u^{\varepsilon}_{n}(t))_{g}| \le  \frac{1}{\delta}e^{- \delta t}\|f(s)\|^{2}_{V^{*}} + \delta e^{- \delta t} \|u^{\varepsilon}_{n}(s)\|^{2}_{g} .   
		\end{aligned}
	\end{equation}
	Define the sequence of stopping times as follows
	\[
	\tau_{N} = \inf\{t : \max(\|u^{\varepsilon}_{n}(t)\|^{2}_{g}, \, \int_{0}^{t}\|u^{\varepsilon}_{n}(s)\|^{2}_{V_{g}}ds) \ge  N\}.
	\]
	By \eqref{eq:4.55}- \eqref{eq:4.57}, taking the expectation, we obtain that for any $s \in [0,t \wedge \tau_{N}]$,
	\begin{equation}
		\label{eq:1.72}
		\begin{aligned}
			&\textbf{E} e^{- \delta (t\wedge \tau_{N})} \|u^{\varepsilon}_{n}(t \wedge \tau_{N})\|^{2}_{g} + 2\nu \textbf{E}\int_{0}^{t\wedge \tau_{N}} e^{-\delta s} \|u^{\varepsilon}_{n}(s)\|^{2}_{V_{g}}ds 
			\\& \le \textbf{E}\|u^{\varepsilon}_{n}(0)\|^{2}_{g} 
			 + \dfrac{\nu}{8}  \textbf{E}\int_{0}^{t\wedge \tau_{N}} e^{-\delta s}\|u^{\varepsilon}_{n}(s)\|^{2}_{V_{g}}ds + \dfrac{1}{\delta}\textbf{E}\int_{0}^{t\wedge \tau_{N}} e^{-\delta s}\|f(s)\|^{2}_{V^{*}}ds  
			 \\&+ \varepsilon \textbf{E}\int_{0}^{t \wedge \tau_{N}} e^{-\delta s} \|\sigma_{n}(s,u^{\varepsilon}_{n}(s))\|^{2}_{L_{Q}}ds . 
		\end{aligned}
	\end{equation}
	Combining this estimate with (\textbf{A2}) and assuming $\varepsilon < \frac{3\nu}{4K}$, we have
	\begin{equation}
		\label{eq:1.73}
		\begin{aligned}
			\textbf{E} e^{- \delta (t\wedge \tau_{N})} \|u^{\varepsilon}_{n}(t \wedge \tau_{N})\|^{2}_{g} + \nu  \textbf{E}\int_{0}^{t\wedge \tau_{N}} e^{-\delta s} \|u^{\varepsilon}_{n}(s)\|^{2}_{V_{g}}ds  
			  \le C_{2}\left( \mathbf{E}\|u(0) \|^{2}_{g} + \mathbf{E}\int_{0}^{t}e^{-\delta s} \| f(s)\|^{2}_{V^{*}}ds + 1 \right),
		\end{aligned}
	\end{equation}
	where $C_{2} >0 $ is independent of $\varepsilon$, $u_{0}$ and $N$. Note that we can obtain $\tau_{N} \to \infty$, a.s., as $N\to \infty$. Thus, the proof of \eqref{eq:1.70} is completed.
	\par 
	Proof of \eqref{eq:1.75} and \eqref{eq:1.79}. By the function $h(t, u^{\varepsilon}_{n}) = e^{- 2\delta t}\| u^{\varepsilon}_{n}\|^{4}_{g}$ and It\^o formula, taking the expectation, we obtain that for all $t \in [0,T]$,
		\begin{equation}
			\label{eq:1.76}
			\begin{aligned}
				&\textbf{E}e^{-2\delta t}\|u^{\varepsilon}_{n}(t)\|^{4}_{g} + 4\nu \textbf{E} \int_{0}^{t}e^{- 2\delta s}\|u^{\varepsilon}_{n}(s)\|^{2}_{V_{g}} \|u^{\varepsilon}_{n}(s)\|^{2}_{g} ds + 2\delta \textbf{E}\int_{0}^{t}e^{-2\delta s} \|u^{\varepsilon}_{n}(s)\|^{4}_{g}ds 
				\\& \le \textbf{E}\|u^{\varepsilon}_{n}(0)\|^{4}_{g} -
				4\nu \textbf{E}\int_{0}^{t}e^{- 2\delta s}\|u^{\varepsilon}_{n}(s)\|^{2}_{g}(u^{\varepsilon}_{n}(s), R_{g}u^{\varepsilon}_{n}(s))_{g} ds + 
				4\textbf{E}\int_{0}^{t}e^{-2\delta s} \|u^{\varepsilon}_{n}(s) \|^{2}_{g}(f(s), u^{\varepsilon}_{n}(s))_{g}ds
				\\& +  2\varepsilon \textbf{E}\int_{0}^{t} e^{-2\delta s}\|u^{\varepsilon}_{n}(s)\|^{2}_{g}\|\sigma_{n}(s, u^{\varepsilon}_{n}(s))\|^{2}_{L_{Q}}ds +  \textbf{E}\int_{0}^{t} e^{-2\delta s}\|\langle 2u^{\varepsilon}_{n}(s) , \sqrt{\varepsilon}\sigma(s, u^{\varepsilon}_{n}(s)) \rangle \|^{2}  ds.
			\end{aligned}
		\end{equation}
		For the second term on the right-hand side of \eqref{eq:1.76}, by Lemma \ref{lem:3.65} and  H\"older inequality, we have
		\begin{equation}
			\label{eq:4.37}
			\begin{aligned}
				4\nu \left| \textbf{E}\int_{0}^{t}e^{- 2\delta s}\|u^{\varepsilon}_{n}(s)\|^{2}_{g}(u^{\varepsilon}_{n}(s), R_{g}u^{\varepsilon}_{n}(s))_{g} ds \right| &\le 4\nu \frac{\|\nabla g\|_{\infty}}{m_{0}\lambda^{1/2}_{1}}\textbf{E}\int_{0}^{t}e^{- 2\delta s}\|u^{\varepsilon}_{n}(s)\|^{2}_{g} \|u^{\varepsilon}_{n}(s)\|^{2}_{V_{g}} ds
				\\& \le \frac{\nu }{4} \textbf{E}\int_{0}^{t}e^{- 2\delta s}\|u^{\varepsilon}_{n}(s)\|^{2}_{g} \|u^{\varepsilon}_{n}(s)\|^{2}_{V_{g}} ds. 
			\end{aligned}
		\end{equation}
		For the third term on the right-hand side of \eqref{eq:1.76}, by H\"older inequality and Young's inequality, we get
		\begin{equation}
			\label{eq:2.30}
			\begin{aligned}
				4\textbf{E}\int_{0}^{t}e^{-2\delta s} \|u^{\varepsilon}_{n}(s) \|^{2}_{g}(f(s), u^{\varepsilon}_{n}(s))_{g}ds  \le 	2 \delta\textbf{E}\int_{0}^{t}e^{-2\delta s} \|u^{\varepsilon}_{n}(s) \|^{4}_{g} ds +   \frac{27}{8 \delta^{3}}              \textbf{E} \int_{0}^{t}e^{-2\delta s}\|f(s)\|^{4}_{V^{*}}ds. 
			\end{aligned}
		\end{equation}
		For the fourth and last term on the right-hand side of \eqref{eq:1.76}, by condition (\textbf{A2}) we have
		\begin{equation}
			\label{eq:2.31}
			\begin{aligned}
				&2\varepsilon \textbf{E}\int_{0}^{t} e^{-2\delta s}\|u^{\varepsilon}_{n}(s)\|^{2}_{g}\|\sigma_{n}(s, u^{\varepsilon}_{n}(s))\|^{2}_{L_{Q}}ds +  \textbf{E}\int_{0}^{t} e^{-2\delta s} \|\langle 2u^{\varepsilon}_{n}(s) , \sqrt{\varepsilon}\sigma(s, u^{\varepsilon}_{n}(s)) \rangle \|^{2}ds
				\\& \le  2\varepsilon K\textbf{E}\int_{0}^{t} e^{-2\delta s}\|u^{\varepsilon}_{n}(s)\|^{2}_{g}\left( 1 + \| u^{\varepsilon}_{n}(s)\|^{2}_{V_{g}} \right)ds  + 4\varepsilon \textbf{E}\int_{0}^{t} e^{-2\delta s}\|\sigma_{n}(s, u^{\varepsilon}_{n}(s))\|^{2}_{L_{Q}}\|u^{\varepsilon}_{n}(s)\|^{2}_{g}ds
				\\& \le 6\varepsilon K\textbf{E}\int_{0}^{t} e^{-2\delta s}\|u^{\varepsilon}_{n}(s)\|^{2}_{V_{g}}\|u^{\varepsilon}_{n}(s)\|^{2}_{g}ds + 6\varepsilon K\textbf{E}\int_{0}^{t} e^{-2\delta s}\|u^{\varepsilon}_{n}(s)\|^{2}_{g}ds.
			\end{aligned}
		\end{equation}
		Combining \eqref{eq:2.30}-\eqref{eq:2.31}, we obtain
		\begin{equation}
			\label{eq:1.77}
			\begin{aligned}
				&\textbf{E}e^{-2\delta t}\|u^{\varepsilon}_{n}(t)\|^{4}_{g} +(4\nu- \frac{\nu}{4} - 6\varepsilon K)\textbf{E} \int_{0}^{t}e^{- \delta s}\|u^{\varepsilon}_{n}(s)\|^{2}_{V_{g}} \|u^{\varepsilon}_{n}(s)\|^{2}_{g} ds 
				\\& \le \textbf{E}\|u^{\varepsilon}_{n}(0)\|^{4}_{g} 
				+ \dfrac{27}{8\delta^{3}}\textbf{E}\int_{0}^{t}e^{-2\delta s} \|f(s)\|^{4}_{V^{*}}ds + 6\varepsilon K\textbf{E}\int_{0}^{t} e^{-2\delta s}\|u^{\varepsilon}_{n}(s)\|^{2}_{g}ds. 
			\end{aligned}
		\end{equation}
		By \eqref{eq:4.39}, there exists a constant $C_{1} >0$ such that $\textbf{E} \underset{t\in [0,T]}{\sup} \| u^{\varepsilon}_{n}(t)\|^{2}_{g} < C_{1}$. Taking $\varepsilon_{0} < \frac{\nu }{8K}$, we have 
		\begin{equation}
			\label{eq:1.78}
			\begin{aligned}
				& \textbf{E}e^{-2\delta t}\|u^{\varepsilon}_{n}(t)\|^{4}_{g} + 3\nu\textbf{E} \int_{0}^{t}e^{- 2\delta s}\|u^{\varepsilon}_{n}(t)\|^{2}_{V_{g}} \|u^{\varepsilon}_{n}(s)\|^{2}_{g} ds 
				\\&  \le \textbf{E}\|u^{\varepsilon}_{n}(0)\|^{4}_{g} 
				+ \dfrac{27}{8\delta^{3}}\textbf{E}\int_{0}^{t}e^{-2\delta s}\|f(s)\|^{4}_{V^{*}}ds +  \dfrac{3\varepsilon K C_{1}}{\delta}
				\\& \le C_{3}\left(\textbf{E}\|u^{\varepsilon}_{n}(0)\|^{4}_{g} 
				+ \textbf{E}\int_{0}^{t}e^{-2\delta s}\|f(s)\|^{4}_{V^{*}}ds +1  \right),
			\end{aligned}
		\end{equation}
		where $C_{3}$ is a positive constant independent of  $\varepsilon$, $u_{0}$ and $N$. The proof of \eqref{eq:1.75} has been completed.
	\par 
	Finally, we prove \eqref{eq:1.79}. Taking the supremum and expectation in \eqref{eq:1.76}, we obtain that for all  $ t \in [0,T \wedge \tau_{N} ]$, 
	\begin{equation}
		\label{eq:1.80}
		\begin{aligned}
			&\textbf{E} \underset{0 \le t \le T \wedge \tau_{N} }{\sup}e^{-2\delta t}\|u^{\varepsilon}_{n}(t)\|^{4}_{g} +(4\nu - \frac{\nu}{4})\textbf{E}\underset{0 \le t \le T \wedge \tau_{N} }{\sup}\int_{0}^{t} e^{- 2\delta s}\|u^{\varepsilon}_{n}(s)\|^{2}_{V_{g}} \|u^{\varepsilon}_{n}(s)\|^{2}_{g} ds 
			\\& \le \textbf{E}\|u^{\varepsilon}_{n}(0)\|^{4}_{g} + \dfrac{27}{8\delta^{3} } \textbf{E} \underset{0\le t\le T\wedge \tau_{N} }{\sup}\int_{0}^{t}e^{-2\delta s}\|f(s)\|^{4}_{V^{*}}ds
			\\& +  2\varepsilon \textbf{E}\underset{0 \le t \le T\wedge \tau_{N}  }{\sup}\int_{0}^{t} e^{-2\delta s}\|u^{\varepsilon}_{n}(s)\|^{2}_{g}\|\sigma_{n}(s, u^{\varepsilon}_{n}(s))\|^{2}_{L_{Q}}ds 
			\\&+ 4 \varepsilon \textbf{E}\underset{0 \le t \le T\wedge \tau_{N}  }{\sup}\int_{0}^{t} e^{-2\delta s}\|\sigma_{n}(s, u^{\varepsilon}_{n}(s))\|^{2}_{L_{Q}}\|u^{\varepsilon}_{n}(s)\|^{2}_{g}ds  
			\\& +  4 \sqrt{\varepsilon}\textbf{E}\underset{0 \le t \le T\wedge \tau_{N}  }{\sup}\int_{0}^{s}e^{-2\delta s}\|u^{\varepsilon}_{n}(s)\|^{2}_{g}(\sigma_{n}(s,u^{\varepsilon}_{n}(s))dW_{n}(s),  u^{\varepsilon}_{n}(s))_{g}.
		\end{aligned}
	\end{equation}
	For the last term on the right-hand side of \eqref{eq:1.80}, using condition (\textbf{A2}) and Burkholder-Davis-Gundy inequality, we have 
	\begin{equation}
		\label{eq:1.81}
		\begin{aligned}
			&4 \sqrt{\varepsilon}\textbf{E}\underset{0 \le t \le T\wedge \tau_{N} }{\sup}\int_{0}^{s}e^{-2\delta s}\|u^{\varepsilon}_{n}(s)\|^{2}_{g}(\sigma_{n}(s,u^{\varepsilon}_{n}(s))dW_{n}(s),  u^{\varepsilon}_{n}(s))_{g} 
			\\& \le 4 \sqrt{\varepsilon} c_{2} \textbf{E}\left(\int_{0}^{T\wedge \tau_{N}} e^{-2\delta s} \|u^{\varepsilon}_{n}(s)\|^{2}_{g} \|\sigma_{n}(s,u^{\varepsilon}_{n}(s))\|^{2}_{L_{Q}}\|u^{\varepsilon}_{n}(s)\|^{2}_{g}ds      \right)^{1/2}
			\\& \le \frac{1}{2}\textbf{E} \underset{0\le t \le T\wedge \tau_{N}}{\sup}e^{-2\delta t} \|u^{\varepsilon}_{n}(t)\|^{4}_{g}  + 2\varepsilon K c^{2}_{2}\textbf{E}\int_{0}^{T\wedge \tau_{N}}e^{-2\delta s} \|u^{\varepsilon}_{n}(s)\|^{2}_{g}\|\sigma_{n}(s,u^{\varepsilon}_{n}(s))\|^{2}_{L_{Q}} ds 
	        \\& \le \frac{1}{2}\textbf{E} \underset{0\le t \le T\wedge \tau_{N}}{\sup}e^{-2\delta t} \|u^{\varepsilon}_{n}(t)\|^{4}_{g} + 2\varepsilon K c^{2}_{2}\textbf{E}\int_{0}^{T\wedge \tau_{N}}e^{-2\delta s} \|u^{\varepsilon}_{n}(s)\|^{2}_{g}\|u^{\varepsilon}_{n}(s)\|^{2}_{V_{g}} ds 
	        \\& +  2\varepsilon K c^{2}_{2}\textbf{E}\int_{0}^{T\wedge \tau_{N}}e^{-2\delta s} \|u^{\varepsilon}_{n}(s)\|^{2}_{g}ds.
		\end{aligned}
	\end{equation}
	Combining \eqref{eq:2.31} - \eqref{eq:1.77} with \eqref{eq:1.81}, we obtain
	\begin{equation}
		\label{eq:1.82}
		\begin{aligned}
			&\frac{1}{2}\textbf{E} \underset{0 \le t \le T\wedge \tau_{N} }{\sup}e^{-2\delta t}\|u^{\varepsilon}_{n}(t)\|^{4}_{g} + \left( 4\nu - \frac{\nu}{4} - 6\varepsilon K - 2\varepsilon K c^{2}_{2} \right)\textbf{E} \int_{0}^{T\wedge \tau_{N}} e^{- 2\delta s} \|u^{\varepsilon}_{n}(t)\|^{2}_{V_{g}} \|u^{\varepsilon}_{n}(s)\|^{2}_{g} ds 
			\\& \le \textbf{E}\|u^{\varepsilon}_{n}(0)\|^{4}_{g} + \dfrac{27}{8\delta^{3}} \textbf{E} \underset{0\le t\le T\wedge \tau_{N}}{\sup}\int_{0}^{t}e^{-2\delta s}\|f(s)\|^{4}_{V^{*}}ds +   (6\varepsilon K + 2\varepsilon K c^{2}_{2})\textbf{E}\underset{0 \le t \le T\wedge \tau_{N} }{\sup}\int_{0}^{t} e^{-2\delta s}\|u^{\varepsilon}_{n}(s)\|^{2}_{g}ds.
		\end{aligned}
	\end{equation}
	From \eqref{eq:1.70}, there exists a constant $C_{1} >0$ such that $\textbf{E} \underset{t\in [0,T]}{\sup} \| u^{\varepsilon}_{n}(t)\|^{2}_{g} < C_{1}$. If $\varepsilon < \frac{11\nu}{K(24+8c^{2}_{2})}$, we obtain that
	\begin{equation}
		\label{eq:1.83}
		\begin{aligned}
			&\textbf{E}\underset{0 \le t \le T}{\sup} e^{-2\delta t}\|u^{\varepsilon}_{n}(t)\|^{4}_{g}  + \nu \textbf{E}\int_{0}^{T}e^{-2\delta t}\| u^{\varepsilon}_{n}(s)\|^{2}_{V_{g}}\|u^{\varepsilon}_{n}(s)\|^{2}_{g} ds 
			\\& \le \textbf{E}\|u^{\varepsilon}_{n}(0)\|^{4}_{g} +  \dfrac{27}{4\delta^{3}}\textbf{E} \int_{0}^{T} e^{-2\delta t}\|f(t)\|^{4}_{V^{*}} dt + \frac{(6 + 2c^{2}_{2} ) K C_{1} \varepsilon }{\delta} 
			\\& \le C_{4}\left(\textbf{E}\|u^{\varepsilon}_{n}(0)\|^{4}_{g} 
			+ \textbf{E}\int_{0}^{t}e^{-2\delta s}\|f(s)\|^{4}_{V^{*}}ds +1  \right),
		\end{aligned}
	\end{equation}
	where $C_{4}$ is a positive constant independent of $\varepsilon$, $u_{0}$ and $N$. The proof of \eqref{eq:1.79} has been completed.
	\end{proof}

	\begin{prop}\cite{ref S. S. Sritharan}
		\label{prop:1.7}
		Let $p > 1$ and $q$ be the conjugate of $p$. Let $V \hookrightarrow H \hookrightarrow V^{*}$ be a Gelfand triple, where $V$ is a separable reflexive Banach space, $V^{*}$ is its strong dual and $H$ a separable Hilbert space with dense injections. Let $M$ be an $H$-valued $L^{2}$-integrable continuous martingale whose paths belong P-a.s. to $L^{p}(0, T ; V )$ and with $M(0) = 0$.\par
		Assume that $X$ and $U$ are two adapted processes whose trajectories belong almost surely to $L^{ p}(0, T ; V )$ and $L^{q} (0, T ; V^{*})$, respectively, such that
		\begin{equation}
			\label{eq:1.21}
			X(T) = X_{0} + \int_{0}^{t}U(s)ds + M(t),
		\end{equation}
		holds almost surely for all $t \in [0, T ]$ with $X_{0} \in H$. Then $X$ has paths in $C([0, T ]; H )$ almost surely, and 
		\begin{equation}
			\label{eq:1.22}
			|X(t)|^{2} =  |X_{0}|^{2} + 2\int_{0}^{t} (X(s), U(s))ds +2\int_{0}^{t}(X(s), dM(s)) + tr[M]_{t}
		\end{equation}
		holds a.s. and for all $t\in [0, T]$.
	\end{prop}
	The monotonicity method for proving the existence of strong solutions for SPDEs was first proposed by Pardoux \cite{ref Pardoux. E}. Menaldi and Sritharan \cite{ref Menaldi. J.L} first proved the existence of strong solutions to the Navier--Stokes equations with additive noise using the local monotone method. We next establish the existence of probabilistically strong solutions to \eqref{eq:2.6} and prove pathwise uniqueness. 
	\begin{thm}
		\label{thm:1.8}
		Suppose  $f \in L^{4}(0, T; V^{*}_{g})$, (\textbf{A1})--(\textbf{A3}) hold and $\mathbf{E}\,\|u_{0}\|^{4}_{g } < \infty$, then  there exists $\varepsilon_{0} = \min\{\frac{\nu}{8K}, \frac{7\nu}{8K(1+2c^{2}_{2})},\frac{11\nu }{8K(3+ c^{2}_{2})}\}$  such that for every $T> 0$, $ 0< \varepsilon < \varepsilon_{0}$,  there exists a probabilistically strong solution of the following stochastic $g$-Navier--Stokes equation:
		\begin{equation}
			\label{eq:1.23}
		du^{\varepsilon}+(\nu A_{g}u^{\varepsilon} + B_{g}(u^{\varepsilon}, u^{\varepsilon}) + \nu R_{g}u^{\varepsilon})dt = f(t)dt + \sqrt{\varepsilon}\sigma(t,u^{\varepsilon})dW(t)
		\end{equation}
		in $L^{2}(\Omega : C(0, T ; H_{g} )) \cap L^{2}(\Omega \times(0, T ); V_{g} )$. Meanwhile, the solution is pathwise unique.
	\end{thm}
	\begin{proof}
		Let $\Omega_{T} := \Omega \times (0,T)$. Using the a priori estimates given in the above Proposition  \ref{prop:1.22}, it follows from the Banach-Alaoglu theorem that along a subsequence, the Galerkin approximations $\{u_{n}\}$ have the following convergences:
		\[
		u_{n} \to u \quad weakly \,\, in \,\, L^{2}(\Omega_{T}, V_{g}),
		\]
		\[
		u_{n} \to u \quad weakly \,\, star\,\,  in \,\,  L^{4}(\Omega, L^{\infty}(0,T; H_{g})),
		\]
		and 
		\[
		u_{n}(T) \to \eta \quad weakly \,\, in \,\, L^{2}(\Omega_{T}, H_{g}).
		\]
		Recall that $F(u) := -\nu A_{g}u - B_{g}(u,u)-\nu R_{g}u$. Since $F(u_{n})$ is bounded in $L^{2}(\Omega_{T} , V^{*}_{g})$,
		\[
		F(u_{n})\to F_{0} \quad weakly \,\,  in \,\, L^{2}(\Omega_{T} , V^{*}_{g}),
		\]
		Since $\sigma$ satisfies the linear growth condition (\textbf{A2}) and $u_{n}$ is uniformly bounded in $L^{2}(0, T ; V_{g} )$, we may extract a subsequence such that
		\[
		\sigma_{n}(\cdot, u_{n}) \to  S \quad  weakly\,\, in \,\, L^{2}(\Omega_{T}, L_{Q}(H_{0,g};H_{g})).
		\]
		\par 
		Because $u_n$ and $F(u_n)$ converge only weakly, the nonlinear limit cannot be identified by passing directly to the limit in the differential equation. Then, by setting $t \in [0, T]$ to $0$ outside this range, the equation \eqref{eq:1.20} is extended to the open interval $(-\delta, T + \delta)$ with $\delta > 0$. Let $\phi(t)$ be a function in $H^{1}_{g}(-\delta, T + \delta)$ with $\phi(0) = 1$. Define for all integers $j \ge 1, e _{j }(t) = \phi(t)e_{ j}$ where ${e_{ j }}$ is the fixed orthonormal sequence for $H_{g}$.
		\par
		Since $e_{j}(t)$ is a deterministic function, we regard it as a finite variation process and set $Y(t) = (u^{\varepsilon}_{n}(t), e_{j}(t))$. Using It\^o formula, we can obtain
		\begin{equation}
			\label{eq:1.24}
			\begin{aligned}
				(u_{n}(T), e_{j}(T))_{g} &= (u_{n}(0), e_{j})_{g} + \int_{0}^{T} \left( u_{n}(s), \dfrac{de_{j}(s)}{ds} \right)_{g}ds + \int_{0}^{T} (F(u_{n}(s)), e_{j}(s))_{g}ds 
				\\& + \int_{0}^{T} (f(s), e_{j}(s))_{g}ds +\sqrt{\varepsilon}\int_{0}^{T}(\sigma(s,u_{n}(s))dW_{n}(s), e_{j}(s)).
			\end{aligned}
		\end{equation}
		We pass to the limit in \eqref{eq:1.24} term by term. It remains to treat the stochastic integral for fixed $j$. 
		\par
		Let $\mathcal{P}_{T}$ denote the class of predictable processes with values in $L^{2}(\Omega_{T}; L_{Q}(H_{0,g}; H_{g}))$ with inner product given by
		\begin{equation}
			\label{eq:1.25}
			(G,J)_{\mathcal{P}_{T}} = \textbf{E}\int_{0}^{T} tr\{G(s)QJ^{*}(s)\}ds \quad \forall \,G,J \in \mathcal{P}_{T}.
		\end{equation}
		The map $J : \mathcal{P}_{T} \to L^{2}(\Omega_{T})$, defined by the real-valued integral,  $J_{t}(G) = \int_{0}^{t}(G(s)dW(s), e_{j}(s) )_{g}$ is linear and continuous. Hence, $(\sigma_{n}(s, u_{n}(s))P_{n}, R)_{\mathcal{P}_{T}} \to (SdW (s), R)_{\mathcal{P}_{T}}$  for any  $R \in \mathcal{P}_{T}$ as $n \to \infty$ by the weak convergence of $\sigma_{n}(s, u_{n}(s))P_{n} \to S$ in $L^{2}_{Q}(H_{0,g}; H_{g} )$. Thus we can conclude that
		\begin{equation}
			\label{eq:1.26}
		\int_{0}^{T} (\sigma_{n}(s, u_{n}(s))dW_{n}(s), e_{j}(s))_{g} \to \int_{0}^{T} \phi(s)(S(s)dW(s), e_{j})_{g}, \quad as\,\, n \to \infty \,\,for\,\, each\,\, j.
		\end{equation}
		and
		\begin{equation}
			\label{eq:1.27}
			\int_{0}^{T} (F_{n}( u_{n}(s)), e_{j}(s))_{g} \to \int_{0}^{T} (F_{0}, e_{j}(s))_{g}ds, \quad as \,\, n \to \infty \,\, for\,\, each\,\, j.
		\end{equation}
		Taking the limit of each term in \eqref{eq:1.24}, we obtain
		\begin{equation}
			\label{eq:1.28}
			\begin{aligned}
			(\eta, e_{j})_{g}\phi(T) &= (u_{0}, e_{j})_{g} + \int_{0}^{T} (F_{0}, e_{j})_{g}\phi(s)ds + \int_{0}^{T} \left( u(s), \dfrac{de_{j}(s)}{ds} \right)_{g}ds 
			\\& + \int_{0}^{T} (f(s), e_{j}(s))_{g}ds + \sqrt{\varepsilon}\int_{0}^{T} \phi(s)(S(s)dW(s), e_{j})_{g}.
			\end{aligned}
		\end{equation}
		Choose the approximating sequence so that its derivative converges to $\delta_{t}$ in the distributional sense. Thus, for any $\psi \in C([0,T]; H_{g})$, we have
		\begin{equation}
			\label{eq:1.29}
			\underset{k \to \infty}{\rm{lim}} \int_{0}^{T} \phi'_{k}(s) \psi(s)ds  = -\underset{k \to \infty}{\rm{lim}} \int_{0}^{T} \phi_{k}(s) \psi'(s)ds= -\psi(t).
		\end{equation}
		Choose functions $\{\phi_{k}\}$ such that $\phi_{k} \to  \textbf{1}_{[0,T]}$ and $\phi'_{k} \to \delta_{t}$ weakly as $k \to \infty$. Substituting $\phi_{k}$ for $\phi$ in \eqref{eq:1.28} and then letting $k \to \infty$, we get that
		\begin{equation}
			\label{eq:1.30}
			(u(t), e_{j})_{g} = (u_{0}, e_{j})_{g} + \int_{0}^{T}(F_{0}, e_{j})_{g}ds + \sqrt{\varepsilon}\int_{0}^{T}(S(s)dW(s), e_{j})_{g}.
		\end{equation}
		When $t < T$ with $(u_{T} , e_{ j} )_{g} = (\eta, e_{j})_{g}$ for all $j$, we obtain 
		\begin{equation}
			\label{eq:1.31}
			u(t) = u(0) + \int_{0}^{t}F_{0}(s)ds + \sqrt{\varepsilon}\int_{0}^{t}S(s)dW(s)
		\end{equation}
		with $u(T) = \eta$.
		\par
		Let $v \in L^{\infty}(\Omega_{T}; H_{m})$ for $m \le n$, and  define $r(t) = \frac{27}{\nu^{3}} \int_{0}^{t}\| v(s)\|^{4}_{L^{4}(\mathcal{O})}ds$. 	By using the monotonicity, we get
		\begin{equation}
			\label{eq:1.32}
			\begin{aligned}
			X_{n}(T) &:=  2{\textbf{E}}\int_{0}^{T} e^{-r(s)}(F(u_{n}(s))- F(v(s)), u_{n}(s)-v(s))_{g}ds 
				- {\textbf{E}}\int_{0}^{T} \dfrac{dr(s)}{dt}e^{-r(s)}\| u_{n}(s)-v(s)\|^{2}_{g} ds 
				\\& + {\textbf{E}}\int_{0}^{T}e^{-r(s)} \|\sigma_{n}(s,u_{n}(s)) -\sigma_{n}(s,v(s)) \|^{2}_{L_{Q}}ds
				\\& \le 0.
			\end{aligned}
		\end{equation}
		Write $X_{n}(T)$ = $Y_{n} + Z_{n}$, where
		\begin{equation}
			\label{eq:1.33}
			\begin{aligned}
				Y_{n} &= 2{\textbf{E}}\int_{0}^{T} (F(u_{n}(s)), u_{n}(s))_{g}e^{-r(s)}ds - {\textbf{E}}\int_{0}^{T} \dfrac{dr(s)}{dt}e^{-r(s)}\| u_{n}(s)|^{2}_{g} ds 
				\\& + {\textbf{E}}\int_{0}^{T}e^{-r(s)} \|\sigma_{n}(s,u_{n}(s))\|^{2}_{L_{Q}}ds,
			\end{aligned}
		\end{equation}
		and
		\begin{equation}
			\label{eq:1.34}
			\begin{aligned}
			Z_{n} &= 2{\textbf{E}}\int_{0}^{T} \left( (F(v(s)), v(s))_{g} -(F(v(s)), u_{n}(s))_{g} -(F(u_{n}(s), v(s))_{g}) \right)e^{-r(s)}ds 
			\\& - {\textbf{E}}\int_{0}^{T} \dfrac{dr(s)}{dt}e^{-r(s)} (\| v(s)\|^{2}_{H_{g}} - 2(u_{n}(s),v(s))_{g})ds 
			\\&  +{\textbf{E}}\int_{0}^{T}e^{-r(s)} \{ \|\sigma_{n}(s,v(s))\|^{2}_{L_{Q}} - 2(\sigma_{n}(u_{n}(s)), \sigma_{n}(v(s))_{L_{Q}}) \}ds.
			\end{aligned}
		\end{equation}
		By It\^o formula, we obtain 
		\begin{equation}
			\label{eq:1.35}
			\begin{aligned}
			e^{-r(T)}\|u_{n}(T)\|^{2}_{g} - \|u_{n}(0)\|^{2}_{g} &=  2{\textbf{E}} \int_{0}^{T}e^{-r(s)}(F(u_{n}(s)), u_{n}(s))_{g}ds - {\textbf{E}} \int_{0}^{T}\frac{dr(s)}{dt}e^{-r(s)} \|u_{n}(s)\|^{2}_{g}ds 
			\\&+ {\textbf{E}} \int_{0}^{T}e^{-r(s)}\|\sigma_{n}(s,u_{n}(s))\|^{2}_{L_{Q}}ds + 2{\textbf{E}} \int_{0}^{T}e^{-r(s)}(f(s), u_{n}(s))_{g}ds.	
			\end{aligned}
		\end{equation}
		By the lower-semicontinuity of $\|u_n(T)\|^{2}_{g}$, we obtain 
		\begin{equation}
			\label{eq:1.36}
			\begin{aligned}
			\underset{n\to \infty}{\rm{lim}} \rm{inf}Y_{n} &= {\textbf{E}}e^{-r(T)}\|u_{n}(T)\|^{2}_{g} - \|u_{n}(0)\|^{2}_{g} \ge {\textbf{E}}e^{-r(T)}\|u_{n}(T)\|^{2}_{g} - \|u(0)\|^{2}_{g} 
			\\&=  2{\textbf{E}} \int_{0}^{T}e^{-r(s)}(F(u_{0}(s)), u(s))_{g}ds - {\textbf{E}} \int_{0}^{T}\frac{dr(s)}{dt}e^{-r(s)} \|u(s)\|^{2}_{g}ds 
		\\& + {\textbf{E}} \int_{0}^{T}e^{-r(s)}\|S(s)\|^{2}_{L_{Q}}ds + {\textbf{E}} \int_{0}^{T}e^{-r(s)}(S(s)dW(s), u(s))_{g}.
			\end{aligned}
		\end{equation}
		In \eqref{eq:1.34}, by taking the limit for each term, we obtain
		\begin{equation}
			\label{eq:1.37}
			\begin{aligned}
		\underset{n\to \infty}{\rm{lim}} Z_{n} &= 2{\textbf{E}}\int_{0}^{T} \left[ (F(v(s)), v(s))_{g} -(F(v(s)), u(s))_{g} -(F_{0}(s), v(s))_{g} \right] e^{-r(s)}ds 
		\\& - {\textbf{E}}\int_{0}^{T} \dfrac{dr(s)}{ds}e^{-r(s)} (\| v(s)\|^{2}_{{g}} - 2(u(s),v(s))_{g})ds 
		\\&  + {\textbf{E}}\int_{0}^{T}e^{-r(s)}  \left(\|\sigma(s,v(s))\|^{2}_{L_{Q}} - 2(S(s), \sigma(v(s))_{L_{Q}}) \right)ds.	
			\end{aligned}
		\end{equation}
		The monotonicity condition and $X_{n} = Y_{n} + Z_{n} \le 0$ imply
		\begin{equation}
			\label{eq:1.38}
			\begin{aligned}
				&2{\textbf{E}}\int_{0}^{T} e^{-r(s)}(F_{0}(s) - F(v(s)), u(s) - v(s))_{g} ds  - {\textbf{E}}\int_{0}^{T} \dfrac{dr(s)}{ds}e^{-r(s)} \| u(s) - v(s)\|^{2}_{g}ds 
				\\&+ {\textbf{E}}\int_{0}^{T}e^{-r(s)}\|S(s) - \sigma(s,v(s))\|^{2}_{L_{Q}}ds
				\\& \le \underset{n\to \infty}{\rm{lim}}\rm{inf}\, X_{n} \le 0.
			\end{aligned}
		\end{equation}
		Setting $v = u$ gives $S(s) = \sigma(s, u(s))$. Next, set $v = u - \lambda w$ with $\lambda > 0$ to obtain 
		\begin{equation}
			\label{eq:1.39}
			\begin{aligned}
				&2\lambda{\textbf{E}}\int_{0}^{T} e^{-r(s)}(F_{0}(s) - F(u(s) - \lambda w(s)), w(s))_{g} ds + \lambda^{2}{\textbf{E}}\int_{0}^{T} \dfrac{dr(s)}{ds}e^{-r(s)} \| w(s)\|^{2}_{g}ds
				\\& + \varepsilon L \lambda^{2}\int_{0}^{T}e^{-r(s)}\|w(s)\|^{2}_{V_{g}}  \le 0.	
			\end{aligned}
		\end{equation}
		Dividing by $\lambda$ on both sides of the inequality above, and letting $\lambda \to 0$, we can conclude
		\begin{equation}
			\label{eq:1.40}
			{\textbf{E}}\int_{0}^{T}(F_{0}(s) - F(u(s)), w(s))_{g}ds \le 0.
		\end{equation}
		Since $w$ is arbitrary, $F_{0}$ = $F(u(s))$. Thus the existence of a strong solution has been proved.
		\par
		We next prove pathwise uniqueness. Let $u_{1}$ and $u_{2}$ be two solutions  of the $g$-Navier--Stokes equation  \eqref{eq:1.23}. For $i = 1, 2$, and $N > 0$, define $\tau^{ i}_{ N} := \inf\{t \le T: \|u_{i} (t)\|_{g} \ge N \}$. Set $\tau_{N} = \min\{\tau^{1}  _{N}, \tau^{2}_{N}\}$. By Chebyshev's inequality, we obtain 
		\begin{equation}
			\label{eq:1.41}
			P\{\tau_{N}  < T \} \le P\{\max \{\underset{t\in [0,T]}{\rm{sup}} \|u_{1}(t) \|_{g}, \underset{t\in [0,T]}{\rm{sup}} \|u_{2}(t) \|_{g} \}  \ge N\} \le \dfrac{C}{N^{2}}.
		\end{equation}
		 Thus $\underset{N \to \infty}{\rm{lim}} \tau_{N}= T$ a.s. 
		\par
		By \eqref{eq:1.23} and It\^o formula, we have for $s \in [0, t \wedge \tau_{N}]$,
		\begin{equation}
			\label{eq:1.42}
			\begin{aligned}
		&\| u_{1}( t \wedge \tau_{N} ) - u_{2}( t \wedge \tau_{N} ) \|^{2}_{g} + 2\nu \int_{0}^{t \wedge \tau_{N}} \|u_{1}(s)- u_{2}(s)\|^{2}_{V_{g}}ds 
		\\& + 2\int_{0}^{t \wedge \tau_{N}} (B_{g}(u_{1}(s), u_{1}(s))- B_{g}(u_{2}(s), u_{2}(s)), u_{1}(s)- u_{2}(s))_{g} ds 
		\\& + 2\nu \int_{0}^{t \wedge \tau_{N}} (R_{g}(u_{1}(s)) - R_{g}(u_{2}(s)), u_{1}(s)- u_{2}(s))_{g} ds 
		\\& =    2 \sqrt{\varepsilon} \int_{0}^{t \wedge \tau_{N}} (u_{1}(s)- u_{2}(s), (\sigma(s,u_{1}(s))- \sigma(s,u_{2}(s)))  dW(s))_{g} 
		\\& + \varepsilon \int_{0}^{t \wedge \tau_{N}} \|\sigma(s,u_{1}(s))- \sigma(s,u_{2}(s))\|^{2}_{L_{Q}} ds.
			\end{aligned}
		\end{equation}
		We now estimate the third term on the left-hand side of \eqref{eq:1.42}. By \eqref{eq:2.3} and the bilinearity of $B_{g}$, we have
		\begin{equation}
			\label{eq:1.44}
			\begin{aligned}
			&\left| b_{g}(u_{1}, u_{1}, u_{1} - u_{2}) - b_{g}(u_{2}, u_{2}, u_{1} - u_{2}) \right| 
			\\&\le | b_{g}(u_{1}, u_{1}, u_{1} - u_{2}) - b_{g}(u_{2}, u_{1}, u_{1} - u_{2}) |
			 + |b_{g}(u_{2}, u_{1}, u_{1} - u_{2}) - b_{g}(u_{2}, u_{2}, u_{1} - u_{2})|
			\\& = |b_{g}(u_{1} - u_{2}, u_{1}, u_{1} - u_{2})| + |b_{g}(u_{2}, u_{1} - u_{2}, u_{1} - u_{2})|.
		\end{aligned}
		\end{equation}
		By Ladyzhenskaya inequality and H\"older inequality, we obtain
		\begin{equation}
			\label{eq:1.45}
			\begin{aligned}
			|b_{g}(u_{1}(s) - u_{2}(s), u_{1} ,u_{1}(s) - u_{2}(s) )| & \le \int_{\mathcal{O}}^{}  |u_{1}(s) - u_{2}(s)| |\nabla u_{1}||u_{1}(s) - u_{2}(s)| gdx 
			\\& \le \|u_{1}(s) - u_{2}(s)\|^{2}_{L^{4}}\|\nabla u_{1}\|_{g}
			\\& \le \frac{\sqrt{2M_{0}}}{m_{0}}\|u_{1}(s) - u_{2}(s)\|_{g}\|u_{1}(s) - u_{2}(s)\|_{V_{g}} \|u_{1}\|_{V_{g}}
			\\& \le \frac{15\nu }{8}\|u_{1}(s) - u_{2}(s)\|^{2}_{V_{g}} + \frac{8 M_{0}}{15\nu m^{2}_{0}}\|u_{1}(s) - u_{2}(s)\|^{2}_{{g}} \|u_{1}\|^{2}_{V_{g}}.
			\end{aligned}
		\end{equation}
		For the fourth term on the left-hand side of \eqref{eq:1.42},  by \eqref{eq:2.10}, we get
		\begin{equation}
			\label{eq:1.46}
			\begin{aligned}
		2\nu \int_{0}^{t \wedge \tau_{N}} (R_{g}(u_{1}(s)) -R_{g}(u_{2}(s)), u_{1}(s) - u_{2}(s))_{g}ds  
		& \le 2\nu \int_{0}^{t \wedge \tau_{N}}\frac{\|\nabla g\|_{\infty}}{m_{0}\sqrt{\lambda_{1}}} \|u_{1}(s) - u_{2}(s)\|^{2}_{V_{g}}ds
		\\&\le \frac{\nu}{8} \int_{0}^{t \wedge \tau_{N}} \|u_{1}(s) - u_{2}(s)\|^{2}_{V_{g}}ds.
			\end{aligned}
		\end{equation}
		For the last term on the right-hand side of \eqref{eq:1.42}, by condition (\textbf{A2}), we have 
		\begin{equation}
			\label{eq:1.43}
			\begin{aligned}
			\varepsilon \int_{0}^{t \wedge \tau_{N}} \|\sigma(s,u_{1}(s))- \sigma(s,u_{2}(s))\|^{2}_{L_{Q}} ds  \le \varepsilon L\, \int_{0}^{t \wedge \tau_{N}}  \|u_{1}(s) - u_{2}(s)\|^{2}_{g}ds. 
			\end{aligned}
		\end{equation}
		Substituting \eqref{eq:1.45} - \eqref{eq:1.43} into \eqref{eq:1.42}, and taking the expectation, we obtain
		\begin{equation}
			\label{eq:1.47}
			\begin{aligned}
	 \textbf{E} \| u_{1}( t \wedge \tau_{N} ) - u_{2}( t \wedge \tau_{N} ) \|^{2}_{g} 	\le \textbf{E} \int_{0}^{t \wedge \tau_{N}} \|u_{1}(s) - u_{2}(s)\|^{2}_{g} \left( \frac{8 M_{0}}{15 \nu m^{2}_{0}}\|u_{1} (s) \|^{2}_{V_{g}} + \varepsilon L  \right)ds. 
			\end{aligned}
		\end{equation}
		Finally, using Gronwall inequality we know that
		\begin{equation}
			\label{eq:4.58}
			\begin{aligned}
			\textbf{E} \| u_{1}( t \wedge \tau_{N} ) - u_{2}( t \wedge \tau_{N} ) \|^{2}_{g} = 0.
			\end{aligned}
		\end{equation}
		Taking $N \to \infty$ and applying Fatou’s lemma, the pathwise uniqueness can be proved. This proof is completed.
	\end{proof}
	\subsection{Weak convergence criterion for a large deviation principle}
\par
In this subsection, we recall the weak convergence method for the LDP of stochastic partial differential equations which will be used in this paper. Let $(\Omega, \mathcal{F}, \mathbb{P})$ be a probability space equipped with an increasing family $\{\mathcal{F}_{t} \}_{0\le t\le T}$ of sub $\sigma$-fields of $\mathcal{F}$ satisfying the usual conditions of right continuity and $\mathbb{P}$-completeness. \par
Given $M > 0$, denote by
\begin{equation}
	\label{eq:1.51}
	\begin{aligned}
	\mathcal{S}_{M} = \left\{h \in L^{2}(0,T; H_{0,g}) : \int_{0}^{T} \| h(t)\|^{2}_{0} dt \le M \right\}.
	\end{aligned}
\end{equation}
\par
Then, $\mathcal{S}_{M}$ is a Polish space endowed with the weak topology. Throughout this paper, we always assume $\mathcal{S}_{M}$ is equipped with the weak topology, unless otherwise stated.
\par
Let $\mathcal{A}$ be the space of all $H_{0,g}$-valued stochastic processes $h$ which are progressively measurable with respect to $\{\mathcal{F}_{t}\}_{t\in [0,T]}$ and $ \int_{0 }^{T} \| h(t) \|^{2}_{0} dt < \infty$ $\mathbb{P}$-almost surely. Denote $\mathcal{A}_{M} = \{h \in \mathcal{A} : h(\omega) \in \mathcal{S}_{M} \,\, for \,\, almost \,\,  all\,\, \omega \in \Omega \}$. 
\par
Let $\mathcal{E}$ be a Polish space. For any $\varepsilon > 0$, let $\mathcal{G}^{\varepsilon}: C([0, T], H_{g}) \to \mathcal{E}$ be a measurable map. Denote by
\begin{equation}
	\label{eq:1.52}
	\begin{aligned}
		u^{\varepsilon} = \mathcal{G}^{\varepsilon}(W(\cdot)), \, \forall\,\, \varepsilon > 0. 
	\end{aligned}
\end{equation}
\begin{defn}\label{drfn:1.8}(Rate function). 
	A function $I: \mathcal{E} \to [0, \infty]$ is called a rate function on $\mathcal{E}$ if it is lower semicontinuous in $\mathcal{E}$. A rate function $I$ on $\mathcal{E}$ is said to be a good rate function on $\mathcal{E}$ if for every $0 \le C < \infty$, the level set $\{x \in \mathcal{E}: I(x) \le C\}$ is a compact subset of $\mathcal{E}$.
\end{defn}
\begin{defn}\label{drfn:1.9}(Large Deviation Principle). 
	The family $\{u^{\varepsilon}\}$ is said to satisfy the LDP in $\mathcal{E}$ with a rate function $I: \mathcal{E} \to [0, \infty]$, if for every Borel subset $A$ of $\mathcal{E}$,
	\begin{equation}
		\label{eq:1.53}
		\begin{aligned}
			- \underset{x\in A^{\circ}}{\inf}\, I(x)  \le \underset{\varepsilon \to 0}{\lim\, \inf} \varepsilon\,  \log\mathbb{P} (u^{\varepsilon} \in A) \le \underset{\varepsilon \to 0}{\lim\, \sup}\,  \varepsilon\,  \log\mathbb{P}(u^{\varepsilon} \in A) \le - \underset{x\in \bar{A}}{\inf}\, I(x),
		\end{aligned}
	\end{equation}
	where $A^{\circ}$ and $\bar{A}$ are the interior and the closure of $A$ in $\mathcal{E}$, respectively.
\end{defn}
Since  $\mathcal{E}$ is a Polish space, the Laplace principle and the LDP are equivalent in  $\mathcal{E}$ (see \cite{ref A. Budhiraja1, ref P. Dupuis}). Next, we recall the concept of Laplace principle. 
\begin{defn}\label{drfn:1.10}(Laplace principle). The family $\{u^{\varepsilon}\}$ is said to satisfy the Laplace principle in $\mathcal{E}$ with a rate function $I: \mathcal{E} \to [0, \infty]$ if for all bounded and continuous $\phi:  \mathcal{E} \to \mathbb{R}$, 
	\begin{equation}
		\label{eq:1.54}
		\begin{aligned}
			\underset{\varepsilon \to 0}{\lim}\, \varepsilon\, \log \mathbb{E}\left[e^{-\frac{1}{\varepsilon}\phi(u^{\varepsilon})} \right] = -\underset{x\in \mathcal{E}}{\inf}\{\phi(x) + I(x)  \}.
		\end{aligned}
	\end{equation}
\end{defn}
In order to prove the LDP of $u^{\varepsilon}$, we will assume that the family $\{\mathcal{G}^{\varepsilon} \}$ fulfills the following conditions: there exists a measurable map $\mathcal{G}^{\varepsilon} : C([0,T], H_{g}) \to \mathcal{E}$ such that \par
(\textbf{H1}) For every $M > 0$, the set $\{\mathcal{G}^{\circ}(\int_{0}^{\cdot} h(t)dt) : h \in \mathcal{S}_{M} \}$ is a  compact subset of $\mathcal{E}$. \par
(\textbf{H2}) If $M > 0$ and $\{h^{\varepsilon} \} \subseteq \mathcal{A}_{M}$ such that $\{h^{\varepsilon} \}$ converges in distribution to $h$ as $\mathcal{S}_{M}$-valued stochastic variables, then $\mathcal{G}^{\varepsilon}(W(\cdot) + \varepsilon^{-\frac{1}{2}}\int_{0}^{.}h^{\varepsilon}(t)dt)$ converges in distribution to $\mathcal{G}^{\circ}( \int_{0}^{.}h(t)dt)$. \par
We recall the following theorem proved by Budhiraja and Dupuis \cite{ref P. Dupuis}, which will be used to establish the large deviation of  \eqref{eq:2.6} - \eqref{eq:2.7} .
\begin{prop}\label{prop:1.11}
	Under (\textbf{H1})-(\textbf{H2}), the family $\{u^{\varepsilon}\}$, given by \eqref{eq:1.52}, satisfies the Laplace principle in $\mathcal{E}$ with rate function $I$ defined by, for every $x \in \mathcal{E}$,
	\begin{equation}
		\label{eq:1.55}
		\begin{aligned}
		I(x) = \inf \left\{\frac{1}{2}   \int_{0}^{T}\| h(t)\|^{2}_{0}dt: h\in L^{2}(0,T; H_{0,g})  \quad such\,\, that \quad  \mathcal{G}^{\circ} ( \int_{0}^{.}h(t)dt) =x \right\},
		\end{aligned}
	\end{equation}
	where the infimum over an empty set is taken to be $\infty$.
\end{prop}
The main result of this paper is given below.
\begin{thm}
	\label{thm:1.12}
	Suppose that (\textbf{A1})--(\textbf{A3}) hold and $f \in L^{2}(0,T; V^{*}_{g})$. Let $u^{\varepsilon}$ denote the solution of \eqref{eq:2.6}. Then the family $\{u^{\varepsilon}\}$, as $\varepsilon \to 0$, satisfies the LDP in $C([0, T]; H_{g}) \cap L^{2}(0,T; V_{g})$ with the good rate function as given by \eqref{eq:1.55}.
\end{thm}

\section{The LDP of 2D stochastic $g$-Navier--Stokes equations}
\indent
\par
In this section, we prove the LDP of \eqref{eq:2.6}-\eqref{eq:2.7} in $C([0, T]; H_{g}) \cap L^{2}(0,T; V_{g})$  by verifying (\textbf{H1}) and (\textbf{H2}) in Proposition \ref{prop:1.11}. As usual, to study the LDP of $u^{\varepsilon}$ as $\varepsilon \to 0$, we need to consider the controlled system corresponding to \eqref{eq:2.6}. Given a control $v \in L^{2}(0, T; H_{0,g})$, the controlled equation is as follows:
\begin{equation}
	\label{eq:1.56}
	\begin{aligned}
		\begin{split}
			\left \{
			\begin{array}{ll}
					&du_{v}(t) + (\nu A_{g}u_{v}(t) + B_{g}(u_{v}(t), u_{v}(t)) + \nu R_{g}u_{v}(t))dt = f(t)dt + \sigma(t,u_{v}(t))v(t)dt,
					\\& u_{v}(0) = u_{0} \in H_{g}.
			\end{array}
			\right.
		\end{split}
	\end{aligned}
\end{equation}
By a solution $u_{v}$ of \eqref{eq:1.56}, we mean that $u_{v} \in C([0, T]; H_{g})$ such that for all $t \in [0, T]$ and $\eta \in H_{g}$,
\begin{equation}
	\label{eq:3.35}
	\begin{aligned}
		(u_{v}(t), \eta)_{g} &+ \int_{0}^{t}(\nu A_{g}u_{v}(s) + B_{g}(u_{v}(s), u_{v}(s)) + \nu R_{g}u_{v}(s), \eta)_{g}ds = (u_{0}, \eta)_{g} 
		\\& + \int_{0}^{t}(f(s),\eta)_{g}ds + \int_{0}^{t}(\sigma(s,u_{v}(s))v(s), \eta)_{g}ds.
	\end{aligned}
\end{equation}
The existence and uniqueness of solutions of  \eqref{eq:1.56} are given below.
\begin{lem}
	\label{lem:1.13}
	Suppose that $f \in L^{4}(0,T; V^{*}_{g})$  and (\textbf{A1})-(\textbf{A3}) hold. Then for every $v\in L^{2}(0,T; H_{0,g})$ problem  \eqref{eq:1.56}  has a unique solution $u_{v} \in C([0, T], H_{g}) \cap L^{2}( 0,T; V_{g})$.\par 
	Moreover, for any $\|u_{0}\|_{{g}} \vee \|u_{0,1}\|_{{g}} \vee \|u_{0,2}\|_{{g}} \le R_{1}$ and $\|v\|_{L^{2}(0,T; H_{0,g})}$ $ \vee \|v_{1}\|_{L^{2}(0,T; H_{0,g})} \vee \|v_{2}\|_{L^{2}(0,T; H_{0,g})}$ $\le R_{2}$, the solutions  $u_{v}, u_{v_{1}}$, and $u_{v_{2}}$ of \eqref{eq:1.56} with initial data $u_{0}, u_{0,1}$, and $u_{0,2}$, respectively, satisfy for all $t \in [0, T]$,
	\begin{equation}
		\label{eq:1.84}
		\begin{aligned}
		\| u_{v}(t)\|^{2}_{g} + \nu \int_{0}^{t}\| u_{v}(s)\|^{2}_{V_{g}} ds \le C_{1}, 
		\end{aligned}
	\end{equation}
	\begin{equation}
		\label{eq:1.85}
		\begin{aligned}
			\| u_{v_{1}}(t)  - u_{v_{2}}(t) \|^{2}_{g} + \nu \int_{0}^{t}\| u_{v_{1}}(t)  - u_{v_{2}}(t) \|^{2}_{V_{g}} ds  \le C_{1} \left(\|u_{0,1} - u_{0,2}\|^{2}_{g} + \|v_{1} - v_{2}\|^{2}_{L^{2}(0,T; H_{0,g})} \right), 
		\end{aligned}
	\end{equation}
	where $R_{1}, R_{2} > 0$ are given, and $C_{1} > 0$ depends on $R_{1}, R_{2}$, and $T$.
\end{lem}
\begin{proof}
Under conditions (\textbf{A1})--(\textbf{A3}), by the argument of [\cite{ref Da Prato3}, Theorem 7.10], one can show that for every $T > 0$, $v \in L^{2}(0, T; H_{0,g})$ and $u_{0} \in H_{g}$, problem \eqref{eq:1.56}  has a unique solution $u_{v}\in C([0, T], H_{g}) \cap L^{2}( 0,T; V_{g})$. 
\par 
Next, we prove \eqref{eq:1.84}. By \eqref{eq:2.10}, \eqref{eq:1.56}, we get
\begin{equation}
	\label{eq:1.86}
	\begin{aligned}
	d\| u_{v}(t)\|^{2}_{g} + 2\nu \|u_{v}(t)\|^{2}_{V_{g}}dt \le \frac{\nu }{8} \|u_{v}(t)\|^{2}_{V_{g}}dt + 2(f(t), u_{v}(t))_{g} dt
	 +   2(\sigma(t,  u_{v}(t))v(t),  u_{v}(t))_{g}dt.
	\end{aligned}
\end{equation}
For the second term on the right-hand side of \eqref{eq:1.86}, by H\"older inequality, we obtain
\begin{equation}
	\label{eq:4.60}
	\begin{aligned}
	2\left|(f(t), u_{v}(t))_{g}  \right| \le \dfrac{\nu }{8}\|u_{v}(t)\|^{2}_{V_{g}} + \frac{8}{\nu} \|(f(t)\|^{2}_{V^{*}}. 
	\end{aligned}
\end{equation}
For the last term on the right-hand side of \eqref{eq:1.86}, by (\textbf{A2}) and  Young's inequality, we find
\begin{equation}
	\label{eq:1.87}
	\begin{aligned}
	2(\sigma(t,  u_{v}(t))v(t),  u_{v}(t))_{g} &= 2\int_{\mathcal{O}}^{}\sigma(t,  u_{v}(t))v(t) \cdot u_{v}(t)\,gdx 
	\\& \le \frac{3\nu }{4KR_{2}}\|\sigma(t,  u_{v}(t))\|^{2}_{L_{Q}} \|v(t)\|^{2}_{L^{2}(0,T; H_{0,g})} + \frac{4KR_{2}}{3\nu }\|u_{v}(t)\|^{2}_{g}
	\\& \le \frac{3\nu }{4}\|u_{v}(t)\|^{2}_{V_{g}} + \frac{3\nu }{4} + \frac{4KR_{2}}{3\nu }\|u_{v}(t)\|^{2}_{{g}}. 
	\end{aligned}
\end{equation}
Combining \eqref{eq:4.60}- \eqref{eq:1.87}, we get for any $t \in [0,T]$,
\begin{equation}
	\label{eq:1.88}
	\begin{aligned}
		d\| u_{v}(t)\|^{2}_{g}  + \nu \|u_{v}(t)\|^{2}_{V_{g}}dt \le \frac{4KR_{2}}{3\nu }\|u_{v}(t)\|^{2}_{{g}}dt + \frac{3\nu }{4} + \dfrac{8}{\nu}\|f(t)\|^{2}_{V^{*}}dt.
	\end{aligned}
\end{equation}
From \eqref{eq:1.88} and Gronwall  inequality, it follows that for any $t \in [0,T]$,
\begin{equation}
	\label{eq:1.89}
	\begin{aligned}
		\| u_{v}(t)\|^{2}_{g} \le \|u_{0}\|^{2}_{g}e^{\frac{4KR_{2}T}{3\nu }} + e^{\frac{4KR_{2}T}{3\nu }} \int_{0}^{T} \left(  \dfrac{8}{ \nu} \|f(s)\|^{2}_{V^{*}} + \frac{3\nu}{4} \right)ds.
	\end{aligned}
\end{equation}
\par 
By \eqref{eq:1.89} we infer that for every $R_{1} > 0$ and $R_{2} > 0$ , there exists $\textbf{c}_{1} = \textbf{c}_{1}(R_{2}, K, T,\nu  ) > 0$ such that for any $u_{0} \in H_{g}$ with $\|u_{0}\|_{g} \le R_{1}$ and any $v \in L^{2}(0, T ; H_{0,g})$ with $\|v\|_{L^{2}(0, T ; H_{0,g})} \le  R_{2}$, the solution $u_{v}$ satisfies, for all $t \in [0, T]$,
\begin{equation}
	\label{eq:4.40}
	\begin{aligned}
		\| u_{v}(t)\|^{2}_{g} \le \textbf{c}_{1}, \forall t\in[0,T].
	\end{aligned}
\end{equation}
\par 
Integrating \eqref{eq:1.88} over $[0, T ]$ and using \eqref{eq:4.40}, we find a constant $\textbf{c}_{2} = \textbf{c}_{2}(R_{1}, R_{2}, K, T,\nu) > 0$ such that for all $u_{0}\in H_{g}$ with $\|u_{0}\|_{g} \le R_{1}$ and any $v \in L^{2}(0, T ; H_{0,g})$ with $\|v\|_{L^{2}(0, T ; H_{0,g})} \le  R_{2}$, 
\begin{equation}
	\label{eq:4.41}
	\begin{aligned}
	\nu \int_{0}^{t}\|u_{v}(s)\|^{2}_{V_{g}}ds \le \textbf{c}_{2}. 
	\end{aligned}
\end{equation}
Then \eqref{eq:1.84} follows from \eqref{eq:4.40}-\eqref{eq:4.41}.
\par
Next, we prove \eqref{eq:1.85}. Let $u_{v_{1}}$ and $u_{v_{2}}$ be the solutions of \eqref{eq:1.56} with initial data $u_{0,1}$ and $u_{0,2}$, respectively. By \eqref{eq:1.84}, we find that if $ \|u_{0,1}\|_{{g}} \vee \|u_{0,2}\|_{{g}} \le R_{1}$ and $\| v_{1}\|_{L^{2}(0,T;H_{0,g})} \vee \| v_{2}\|_{L^{2}(0,T;H_{0,g})} \le  R_{2}$, then
\begin{equation}
	\label{eq:1.91}
	\begin{aligned}
		\| u_{v_{1}}(t)\|_{g}  + \| u_{v_{2}}(t)\|_{g} \le 2 C_{1},\,\, \forall t\in[0,T],
	\end{aligned}
\end{equation}
where $C_{1} = C_{1}(R_{1}, R_{2}, T) > 0$. Similarly, by \eqref{eq:2.10}, \eqref{eq:1.56}, we obtain
\begin{equation}
	\label{eq:1.92}
	\begin{aligned}
		& d\| u_{v_{1}}(t) - u_{v_{2}}(t)\|^{2}_{g} + 2\nu \|u_{v_{1}}(t) - u_{v_{2}}(t)\|^{2}_{V_{g}}dt 
		\\& + (B_{g}(u_{v_{1}}(t), u_{v_{1}}(t)) -B_{g}(u_{v_{2}}(t), u_{v_{2}}(t)), u_{v_{1}}(t) - u_{v_{2}}(t))_{g}dt \\&\le \dfrac{\nu }{8} \|u_{v_{1}}(t) - u_{v_{2}}(t)\|^{2}_{V_{g}} + 2(\sigma(t, u_{v_{1}}(t))v_{1}(t)-\sigma(t, u_{v_{2}}(t))v_{2}(t),  u_{v_{1}}(t) - u_{v_{2}}(t))_{g}dt.
	\end{aligned}
\end{equation}
For the third term on the left-hand side of \eqref{eq:1.92}, by \eqref{eq:2.3}, we get
\begin{equation}
	\label{eq:2.50}
	\begin{aligned}
    &\left| (B_{g}(u_{v_{1}}(t), u_{v_{1}}(t)) -B_{g}(u_{v_{2}}(t), u_{v_{2}}(t)), u_{v_{1}}(t) - u_{v_{2}}(t))_{g} \right|
    \\& \le  \left|b_{g}(u_{v_{1}}(t), u_{v_{1}}(t), u_{v_{1}}(t) - u_{v_{2}}(t))  -b_{g}(u_{v_{1}}(t), u_{v_{1}}(t), u_{v_{2}}(t) - u_{v_{2}}(t))\right|
    \\& + \left|b_{g}(u_{v_{1}}(t), u_{v_{2}}(t), u_{v_{1}}(t) - u_{v_{2}}(t))  -b_{g}(u_{v_{2}}(t), u_{v_{2}}(t), u_{v_{1}}(t) - u_{v_{2}}(t))\right|
    \\& \le \left|b_{g}(u_{v_{1}}(t), u_{v_{1}}(t)- u_{v_{2}}(t)) , u_{v_{1}}(t) - u_{v_{2}}(t)) \right| +  \left|b_{g}(u_{v_{1}}(t) - u_{v_{2}}(t), u_{v_{2}}(t) , u_{v_{1}}(t) - u_{v_{2}}(t)) \right|
    \\& = \left|b_{g}(u_{v_{1}}(t) - u_{v_{2}}(t), u_{v_{2}}(t) , u_{v_{1}}(t) - u_{v_{2}}(t)) \right|.
	\end{aligned}
\end{equation}
By Ladyzhenskaya inequality and H\"older inequality, we obtain
\begin{equation}
	\label{eq:2.51}
	\begin{aligned}
	& \left|b_{g}(u_{v_{1}}(t) - u_{v_{2}}(t), u_{v_{2}}(t) , u_{v_{1}}(t) - u_{v_{2}}(t)) \right| 
	\\& \le \|u_{v_{1}}(t) - u_{v_{2}}(t)\|^{2}_{L^{4}}\|\nabla u_{v_{2}}(t)\|_{g}. 
	\\& \le \frac{\sqrt{2 M_{0}}}{m_{0}} \|u_{v_{1}}(t) - u_{v_{2}}(t)\|_{g}\|u_{v_{1}}(t) - u_{v_{2}}(t)\|_{V_{g}}\|\nabla u_{v_{2}}(t)\|_{g}
	\\& \le \frac{7\nu}{8} \|u_{v_{1}}(t) - u_{v_{2}}(t)\|^{2}_{V_{g}} + \frac{8M_{0}}{7\nu m^{2}_{0}}\|u_{v_{1}}(t) - u_{v_{2}}(t)\|^{2}_{g} \| u_{v_{2}}(t)\|^{2}_{V_{g}}. 
	\end{aligned}
\end{equation}
For the last term on the right-hand side of \eqref{eq:1.92}, by (\textbf{A2}) and (\textbf{A3}) we get for any $t\in[0,T]$,
\begin{equation}
	\label{eq:1.93}
	\begin{aligned}
		& 2\left|(\sigma(t, u_{v_{1}}(t))v_{1}(t)-\sigma(t, u_{v_{2}}(t))v_{2}(t),  u_{v_{1}}(t) - u_{v_{2}}(t))_{g}\right|
		\\& =    2 \left|(\sigma(t, u_{v_{1}}(t))v_{1}(t)-\sigma(t, u_{v_{2}}(t))v_{1}(t),  u_{v_{1}}(t) - u_{v_{2}}(t))_{g}\right| 
		\\& + 2 \left| (\sigma(t, u_{v_{2}}(t))v_{1}(t)-\sigma(t, u_{v_{2}}(t))v_{2}(t),  u_{v_{1}}(t) - u_{v_{2}}(t))_{g}\right|
		\\& \le 2\| \sigma(t, u_{v_{1}}(t))-\sigma(t, u_{v_{2}}(t))\|_{L_{Q}}\|v_{1}(t)\|_{L^{2}(0,T; H_{0,g})} \|u_{v_{1}}(t) - u_{v_{2}}(t)\|_{g} 
		\\&+ 2\| \sigma(t, u_{v_{2}}(t))\|_{L_{Q}}\|v_{1}(t) - v_{2}(t)\|_{L^{2}(0,T ; H_{0,g})} \|u_{v_{1}}(t) - u_{v_{2}}(t)\|_{g}
		\\& \le 2\sqrt{L}\|v_{1}(t)\|_{L^{2}(0,T ; H_{0,g})} \|u_{v_{1}}(t) - u_{v_{2}}(t)\|^{2}_{g} + c_{3} \|v_{1}(t) - v_{2}(t)\|^{2}_{L^{2}(0,T; H_{0,g})} 
		\\& + c_{3}\|u_{v_{1}}(t) - u_{v_{2}}(t)\|^{2}_{g},
	\end{aligned}
\end{equation}
where $c_{3} = \sqrt{K( 1+ C^{2}_{1})}$. Therefore, by \eqref{eq:2.50}- \eqref{eq:1.93}, we find that for any $t \in [0,T]$, 
\begin{equation}
	\label{eq:1.94}
	\begin{aligned}
		&d\| u_{v_{1}}(t) - u_{v_{2}}(t)\|^{2}_{g} + \nu \int_{0}^{t}\| u_{v_{1}}(s)  - u_{v_{2}}(s) \|^{2}_{V_{g}} ds
		\\& \le  \left( 2\sqrt{L}\|v_{1}(t)\|_{L^{2}(0,T ; H_{0,g})} + \frac{8M_{0} C_{1}}{7 \nu m^{2}_{0}} + c_{3} \right) \|u_{v_{1}}(t) - u_{v_{2}}(t)\|^{2}_{g} + c_{3}\|v_{1}(t) - v_{2}(t)\|^{2}_{L^{2}(0,t ; H_{0,g})}.
	\end{aligned}
\end{equation}
From \eqref{eq:1.94} and Gronwall’s inequality, we obtain that for any $t \in [0,T]$, 
\begin{equation}
	\label{eq:1.95}
	\begin{aligned}
		 \| u_{v_{1}}(t) - u_{v_{2}}(t)\|^{2}_{g}   
		 & \le  e^{\int_{0}^{T}\left( 2 \sqrt{L}\|v_{1}(s)\|_{L^{2}(0,t ; H_{0,g})} + \frac{8M_{0} C_{1}}{7 \nu m^{2}_{0}}  + c_{3}  \right)ds} \|u_{0,1} - u_{0,2}\|^{2}_{g} 
		 \\&+ c_{3}e^{\int_{0}^{T}\left( 2 \sqrt{L}\|v_{1}(s)\|_{L^{2}(0,t ; H_{0,g})} +\frac{8M_{0} C_{1}}{7 \nu m^{2}_{0}} +c_{3}  \right)ds}\int_{0}^{t}\|v_{1}(s) - v_{2}(s)\|^{2}_{L^{2}(0,t ; H_{0,g})}ds
		 \\& \le \textbf{c}_{3} \left(\|u_{0,1} - u_{0,2}\|^{2}_{g} + \|v_{1} - v_{2}\|^{2}_{L^{2}(0,T; H_{0,g})} \right), 
	\end{aligned}
\end{equation}
where $\textbf{c}_{3} = \textbf{c}_{3}(R_{1}, R_{2},\nu, T,K) > 0$. By \eqref{eq:1.94} - \eqref{eq:1.95}, we get
\begin{equation}
	\label{eq:4.42}
	\begin{aligned}
	\nu \int_{0}^{t}\| u_{v_{1}}(s)  - u_{v_{2}}(s) \|^{2}_{V_{g}} ds \le \textbf{c}_{4} \left(\|u_{0,1} - u_{0,2}\|^{2}_{g} + \|v_{1} - v_{2}\|^{2}_{L^{2}(0,T; H_{0,g})} \right),
	\end{aligned}
\end{equation}
where $\textbf{c}_{4} = \textbf{c}_{4}(R_{1}, R_{2},\nu, T,K) > 0$, which along with \eqref{eq:1.95} yields \eqref{eq:1.85}.
\end{proof}

Next, we derive the uniform estimates for the tails of segments of solutions by using scalar stream function as in \cite{ref Wang. B4}. These estimates will   be crucial for addressing the non-compactness of Sobolev embeddings on unbounded domains. 
\par 
Let $\varrho : \mathbb{R}  \to  \mathbb{R}$ be a smooth function such that $0 \le \varrho \le 1$ and 
\[ 
	\varrho(s) = 0  \quad for \quad  |s| \le \dfrac{1}{2}, \quad and \quad \varrho(s) = 1 \quad for \quad |s| \ge  1.
\]
Note that there exists $C > 0$ such that $|\varrho'(s)|+|\varrho''(s)|\le C$ for all $s \in \mathbb{R}$.
Given  $k \in \mathbb{N} $ and $y = (y_{1}, y_{2}) \in \mathcal{O}$,  let $\varrho_{k}(x) = \varrho( \frac{x_{1}}{k})$.
Next, we introduce a scalar stream function for \eqref{eq:1.1} as in \cite{ref Wang. B4}. Given $u = (u_{1}, u_{2}) \in V_{g} $ and $x = (x_{1}, x_{2}) \in \overline{\mathcal{O}}$ define
\begin{equation}
	\label{eq:3.65}
	\begin{aligned}
	\tilde{u}(x) = \tilde{u}(x_{1}, x_{2}) = \int_{(0,0)}^{(x_{1}, x_{2})} -u_{2} \, gdx_{1} + u_{1}\, gdx_{2}.
	\end{aligned}
\end{equation}
Since $u = 0$ on $\partial \mathcal{O}$ and $\nabla \cdot (gu) = 0$ in $\mathcal{O}$, the integral in \eqref{eq:3.65} is independent of the path of integration. Hence, we have
\begin{equation}
	\label{eq:3.66}
	\begin{aligned}
		\partial_{ x_{1}}\tilde{u} = -gu_{2} ,\,\quad \partial_{ x_{2}}\tilde{u} = gu_{1}. 
	\end{aligned}
\end{equation}
and
\begin{equation}
	\label{eq:3.69}
	\begin{aligned}
		\tilde{u}\vert_{\partial \mathcal{O}} = 0 ,\,\quad \nabla \tilde{u}\vert_{\partial \mathcal{O}} = 0. 
	\end{aligned}
\end{equation}
Let $\mathcal{L}$ be the curl operator given by
\begin{equation}
	\label{eq:3.70}
	\begin{aligned}
		\mathcal{L} u &= \partial_{ x_{2}}u_{1} - \partial_{ x_{1}}u_{2} = \nabla \cdot (\frac{1}{g}\nabla\tilde{u} ) \triangleq  \mathcal{L}_{g}\tilde{u}. 
	\end{aligned}
\end{equation}
Applying the curl operator to the controlled equation \eqref{eq:1.56} and using \eqref{eq:3.66} and \eqref{eq:3.70}, we obtain
\begin{equation}
	\label{eq:3.71}
	\begin{aligned}
	\partial_{t}  \mathcal{L}_{g}\tilde{u} - \nu \Delta \mathcal{L}_{g}\tilde{u} +  \tilde{B_{g}}(\tilde{u}, \tilde{u}) = \sqrt{\varepsilon} \tilde{\sigma}(t,u)v(t), 
	\end{aligned}
\end{equation}
where 
\[\tilde{B_{g}}(\tilde{u}, \tilde{u}) = -\partial_{ x_{2}}[(\frac{1}{g}\partial_{ x_{1}} \tilde{u} ) (\mathcal{L}_{g}\tilde{u} )] + \partial_{ x_{1}}[(\frac{1}{g}\partial_{ x_{2}} \tilde{u} ) (\mathcal{L}_{g}\tilde{u} )]\,\, , \,\tilde{\sigma}(t,u) = \mathcal{L}\sigma(t,u). 
\]
\par 
For convenience, for every $k \in \mathbb{N},$ denote by $\mathcal{O}_{k} = (-k, k) \times (0, d)$. By using equation \eqref{eq:3.71} for the stream function, we can derive the uniform estimates on the tails of the solutions of the controlled equation on $\mathcal{O} \backslash \mathcal{O}_{k}$, which is given below.
\begin{lem}
	\label{lem:3.57}
	Suppose that (\textbf{A1})--(\textbf{A3}) hold. Then for every $\varepsilon'> 0$, $T > 0$, $R > 0$, and $u_{0} \in H_{g}$, there exists $K^{*} = K^{*}(\alpha, \lambda_{1}, T, u_{0}, R_{3}, \varepsilon') > 0$ with $\alpha < \frac{\nu \lambda_{1} - 2\sqrt{\varepsilon L}R_{3}M_{0}}{\nu }$ such that for all $v \in L^{2}(0, T ; H_{0,g})$ with $\| v \|_{L^{2}(0, T ; H_{0,g})} \le R_{3}$, the solution $u_{v}$ of \eqref{eq:1.56} satisfies, for all $k \ge K^{*}$ and $t \in [0,T]$,
	\begin{equation}
		\label{eq:1.96}
		\begin{aligned}
		\int_{|x|\ge k}^{} |u_{v}(t,x)|^{2} dx  < \varepsilon'. 
		\end{aligned}
	\end{equation}
\end{lem}
\begin{proof}
	Let $\varrho : \mathbb{R}^{n}  \to  [0, 1]$ be a smooth function satisfying
	\begin{equation}
		\label{eq:1.97}
		\begin{aligned}
			\varrho(x) = 0  \quad for \quad  |x| \le \dfrac{1}{2}, \quad and \quad \varrho(x) = 1 \quad for \quad |x| \ge  1.
		\end{aligned}
	\end{equation}
	Note that there exists $C > 0$ such that $|\varrho'(s)|+|\varrho''(s)|\le C$ for all $s \in \mathbb{R}$. Set $\varrho_{k}(x) =  \varrho(\frac{x}{k})$ for $k \in \mathbb{N}$. Multiplying \eqref{eq:3.71} by $- \varrho^{2}_{k}\tilde{u}_{v}$ and integrating this result over $\mathcal{O}$, we obtain
	\begin{equation}
		\label{eq:1.98}
		\begin{aligned}
		-(\partial_{t}  \mathcal{L}_{g}\tilde{u}_{v}, \varrho^{2}_{k}\tilde{u}_{v})_{g} = - (\nu \Delta \mathcal{L}_{g}\tilde{u}_{v}, \varrho^{2}_{k}\tilde{u}_{v})_{g} +  (\tilde{B_{g}}(\tilde{u}_{v}, \tilde{u}_{v}),\varrho^{2}_{k}\tilde{u}_{v} )_{g}- (\sqrt{\varepsilon} \tilde{\sigma}(t,u_{v})v(t),\varrho^{2}_{k}\tilde{u}_{v})_{g}.
		\end{aligned}
	\end{equation}
	For the left-hand side of \eqref{eq:1.98} we have
	\begin{equation}
		\label{eq:3.90}
		\begin{aligned}
			-(\partial_{t}  \mathcal{L}_{g}\tilde{u}_{v}, \varrho^{2}_{k}\tilde{u}_{v})_{g} &= - \int_{\mathcal{O}}^{}\partial_{t} (\mathcal{L}_{g}\tilde{u}_{v})(\varrho^{2}_{k}\tilde{u}_{v}) gdx = -\int_{\mathcal{O}}^{}\partial_{t} (\frac{1}{g} \Delta \tilde{u}_{v} + \nabla(\frac{1}{g})\cdot \tilde{u}_{v} )(\varrho^{2}_{k}\tilde{u}_{v}) gdx 
			\\&= \int_{\mathcal{O}}^{} \nabla (\partial_{t}\tilde{u}_{v})\cdot (\nabla \varrho^{2}_{k}\tilde{u}_{v})dx - \int_{\mathcal{O}}^{} \partial_{t}(\nabla(\frac{1}{g})\cdot\nabla \tilde{u}_{v})\cdot (\varrho^{2}_{k}\tilde{u}_{v}) gdx
			\\& = \int_{\mathcal{O}}^{} \nabla (\partial_{t}\tilde{u}_{v})\cdot ( \varrho^{2}_{k}\nabla\tilde{u}_{v}  + 2\varrho_{k}\tilde{u}_{v}\nabla \varrho_{k} )dx + \int_{\mathcal{O}}^{} \partial_{t}((\frac{\nabla g}{g^{2}})\cdot\nabla \tilde{u}_{v})\cdot (\varrho^{2}_{k}\tilde{u}_{v}) gdx 
			\\& \ge \frac{1}{2}\frac{d}{dt} \int_{\mathcal{O}}^{}\varrho^{2}_{k}|\nabla \tilde{u}_{v}|^{2}dx - \frac{2}{k} \|\nabla \tilde{u}_{v_{t}}\| \|\tilde{u}_{v}\| - \frac{1}{m_{0}} \|\nabla \tilde{u}_{v_{t}}\| \|\tilde{u}_{v}\| \|\nabla g\|_{\infty}.
		\end{aligned}
	\end{equation}
	For the first term on the right-hand side of \eqref{eq:1.98}, by H\"older inequality we have
	\begin{equation}
		\label{eq:3.91}
		\begin{aligned}
			-(\nu \Delta \mathcal{L}_{g}\tilde{u}_{v}, \varrho^{2}_{k}\tilde{u}_{v})_{g} &= -\nu \int_{\mathcal{O}}^{} \mathcal{L}_{g} \tilde{u}_{v}\cdot (\varrho^{2}_{k} \Delta\tilde{u}_{v} + \tilde{u}_{v}(\Delta \varrho^{2}_{k}) + 2\nabla \varrho^{2}_{k}\cdot \nabla \tilde{u}_{v} ) g dx
			\\& \le -\nu \|\varrho_{k} \Delta \tilde{u}_{v}\|^{2}  + \frac{\nu }{m_{0}} \|\varrho_{k} \nabla \tilde{u}_{v}\|\|\nabla g\|_{\infty}\|\| \varrho_{k} \Delta \tilde{u}_{v}\|^{2} - \frac{\nu}{k} \|\Delta \tilde{u}_{v}\|\|\tilde{u}_{v}\| 
			\\& + \frac{\nu}{k^{2}m_{0}}\|\nabla \tilde{u}_{v}\|\|\tilde{u}_{v}\|\|\nabla g\|_{\infty} - \frac{4\nu }{k}\|\Delta \tilde{u}_{v}\|\| \varrho_{k} \nabla \tilde{u}_{v}\| + \frac{4\nu}{m_{0}k} \|\nabla \tilde{u}_{v}\|\|\varrho_{k}\nabla \tilde{u}_{v}\|\|\nabla g\|_{\infty}
			\\& \le - \frac{\nu }{2} \|\varrho_{k} \Delta \tilde{u}_{v}\|^{2} + \frac{\nu^{2}\lambda_{1} }{512} \|\varrho_{k} \nabla \tilde{u}_{v}\|^{2} + \frac{2\nu C}{k^{2}}\|\Delta \tilde{u}_{v}\|\|\tilde{u}_{v}\|
			\\& + \frac{\nu \lambda^{1/2}_{1}}{16k^{2} m_{0}} \|\nabla \tilde{u}_{v}\|\|\tilde{u}_{v}\| + \frac{4\nu}{k}\|\Delta \tilde{u}_{v}\|\|\nabla \tilde{u}_{v}\| + \frac{\nu \lambda_{1}}{k}\|\nabla \tilde{u}_{v}\|^{2}.
		\end{aligned}
	\end{equation}
	Using (\cite{ref Wang. B4} , Lemma 3.1) we obtain
	\begin{equation}
		\label{eq:4.82}
		\begin{aligned}
			\lambda_{1} \|\varrho_{k} \nabla \tilde{u}\|^{2} &\le \|\varrho_{k} \Delta \tilde{u}\|^{2} + \dfrac{c^{2}_{3}}{k^{2}}\|\tilde{u}\|_{H^{1}(\mathcal{O})}	 + \frac{2c_{3}}{k} \|\varrho_{k} \Delta \tilde{u}\|\|\tilde{u}\|_{H^{1}(\mathcal{O})}
			\\& \le \frac{2\lambda_{1}}{2\lambda_{1} -\alpha}\|\varrho_{k} \Delta \tilde{u}\|^{2}  + \dfrac{2\lambda_{1} c^{2}_{3}}{\alpha k^{2}}\|\tilde{u}\|_{H^{1}(\mathcal{O})}.
		\end{aligned}
	\end{equation}
	By \eqref{eq:3.91} and \eqref{eq:4.82} we obtain
	\begin{equation}
		\label{eq:3.92}
		\begin{aligned}
			-(\nu \Delta \mathcal{L}_{g}\tilde{u}_{v}, \varrho^{2}_{k}\tilde{u}_{v})_{g} \le  \frac{\nu}{2}(\frac{1}{2}\alpha - \lambda_{1})\|\varrho_{k} \nabla \tilde{u}_{v}\|^{2} +  \frac{\nu \lambda^{1/2}_{1}c_{4}}{k} \|\Delta\tilde{u}_{v}\|^{2}  ,
		\end{aligned}
	\end{equation}
	where $c_{4} =c_{4}(C, m_{0}) > 0$ is independent of $k$. \par 
	For the second term on the right-hand side of \eqref{eq:1.98}, by the divergence formula, Ladyzhenskaya, H\"older, and Young's inequalities we have
	\begin{equation}
		\label{eq:3.93}
		\begin{aligned}
			 (\tilde{B}_{g}(\tilde{u}_{v},\tilde{u}_{v}), \varrho^{2}_{k}\tilde{u}_{v} )_{g} &= \int_{\mathcal{O}}^{}\left( -\partial_{ x_{2}}[ \frac{1}{g}\partial_{ x_{1}}\tilde{u}_{v} (\mathcal{L}_{g}\tilde{u}_{v})] +\partial_{ x_{1}}[ \frac{1}{g}\partial_{ x_{2}}\tilde{u}_{v} (\mathcal{L}_{g}\tilde{u}_{v})]   \right)\varrho^{2}_{k}\tilde{u}_{v}gdx
			\\& = - \int_{\mathcal{O}}^{} \frac{1}{g} \partial_{ x_{2}}\tilde{u}_{v}\cdot \Delta \tilde{u}_{v}\cdot \tilde{u}_{v} \cdot \partial_{ x_{1}}(\varrho^{2}_{k})dx  - \int_{\mathcal{O}}^{} \nabla(\frac{1}{g})  \partial_{ x_{2}} \tilde{u}_{v}\cdot \nabla\tilde{u}_{v}\cdot \tilde{u}_{v} \cdot \partial_{ x_{1}}(\varrho^{2}_{k})dx
			\\& \le \frac{2C}{m_{0}k} \|\Delta \tilde{u}_{v}\|\|\tilde{u}_{v}\|_{L^{4}}\|\partial _{x_{2}}\tilde{u}_{v}\|_{L^{4}(\mathcal{O})} + \frac{\lambda^{1/2}_{1} C }{8m_{0}k}\|\nabla \tilde{u}_{v}\|\|\tilde{u}_{v}\|_{L^{4}(\mathcal{O})}\|\partial _{x_{2}}\tilde{u}_{v}\|_{L^{4}(\mathcal{O})}
			\\& \le \frac{2\sqrt{2}C}{m_{0}k \lambda_{1}} \|\Delta \tilde{u}_{v}\|^{3/2} \|\nabla \tilde{u}_{v}\|^{3/2} + \frac{2^{1/4}}{8m_{0}k} \|\nabla \tilde{u}_{v}\|^{5/2}\|\Delta \tilde{u}_{v}\|^{1/2}
			\\& \le \frac{ c_{5}}{m_{0}k} \left(2\|\Delta \tilde{u}_{v}\|^{2} +  \|\nabla \tilde{u}_{v}\|^{6}  + \|\nabla \tilde{u}_{v}\|^{10/3}   \right)
			\\& \le \frac{ c_{6}}{m_{0}k} \left(\|\Delta \tilde{u}_{v}\|^{2} +  \|\nabla \tilde{u}_{v}\|^{6} \right), 
		\end{aligned}
	\end{equation}
	where $c_{6} = c_{6}(C, \lambda_{1}) > 0$ is independent of $k$. The final step uses the relevant Sobolev embedding to control the term $\|\nabla \tilde{u}_{v}\|^{10/3}$.
	\par
	For the last term on the right-hand side of \eqref{eq:1.98}, by condition (\textbf{A3}), Poincar\'e inequality and   H\"older inequality we have
	\begin{equation}
		\label{eq:3.94}
		\begin{aligned}
		- (\sqrt{\varepsilon} \tilde{\sigma}(t,u_{v})v(t),\varrho^{2}_{k}\tilde{u}_{v})_{g} &= \sqrt{\varepsilon} \int_{\mathcal{O}}^{} ( - \sigma_{2}(t,u_{v}), \sigma_{1}(t,u_{v}))v(t) \cdot \nabla (\varrho^{2}_{k}\tilde{u}_{v})gdx
		\\& \le \sqrt{\varepsilon }M_{0}\int_{\mathcal{O}}^{}|(-\sigma_{2}(t,u_{v}), \sigma_{1}(t,u_{v}))| |v(t)| \cdot(\varrho^{2}_{k} |\nabla \tilde{u}_{v}| +  \frac{2C}{k}|\varrho_{k}\tilde{u}_{v}|)dx 
		\\& \le \sqrt{\varepsilon L}R_{3}M_{0}\|\varrho_{k} \nabla \tilde{u}_{v}\|^{2} + \frac{2C \sqrt{\varepsilon L} R_{3} M_{0}}{\lambda^{1/2}_{1} k} \int_{\mathcal{O}}^{}\|\nabla \tilde{u}_{v}\| \|\varrho_{k} \nabla \tilde{u}_{v}\|dx
		\\& \le 2\sqrt{\varepsilon L}R_{3}M_{0}\|\varrho_{k} \nabla \tilde{u}_{v}\|^{2}  + \frac{C^{2}\sqrt{\varepsilon L}R_{3}M_{0}}{\lambda_{1} k^{2}}\|\nabla \tilde{u}_{v}\|^{2}.
		\end{aligned}
	\end{equation}
	Substituting \eqref{eq:3.90}- \eqref{eq:3.94} into \eqref{eq:1.98}, we get 
	\begin{equation}
		\label{eq:3.96}
		\begin{aligned}
	\frac{d}{dt} \|\varrho_{k}\nabla \tilde{u}_{v}\|^{2} &\le \frac{4}{k}\|\nabla \tilde{u}_{v_{t}}\| \|\tilde{u}_{v}\| + \frac{\lambda^{1/2}_{1}}{8} \|\nabla \tilde{u}_{v_{t}}\|\|\tilde{u}_{v}\| + (\frac{1}{2}\nu\alpha - \nu\lambda_{1} + 2\sqrt{\varepsilon L}R_{3}M_{0})\|\varrho_{k} \nabla \tilde{u}_{v}\|^{2} 
	\\& +  \frac{2\lambda^{1/2}_{1}c_{4}}{k}\|\Delta\tilde{u}_{v}\|^{2} + \frac{ c_{2}}{m_{0}k} \left(\|\Delta \tilde{u}_{v}\|^{2} + \|\nabla \tilde{u}_{v}\|^{6} \right) + \frac{C^{2}\sqrt{\varepsilon L}R_{3}M_{0}}{\lambda_{1} k^{2}}\|\nabla \tilde{u}_{v}\|^{2}. 
		\end{aligned}
	\end{equation}
	By Poincar\'e inequality, the identity $|\nabla \tilde{u}_{v}| = |u_{v}|$ and $\alpha < \frac{\nu \lambda_{1} - 2\sqrt{\varepsilon L}R_{3}M_{0}}{\nu }$, we get
	\begin{equation}
		\label{eq:3.97}
		\begin{aligned}
		\frac{d}{dt} \|\varrho_{k}u_{v}\|^{2} &\le \frac{4}{k \lambda^{1/2}_{1}}\|u_{v_{t}}\| \|u_{v}\| + \frac{1}{8} \|u_{v_{t}}\|\|u_{v}\| -\frac{\nu \alpha }{2}\|\varrho_{k} u_{v}\|^{2} 
		+  \frac{2\lambda^{1/2}_{1}c_{4}}{k}\|\nabla u_{v}\|^{2} 
		\\& + \frac{c_{6}}{m_{0}k} \left(\|\nabla u_{v}\|^{2} + \|u_{v}\|^{6}\right)+ \frac{C^{2}\sqrt{\varepsilon L}R_{3}M_{0}}{\lambda_{1} k}\|u_{v}\|^{2},
		\end{aligned}
	\end{equation}
	where $c_{4}, c_{6} > 0$ depend on $C, \lambda_{1}$, and $\varepsilon$, but not on $k$. \par 
	Then, by Gronwall's inequality, we get for any $t \in [0, T]$, 
	\begin{equation}
		\label{eq:3.98}
		\begin{aligned}
	&	\|\varrho_{k} u_{v}(t)\|^{2} \le \|\varrho_{k} u_{0}\|^{2}e^{- \frac{\nu \alpha}{2}t} 
		\\& + e^{ \frac{\nu \alpha}{2}t} \times \int_{0}^{t} \left( \frac{4}{k \lambda^{1/2}_{1} }\| u_{v_{t}}(s)\|\|u_{v}(s)\|  + \frac{1}{8}\| u_{v_{t}}(s)\|\|u_{v}(s)\| + \frac{2\lambda^{1/2}_{1}c_{4}}{k}\|\nabla u_{v}\|^{2} \right)ds
		\\& + e^{ \frac{\nu \alpha}{2}t} \times	\left(  \frac{c_{6}}{m_{0} k}(\|\nabla u_{v}(s)\|^{2}   + \|u_{v}(s)\|^{6} )\right)ds
		\\& +  e^{ \frac{\nu \alpha}{2}t} \times \int_{0}^{t} \left(  \frac{C^{2}\sqrt{\varepsilon L}R_{3}M_{0}}{\lambda_{1} k^{2}}\|u_{v}(s)\|^{2} \right)ds.
		\end{aligned}
	\end{equation}
	For the ﬁrst term on the right-hand side of \eqref{eq:3.98}, by $u_{0} \in H_{g}$ and Lebesgue's Dominated Convergence Theorem, we can obtain that 
	\begin{equation}
		\label{eq:4.73}
		\begin{aligned}
			e^{\frac{- \nu \alpha }{2} t}\|\varrho_{k} u_{0}\|^{2}  \le \int_{\{(x_{1},x_{2}) \in \mathcal{O}:|x_{1}| \ge \frac{k}{2} \}}^{} |u(0,x)|^{2}dx \to 0,\,\, as\,\, k \to \infty.
		\end{aligned}
	\end{equation}
For the second and third term on the right-hand side of \eqref{eq:3.98}, by Lemma \ref{lem:1.13} and   Lemma 2.2 in Wang \cite{ref Wang. B4} we have 
	\begin{equation}
		\label{eq:4.79}
		\begin{aligned}
		e^{ \frac{\nu \alpha}{2}t} \times \int_{0}^{t} \left( \frac{4}{\lambda^{1/2}_{1} k}\| u_{v_{t}}(s)\|\|u_{v}(s)\| + \frac{1}{8}\| u_{v_{t}}(s)\|\|u_{v}(s)\|  + \frac{2\lambda^{1/2}_{1}c_{4}}{k}\|\nabla u_{v}\|^{2} \right) ds \to 0,\,\, as\,\, k \to \infty.
		\end{aligned}
	\end{equation}
	For the fourth  term on the right-hand side of \eqref{eq:3.98}, by Proposition \ref{prop:1.22} and   Lemma 2.2 in Wang \cite{ref Wang. B4} we have
	\begin{equation}
		\label{eq:3.99}
		\begin{aligned}
	 e^{ \frac{\nu \alpha}{2}t} \times  \left(\int_{0}^{t} \frac{c_{6}}{m_{0}k}\left(  \|\nabla u_{v}(s)\|^{2} + \|u_{v}(s)\|^{6} \right) ds\right) \to 0, \,as\,\, k \to \infty. 
		\end{aligned}
	\end{equation}
	For the last term on the right-hand side of \eqref{eq:3.98}, by $\| v \|_{L^{2}(0, T ; H_{0,g})} \le R_{3}$ and Lemma \ref{lem:1.13} we have
	\begin{equation}
		\label{eq:4.80}
		\begin{aligned}
		e^{ \frac{\nu \alpha}{2}t} \times \int_{0}^{t} \left(  \frac{C^{2}\sqrt{\varepsilon L}R_{3}M_{0}}{\lambda_{1} k^{2}}\|u_{v}(s)\|^{2} \right)ds \to 0, \,as\,\, k \to \infty. 
		\end{aligned}
	\end{equation}
	It follows from \eqref{eq:4.73}-\eqref{eq:4.80} that for every $\varepsilon' > 0 $, there exists $K^{*} = K^{*}(\alpha, \lambda_{1}, T, u_{0}, R_{3}, \varepsilon') > 0$ such that for all $k \ge K^{*}$ and $t \in [0, T]$,
	\begin{equation}
		\label{eq:3.100}
		\begin{aligned}
			\int_{|x| \ge k}^{} |u_{v}(t,x)|^{2}dx \le  \|\varrho_{k} u_{v}(t)\|^{2}  < \varepsilon'.
		\end{aligned}
	\end{equation}
	This proof is complete.
\end{proof}
\par
We now define $\mathcal{G}^{0}: C([0,T]; H_{g}) \to C([0,T]; H_{g}) \cap L^{2}([0,T]; V_{g})$ by, for every $\xi \in C([0, T]; H_{0,g})$, such that $u^{\varepsilon}(\cdot) = \mathcal{G}^{\varepsilon}(W(\cdot))$ a.s.\par 
For $\phi \in C([0,T], H_{g})$, denote by
\begin{equation}
	\label{eq:3.37}
	\begin{aligned}
		I(\phi) = \inf\left\{\dfrac{1}{2}\int_{0}^{T}\|v(t)\|^{2}_{0}dt: v \in L^{2}(0,T; H_{0,g}), u_{v} =\phi  \right\},
	\end{aligned}
\end{equation}
where $u_{v}$ is the solution of \eqref{eq:1.56}. The infimum of the empty set is taken to be $\infty$.
\par 
	In what follows, we prove that the solution of \eqref{eq:2.6} satisfies the Laplace principle in $C([0, T], H_{g}) \cap L^{2}([0,T]; V_{g})$ with rate function as given by \eqref{eq:3.37}. Therefore, we first show $\mathcal{G}^{\circ}$ satisfies (\textbf{H1}).
	\begin{lem}
		\label{lem:3.6}
		Suppose that $(\textbf{A1})$--$(\textbf{A3})$ hold and $f \in L^{2}(0,T;V^{*}_{g})$. For any $ 0 \le M < \infty$, the set
	\begin{equation}
		\label{eq:4.36}
		\begin{aligned}
			K_{M}:= \left\{u_{v} \in C([0,T] ;  H_{g}) \cap L^{2}(0,T; V_{g}) : v\in \mathcal{S}_{M}   \right\},
		\end{aligned}
	\end{equation}
	where $u_{v}$ is the unique solution in $C([0,T] ;  H_{g}) \cap L^{2}(0,T; V_{g})$ of the equation \eqref{eq:1.56}. Then $K_{ M}$ is compact in $C([0,T] ;  H_{g}) \cap L^{2}(0,T; V_{g})$.
	\end{lem}
	\begin{proof}
	The proof consists of four steps. Let $\{u_{v_{n}}\}^{\infty}_{n=1}$ be a sequence in $K_{ M }$, where $u_{v_{n}}$ corresponds to the solution of \eqref{eq:1.56}. By weak compactness of $\mathcal{S}_{M}$, there exists a subsequence $\{{v_{n_{k}}}\}^{\infty}_{k=1}$ of $K_{ M }$ which converges to a limit $v \in \mathcal{S}_{M}$ weakly in $L^{2}(0, T; H_{0,g})$. Next, we prove that $u_{v_{n_{k}}} \to u_{v}$ strongly in $C([0,T]; H_{g}) \cap L^{2}([0,T]; V_{g})$. 
	\par
	\textbf{Step 1: establish weak convergence}. Since $v_{n} \to  v$ weakly in $L^{2}(0, T; H_{0,g})$, we know  the sequence $\{v_{n}\}^{\infty}_{n=1}$ is bounded in $L^{2}(0, T; H_{0,g})$, and thus by Lemma \ref{lem:1.13} we see that there exists $C_{1} = C_{1}(T, N) > 0$ such that for all $n \in \mathbb{N}$,
	\begin{equation}
		\label{eq:4.1}
		\begin{aligned}
		\| u_{v}\|^{2}_{g} + \| u_{v_{n}}\|^{2}_{g} + \|u_{v_{n}}\|^{2}_{V_{g}} \le C_{1}.
		\end{aligned}
	\end{equation}
	By \eqref{eq:4.1} and Proposition \ref{prop:1.4} we infer that there exist $ \tilde{\mathbf{u}} \in L^{2}(0, T; V_{g}) \cap  L^{\infty}(0,T; H_{g})$, $F_{0}  \in L^{2}(0, T; V_{g})$ and a subsequence of $\{u_{v_{n_{k}}}\}^{\infty}_{k=1}$ (not relabeled) such that
	\begin{equation}
		\label{eq:4.2}
		\begin{aligned}
			u_{v_{n_{k}}} \to \ \tilde{\mathbf{u}}\,\,\,\, weak-star \,\,in \,\, L^{\infty}(0, T; H_{g}),
		\end{aligned}
	\end{equation}
	\begin{equation}
		\label{eq:4.61}
		\begin{aligned}
			u_{v_{n_{k}}} \to \ \tilde{\mathbf{u}}\,\,\,\, weakly \,\,in \,\, L^{2}(0, T; V_{g}),
		\end{aligned}
	\end{equation}
	\begin{equation}
		\label{eq:4.78}
		\begin{aligned}
		F(	u_{v_{n_{k}}}) \to \ F_{0}\,\,\,\, weakly \,\,in \,\, L^{2}(0, T; V_{g}).
		\end{aligned}
	\end{equation}
	\par 
	\textbf{Step 2: prove the convergence}.
		\begin{equation}
		\label{eq:4.3}
		\begin{aligned}
			u_{v_{n_{k}}} \to \ \tilde{\mathbf{u}}\,\,\,\, strongly \,\,in \,\, C(0, T; H_{g}).
		\end{aligned}
	\end{equation}
	To that end, we will verify that the sequence $\{u_{v_{n_{k}}}\}^{\infty}_{k=1}$ is equicontinuous in $H_{g}$ on $[0, T]$, and $\{u_{v_{n_{k}}}\}^{\infty}_{k=1}$ is precompact in $H_{g}$ for every $t\in [0, T ]$. 
	For $0 \le  r \le t \le T$ and \eqref{eq:3.35}, we have for every $\eta \in H_{g},$
	\begin{equation}
		\label{eq:4.4}
		\begin{aligned}
		&(u_{v_{n_{k}}}(t) -u_{v_{n_{k}}}(r) , \eta)_{g} + \nu\int_{r}^{t}(A_{g}u_{v_{n_{k}}}(s), \eta)_{g}ds + \int_{r}^{t}(B_{g}(u_{v_{n_{k}}}(s), u_{v_{n_{k}}}(s) ), \eta)_{g}ds + \nu\int_{r}^{t}(R_{g}u_{v_{n_{k}}}(s), \eta)_{g}ds 
		\\& = \int_{r}^{t}\int_{\mathcal{O}}^{}f(s)\cdot\eta gdxds + \int_{r}^{t}(\sigma(s,u_{v_{n_{k}}}(s))v_{n_{k}}(s), \eta)_{g}ds.
		\end{aligned}
	\end{equation}
	For the second term on the left-hand side of \eqref{eq:4.4}, by the definition of $A_{g}$ we have
	\begin{equation}
		\label{eq:4.5}
		\begin{aligned}
		\left|\nu \int_{r}^{t}(A_{g}u_{v_{n_{k}}}(s), \eta)_{g}ds  \right| &\le \nu \int_{r}^{t}|(A_{g}u_{v_{n_{k}}}(s), \eta)_{g}|ds \le \nu(t-r)^{\frac{1}{2}} \|u_{v_{n_{k}}}\|_{V_{g}}\|\eta\|_{g}
		\\& \le \nu C^{\frac{1}{2}}_{1}(t-r)^{\frac{1}{2}}\|\eta\|_{g}.
		\end{aligned}
	\end{equation}
	For the third  term on the left-hand side of \eqref{eq:4.4}, by \eqref{eq:2.4} we have
	\begin{equation}
		\label{eq:4.6}
		\begin{aligned}
		\left| \int_{r}^{t}(B_{g}(u_{v_{n_{k}}}(s), u_{v_{n_{k}}}(s) ), \eta)_{g}ds    \right| &\le \int_{r}^{t}|b_{g}(u_{v_{n_{k}}}(s), u_{v_{n_{k}}}(s) ,\eta)|ds
		\\& \le c(t-r)^{\frac{1}{2}}\|u_{v_{n_{k}}}(s)\|^{\frac{1}{2}}_{g} \|u_{v_{n_{k}}}(s)\|^{\frac{1}{2}}_{V_{g}}\|u_{v_{n_{k}}}(s)\|^{\frac{1}{2}}_{V_{g}}\|A_{g}u_{v_{n_{k}}}(s)\|^{\frac{1}{2}}_{g}\|\eta\|_{g}
		\\& \le c\,\,C^{2}_{1}\, (t-r)^{\frac{1}{2}} \|\eta\|_{g}. 
		\end{aligned}
	\end{equation}
	For the fourth  term on the left-hand side of \eqref{eq:4.4}, by \eqref{eq:2.11} we have
	\begin{equation}
		\label{eq:4.7}
		\begin{aligned}
			\left|\nu \int_{r}^{t}(R_{g}u_{v_{n_{k}}}(s), \eta)_{g}ds  \right| &\le \nu \int_{r}^{t}|(R_{g}u_{v_{n_{k}}}(s), \eta)_{g}|ds \le \nu(t-r)^{\frac{1}{2}} \|R_{g}u_{v_{n_{k}}}\|_{{g}}\|\eta\|_{g}
			\\& \le \nu(t-r)^{\frac{1}{2}}\|\eta\|_{g}\frac{\|\nabla g\|_{\infty}}{m_{0}}\|u_{v_{n_{k}}}(s)\|_{V_{g}}
			\\& \le \nu(t-r)^{\frac{1}{2}} \frac{\lambda^{\frac{1}{2}}_{1} C^{\frac{1}{2}}_{1}}{16}\|\eta\|_{g}.
		\end{aligned}
	\end{equation}
	For the second term on the right-hand side of \eqref{eq:4.4}, by H\"older inequality we obtain
	\begin{equation}
		\label{eq:4.8}
		\begin{aligned}
	\left| \int_{r}^{t}\int_{\mathcal{O}}^{}f(s)\cdot\eta gdxds   \right| \le \int_{r}^{t}\int_{\mathcal{O}}^{}|f(s)\cdot\eta| gdxds \le (t-r)^{\frac{1}{2}}\|f\|_{V^{*}}  \|\eta\|_{g}. 
		\end{aligned}
	\end{equation}
	For the last term on the right-hand side of \eqref{eq:4.4}, by condition (\textbf{A2}) and \eqref{eq:4.1}, we obtain
	\begin{equation}
		\label{eq:4.9}
		\begin{aligned}
		\left| \int_{r}^{t}(\sigma(s,u_{v_{n_{k}}}(s))v_{n_{k}}(s), \eta)_{g}ds  \right| &\le \|\eta\|_{g}\int_{r}^{t} \|\sigma(s, u_{v_{n_{k}}}(s))\|_{L_{Q}}\|v_{n_{k}}(s)\|_{L^{2}(0,T;H_{0,g})}ds
		\\& \le (t-r)^{\frac{1}{2}} (MK(1 + C^{2}_{1}))^{\frac{1}{2}}\|\eta\|_{g}
		\end{aligned}
	\end{equation}
	It follows from \eqref{eq:4.5}-\eqref{eq:4.9} that there exists $C_{5} = C_{5}(\nu, \lambda_{1}, T,  M, K, C_{1}) > 0$ such that for all $0 \le r \le t \le  T$ and $\eta \in H_{g}$, 
	\begin{equation}
		\label{eq:4.10}
		\begin{aligned}
		\left|(u_{v_{n_{k}}}(t) - u_{v_{n_{k}}}(r) , \eta)_{g} \right| \le C_{5}(t-r)^{\frac{1}{2}}\|\eta\|_{g}. 
		\end{aligned}
	\end{equation}
	which shows that for all $0 \le r \le t \le  T$,
		\begin{equation}
		\label{eq:4.11}
		\begin{aligned}
			\|u_{v_{n_{k}}}(t) - u_{v_{n_{k}}}(r)\|_{g}  \le C_{5}(t-r)^{\frac{1}{2}}. 
		\end{aligned}
	\end{equation}
	Hence the sequence $\{u_{v_{n_{k}}}\}^{\infty}_{k=1}$ is equicontinuous in $H_{g}$ on $[0, T]$. \par 
	Next, we show for every $t \in [0, T]$,  the set $\{u_{v_{n_{k}}}\}^{\infty}_{k=1}$ is precompact in $V_{g}$.  Let $\varrho$ be the cut--off function as given by \eqref{eq:1.97} and $\varrho_{m}(x) =\varrho(\frac{x}{m})$ for $m \in \mathbb{N}$.  Then we have decomposed $u_{v_{n_{k}}}(t) = \tilde{\textbf{{u}}}^{m}_{v_{n_{k}}}(t) + \hat{\textbf{u}}^{m}_{v_{n_{k}}}(t)$, where $\tilde{\textbf{{u}}}^{m}_{v_{n_{k}}}(t, x) = \varrho_{m}(x)u_{v_{n_{k}}}(t, x)$ and $\hat{\textbf{u}}^{m}_{v_{n_{k}}}(t, x) = (1 - \varrho_{m}(x))u_{v_{n_{k}}}(t, x)$. By \eqref{eq:4.1}, we find that there exists $\textbf{C}_{4} =  \textbf{C}_{4}(T, N) > 0$ such that for all $k,m \in \mathbb{N}$, 
	\begin{equation}
		\label{eq:4.12}
		\begin{aligned}
		\| 	\hat{\textbf{u}}^{m}_{v_{n_{k}}}\|_{{g}} \le \textbf{C}_{4}.
		\end{aligned}
	\end{equation}
	For $m \in \mathbb{N}$, define $\textbf{Z}_{m} = \{u \in H_{g}: \|u\|_{g} \le  \textbf{C}_{4}, \,\, u(x)=0\,\, for\,\, almost\,\, all \,\, |x| \ge m \}$. By \eqref{eq:4.12} and the definition of $\hat{\textbf{u}}^{m}_{v_{n_{k}}}$, it follows that $\{\hat{\textbf{u}}^{m}_{v_{n_{k}}}\}^{\infty}_{k=1}  \subseteq \textbf{Z}_{m} $ for every $m \in \mathbb{N}$ and $t \in [0,T]$. Thanks to the compact embedding of $H^{1}(|x| < m) \hookrightarrow L^{2}(|x| < m)$, we deduce that the set $\textbf{Z}_{m}$ is precompact in $V_{g}$. Therefore, for every $m\in \mathbb{N}$ and $t \in [0,T]$, 
	\begin{equation}
		\label{eq:4.13}
		\begin{aligned}
	\{ 	\hat{\textbf{u}}^{m}_{v_{n_{k}}} \}^{\infty}_{k=1} \,\, is\,\, precompact\,\, in\,\, V_{g}.
		\end{aligned}
	\end{equation}
	In addition, by Lemma \ref{lem:3.57}, we find that for every $\varepsilon'> 0$, there exists $K^{*} = K^{*}(T, N, u_{0}, \varepsilon') > 0$ such that for all $m \ge  K^{*}$, $k \in \mathbb{N}$ and $t \in [0,T]$, 
	\begin{equation}
		\label{eq:4.14}
		\begin{aligned}
		\|\tilde{\textbf{{u}}}^{m}_{v_{n_{k}}}(t)\|_{{g}} \le  \left(\int_{|x| \ge \frac{m}{2}}^{}   |\textbf{{u}}_{v_{n_{k}}}(t,x)|^{2}dx \right)^{\frac{1}{2}} \le \frac{\varepsilon'}{2}. 
		\end{aligned}
	\end{equation}
	By \eqref{eq:4.13} and \eqref{eq:4.14}, we obtain that $\{\tilde{\textbf{{u}}}^{m}_{v_{n_{k}}}(t) + \hat{\textbf{u}}^{m}_{v_{n_{k}}}(t)\}^{\infty}_{k=1}$ has a finite open cover with radius $\varepsilon'$ in $V_{g}$, which implies that $\{u_{v_{n_{k}}}(t)\}^{\infty}_{k=1}$ is precompact in $V_{g}$. 
	\par 
	Therefore, by \eqref{eq:4.2} and the Arzela-Ascoli Theorem imply the asserted strong convergence in $C([0,T];H_g)$.
	\par 
	\textbf{Step 3: prove the convergence}.
	\begin{equation}
		\label{eq:4.62}
		\begin{aligned}
			u_{v_{n_{k}}} \to \ u_{v}\,\,\,\, strongly \,\,in \,\, C(0, T; H_{g}).
		\end{aligned}
	\end{equation}
	To prove \eqref{eq:4.62}, by \eqref{eq:4.3} we only need to show $\tilde{\mathbf{u}} = u_{v}$. Equation \eqref{eq:3.35} gives, for every $t \in [0,T]$ and $\eta \in H_{g}$,
	\begin{equation}
		\label{eq:4.63}
		\begin{aligned} 
		&(u_{v_{n_{k}}}(t), \eta)_{g} + \int_{0}^{t} (\nu A_{g}u_{v_{n_{k}}}(s), \eta)_{g} ds + \int_{0}^{t} (B_{g}(u_{v_{n_{k}}}(s), u_{v_{n_{k}}}(s)), \eta)_{g} ds 
		+ \int_{0}^{t} (\nu R_{g}u_{v_{n_{k}}}(s), \eta)_{g} ds 
		\\& = (u_{0}, \eta)_{g} + \int_{0}^{t}(f(s), \eta)_{g}ds + \int_{0}^{t}(\sigma(s,u_{v_{n_{k}}}(s))v_{n_{k}}(s), \eta)_{g}ds.
		\end{aligned}
	\end{equation}
	We now prove $\tilde{\mathbf{u}} = u_{v}$ by taking the limit of \eqref{eq:4.63} as $k \to \infty$. We ﬁrst deal with the nonlinear term on the left-hand side of \eqref{eq:4.63}. By \eqref{eq:4.3} we infer that, up to a subsequence
	\[
			u_{v_{n_{k}}}  \to \tilde{\mathbf{u}} \,\,\,a.e.\,\, on \,\,  (0, T ) \times \mathbb{R}^{2}.
	\]	
	By the definition of $F$ in Proposition \ref{prop:1.4} and its local monotonicity property, we obtain 
	\begin{equation}
		\label{eq:4.89}
		\begin{aligned}
		F(u_{v_{n_{k}}}(s)) \to F_{0}(\tilde{\mathbf{u}}(s))\,\, weakly\,\, in\,\, L^{2}(0,T; V_{g}). 
		\end{aligned}
	\end{equation}
	For the last term on the right-hand side of \eqref{eq:4.63}, we have for any $t \in [0,T]$, and $\eta \in H_{g}$,
	\begin{equation}
		\label{eq:4.65}
		\begin{aligned}
		&\int_{0}^{t}(\sigma(s,u_{v_{n_{k}}}(s))v_{n_{k}}(s), \eta)_{g}ds 
		\\& = \int_{0}^{t}((\sigma(s,u_{v_{n_{k}}}(s)) - \sigma(s,\tilde{\mathbf{u}}(s) ))v_{n_{k}}(s), \eta)_{g}ds + \int_{0}^{t}(\sigma(s,\tilde{\mathbf{u}}(s))v_{n_{k}}(s), \eta)_{g}ds. 
		\end{aligned}
	\end{equation}
	By \eqref{eq:4.3}, condition (\textbf{A3}) and H\"older inequality, we obtain
	\begin{equation}
		\label{eq:4.66}
		\begin{aligned}
		& \left| \int_{0}^{t}((\sigma(s,u_{v_{n_{k}}}(s)) - \sigma(s,\tilde{\mathbf{u}}(s) ))v_{n_{k}}(s), \eta)_{g}ds          \right| 
		\\& \le \|\eta\| \int_{0}^{t} \|\sigma(s,u_{v_{n_{k}}}(s)) - \sigma(s,\tilde{\mathbf{u}}(s)\|_{L_{Q}} \|v_{n_{k}}(s)\|_{L^{2}(0,T; H_{0,g})} ds
		\\& \le  \sqrt{M L}  \|\eta\| \left( \int_{0}^{t} \| u_{v_{n_{k}}}(s) - \tilde{\mathbf{u}}(s)\|^{2}_{g} ds      \right)^{1/2} \,\, \to\,\, 0\,\,as\,\,k\to\infty. 
		\end{aligned}
	\end{equation}
	We now deal with the last term of \eqref{eq:4.65}. Define $\mathcal{T} : L^{2}(0,T; H_{0,g}) \to H_{g}$ by
	\[
	\mathcal{T}(\omega) = \int_{0}^{t}\sigma(s,\tilde{\mathbf{u}}(s))\omega(s) g ds,\,\, \forall\, \omega \in L^{2}(0,T; H_{0,g}). 
	\]
	It follows from condition (\textbf{A2}) and \eqref{eq:4.2} that $\mathcal{T}$ is well-defined, and
	\begin{align*}
    \| \mathcal{T}(\omega) \|^{2}_{g} 
    & \le \left(\int_{0}^{t}\| \sigma(s,\tilde{\mathbf{u}}(s))\|^{2}_{L^{2}(H_{0,g}; H_{g})} ds   \right)\left(\int_{0}^{t}\|\omega(s)\|^{2}_{L^{2}(0,T; H_{0,g})} \right)
    \\& \le \left(\int_{0}^{t}\left(K(1+ \|\tilde{\mathbf{u}}(s)\|^{2}_{V_{g}}) \right)  ds   \right) \left(\int_{0}^{t}\|\omega(s)\|^{2}_{L^{2}(0,T; H_{0,g})} \right),
	\end{align*}
	which shows $\mathcal{T}$ is a bounded linear operator and thus it is weakly continuous. Since $v_{n_{k}} \rightharpoonup v$ in $ L^{2}(0,T; H_{0,g})$, we have $\mathcal{T}(v_{n_{k}}) \to \mathcal{T}(v)$ weakly in $H_{g}$. Hence, for all $t \in [0,T]$, 
	\begin{equation}
		\label{eq:4.68}
		\begin{aligned}
		\underset{k\to \infty}{\lim} \int_{0}^{t}(\sigma(s,\tilde{\mathbf{u}}(s))v_{n_{k}}(s), \eta)_{g}ds = \int_{0}^{t}(\sigma(s,\tilde{\mathbf{u}}(s))v(s), \eta)_{g}ds.
		\end{aligned}
	\end{equation}
	It follows from \eqref{eq:4.65}--\eqref{eq:4.68} that we have for any $t\in [0,T]$ and $\eta \in H_{g}$,
\begin{equation}
	\label{eq:4.87}
	\begin{aligned}
	\underset{k \to \infty}{\lim }\int_{0}^{t}(\sigma(s, u_{v_{n_{k}}}(s))v_{n_{k}}(s) ,\eta  )_{g}ds = \int_{0}^{t}(\sigma(s,\tilde{\mathbf{u}}(s))v(s), \eta)_{g}ds.
	\end{aligned}
\end{equation}
	Now, letting $k \to \infty$ in \eqref{eq:4.63}, by \eqref{eq:4.2}, \eqref{eq:4.89} and \eqref{eq:4.87}, we deduce that for any $t \in [0,T]$ and $\eta \in H_{g}$, 
	\begin{equation}
		\label{eq:4.69}
		\begin{aligned} 
			&(\tilde{\mathbf{u}}(t), \eta)_{g} + \int_{0}^{t} (\nu A_{g}\tilde{\mathbf{u}}(s), \eta)_{g} ds + \int_{0}^{t} (B_{g}(\tilde{\mathbf{u}}(s), \tilde{\mathbf{u}}(s)), \eta)_{g} ds 
			+ \int_{0}^{t} (\nu R_{g}\tilde{\mathbf{u}}(s), \eta)_{g} ds 
			\\& = (u_{0}, \eta)_{g} + \int_{0}^{t}(f(s), \eta)_{g}ds + \int_{0}^{t}(\sigma(s,\tilde{\mathbf{u}}(s))v(s), \eta)_{g}ds,
		\end{aligned}
	\end{equation}
	which shows that $\tilde{\mathbf{u}}$ is a solution of \eqref{eq:1.56}. By the uniqueness of 
	solutions,we have $ \tilde{\mathbf{u}} = u_{v}$, which together with \eqref{eq:4.3} implies \eqref{eq:4.62}.
	\par 
	\textbf{Step 4: prove the convergence}.
	\begin{equation}
		\label{eq:4.47}
		\begin{aligned}
		u_{v_{n_{k}}} \to  u_{v} \,\, strongly\,\,  in\,\,  C([0, T ]; H_{g} ) \cap L^{2}(0, T;V_{g}).
		\end{aligned}
	\end{equation}
	By \eqref{eq:2.10} and \eqref{eq:1.56} we obtain
	\begin{equation}
		\label{eq:3.16}
		\begin{aligned}
		&\frac{d}{dt} \|u_{v_{n_{k}}}(t) - u_{v}(t)\|^{2}_{g} + 2 \nu   \|u_{v_{n_{k}}}(t) - u_{v}(t)\|^{2}_{V_{g}} 
		\\& + (B_{g}(u_{v_{n_{k}}}(t), u_{v_{n_{k}}}(t)) - B_{g}(u_{v}(t), u_{v}(t)), u_{v_{n_{k}}}(t) - u_{v}(t))_{g}
		\\& \le \dfrac{\nu}{8}  \|u_{v_{n_{k}}}(t) - u_{v}(t)\|^{2}_{V_{g}} + 2  (\sigma(s, u_{v_{n_{k}}}(t))v_{n_{k}}(t) - \sigma(s, u_{v}(t))v(t),\, u_{v_{n_{k}}}(t) - u_{v}(t) )_{g}.
		\end{aligned}
	\end{equation}
	For the third term on the left-hand side of \eqref{eq:3.16}, by \eqref{eq:2.3}, Ladyzhenskaya inequality and Young's inequality  we have
	\begin{equation}
		\label{eq:3.17}
		\begin{aligned}
		&	\left|   (B_{g}(u_{v_{n_{k}}}(t), u_{v_{n_{k}}}(t), u_{v_{n_{k}}}(t) - u_{v}(t)) -B_{g}(u_{v}(t), u_{v}(t), u_{v_{n_{k}}}(t) - u_{v}(t)))_{g}\right|  
		\\& \le   \left| b_{g}(u_{v_{n_{k}}}(t), u_{v_{n_{k}}}(t), u_{v_{n_{k}}}(t) - u_{v}(t)) - b_{g}(u_{v_{n_{k}}}(t), u_{v}(t), u_{v_{n_{k}}}(t) - u_{v}(t)) \right|
		\\& +   \left| b_{g}(u_{v}(t), u_{v_{n_{k}}}(t), u_{v_{n_{k}}}(t) - u_{v}(t)) - b_{g}(u_{v}(t), u_{v}(t), u_{v_{n_{k}}}(t) - u_{v}(t)) \right|
		\\& \le   \left| b_{g}(u_{v_{n_{k}}}(t), u_{v_{n_{k}}}(t) - u_{v}(t), u_{v_{n_{k}}}(t) - u_{v}(t)) \right| +    \left| b_{g}(u_{v_{n_{k}}}(t) - u_{v}(t), u_{v}(t), u_{v_{n_{k}}}(t) - u_{v}(t)) \right|
		\\& \le   \|u_{v_{n_{k}}}(t) - u_{v}(t)\|^{2}_{L^{4}(\mathcal{O})} \|\nabla u_{v}(t)\|_{g}
		\\& \le \frac{\sqrt{2M_{0}}}{m_{0}}  \|u_{v_{n_{k}}}(t) - u_{v}(t)\|_{g}\|u_{v_{n_{k}}}(t) - u_{v}(t)\|_{V_{g}}\|u_{v}(t)\|_{V_{g}} 
		\\& \le \frac{7\nu }{8}  \|u_{v_{n_{k}}}(t) - u_{v}(t)\|^{2}_{V_{g}}  + \dfrac{2M_{0}}{7\nu m^{2}_{0}} \|u_{v_{n_{k}}}(t) - u_{v}(t)\|^{2}_{g} \|u_{v}(t)\|^{2}_{V_{g}}. 
		\end{aligned}
	\end{equation}
	For the last term on the right-hand side of \eqref{eq:3.16}, by H\"older inequality  we have
	\begin{equation}
		\label{eq:4.45}
		\begin{aligned}
		&2 \left|  (\sigma(s, u_{v_{n_{k}}}(t))u_{v_{n_{k}}}(t) - \sigma(s, u_{v}(t)))v(t),\, u_{v_{n_{k}}}(t) - u_{v}(t) )_{g} \right|
		\\& \le 2  \left( \|\sigma(s, u_{v_{n_{k}}}(t))\|_{L_{Q}}\|v_{n_{k}}(t)\|_{L^{2}(0,T; H_{0,g})} + \|\sigma(s, u_{v}(t))\|_{L_{Q}}\|v(t)\|_{L^{2}(0,T; H_{0,g})}\right) \|u_{v_{n_{k}}}(t) - u_{v}(t)\|_{g}.
		\end{aligned}
	\end{equation}
	By condition (\textbf{A2}) and \eqref{eq:4.1}, there exists a constant $C_{6}= C_{6} (N, T)>0$ such that for any $t \in [0,T]$, 
	\begin{align*}
		 2  \left( \|\sigma(s, u_{v_{n_{k}}}(t))\|_{L_{Q}} + \|\sigma(s, u_{v}(t))\|_{L_{Q}} \right) \le C_{6},
	\end{align*}
	which, along with \eqref{eq:3.16} shows, that for all $t \in [0,T]$,
	\begin{equation}
		\label{eq:4.90}
		\begin{aligned}
	&\frac{d}{dt} \|u_{v_{n_{k}}}(t) - u_{v}(t)\|^{2}_{g} + \nu   \|u_{v_{n_{k}}}(t) - u_{v}(t)\|^{2}_{V_{g}} 
		  \\& \le   \dfrac{2M_{0}}{7\nu m^{2}_{0}} \|u_{v_{n_{k}}}(t) - u_{v}(t)\|^{2}_{g} \|u_{v}(t)\|^{2}_{V_{g}}
		  + C_{6}\left( \|v_{n_{k}}(t)\|_{L^{2}(0,T; H_{0,g})} + \|v(t)\|_{L^{2}(0,T; H_{0,g})} \right)\|u_{v_{n_{k}}}(t) - u_{v}(t)\|_{g}.
		\end{aligned}
	\end{equation}
	By \eqref{eq:3.17} - \eqref{eq:4.90} and \eqref{eq:4.1}  we obtain for all $ t \in [0,T]$,
	\begin{equation}
		\label{eq:3.18}
		\begin{aligned}
	 \|u_{v_{n_{k}}}(t) - u_{v}(t)\|^{2}_{g}  &+ \nu \int_{0}^{t} \|u_{v_{n_{k}}}(t) - u_{v}(s)\|^{2}_{V_{g}}ds
	\\& \le \dfrac{2M_{0}}{7\nu m^{2}_{0}}\int_{0}^{t}\|u_{v_{n_{k}}}(s) - u_{v}(s)\|^{2}_{g} \|u_{v}(s)\|^{2}_{V_{g}} ds 
	\\& + C_{6}\int_{0}^{t}\left( \|v_{n_{k}}(s)\|_{L^{2}(0,T; H_{0,g})} + \|v(s)\|_{L^{2}(0,T; H_{0,g})} \right)\|u_{v_{n_{k}}}(s) - u_{v}(s)\|_{g}ds
	\\& \le \dfrac{2M_{0} C_{1}}{7\nu m^{2}_{0}}\|u_{v_{n_{k}}}(s) - u_{v}(s)\|^{2}_{g} 
	 + 2C_{6} M\|u_{v_{n_{k}}} - u_{v}\|_{g},
	\end{aligned}
	\end{equation}
	where $C_{1}$ and $C_{6}$ are positive constants that depend on $N$, $T$. Then by \eqref{eq:4.1} and \eqref{eq:4.62} we see that the right-hand side of \eqref{eq:3.18} converges to zero as $n \to \infty$, from which \eqref{eq:4.47} follows. 
	\par 
	This completes the proof.
	\end{proof}
	Recall that there exists a strong solution of \eqref{eq:2.6}  with values in $C([0, T ]; H_{g} ) \cap L^{2}(0, T ; V_{g} )$, and it is pathwise unique. It follows that there exists a Borel-measurable function $\mathcal{G}^{0}: C([0, T ]; H_{g} ) \to  C([0, T ]; H_{g} ) \cap L^{2}(0, T; V_{g} )$ such that $u (\cdot) = \mathcal{G}^{\varepsilon}(W(\cdot))$ almost surely. Note that for every $M > 0$, $\mathcal{S}_{M} = \{v\in L^{2}(0,T; H_{0,g}) : \int_{0}^{T}\|v(t)\|^{2}_{0} dt \le M \}$. 
	\begin{lem}
		\label{lem:3.38}
		Suppose (\textbf{A1})--(\textbf{A3}) hold, and $ f\in L^{2}(0,T; V^{*}_{g})$, $ v \in \mathcal{A}_{M}$ with $0 < M < \infty$. If $ u^{\varepsilon}_{v} = \mathcal{G}^{\varepsilon}\left(W(\cdot) + \frac{1}{\sqrt{\varepsilon}} \int_{0}^{\cdot}v(s)ds\right)$, then $u^{\varepsilon}_{v}$ is the unique solution of
	\end{lem}
	\begin{equation}
		\label{eq:3.39}
		\begin{aligned}
			\begin{split}
				\frac{d}{dt}u^{\varepsilon}_{v}(t) + \nu A_{g}u^{\varepsilon}_{v}(t) + B_{g}(u^{\varepsilon}_{v}(t), u^{\varepsilon}_{v}(t)) + \nu R_{g}u^{\varepsilon}_{v}(t) = f(t)+ \sigma(t,u^{\varepsilon}_{v}(t))v(t) + \sqrt{\varepsilon}\sigma(t,u^{\varepsilon}_{v})dW(t),
			\end{split}
		\end{aligned}
	\end{equation}
	with initial value $u^{\varepsilon}_{v}(0) = u_{0} \in H_{g}$.\par 
	Furthermore, for each $R_{4} > 0, T > 0, N >0$,  there exists $M_{1} = M_{1}(R_{4}, T , M, K, \nu) > 0$ such that for any $u_{0} \in H_{g}$ with $\| u_{0}\|_{g} < R_{4}$ and any $ v \in \mathcal{A}_{M}$, the solution $u^{\varepsilon}_{v}$ satisfies for all $\varepsilon \in (0, 1)$,
	\begin{equation}
		\label{eq:2.32}
		\begin{aligned}
			\textbf{E} \left( \|u^{\varepsilon}_{v}\|^{2}_{C([0,T]; H_{g})}  + \|u^{\varepsilon}_{v}\|^{2}_{L^{2}(0,T; V_{g})}  \right) \le M_{1}. 
		\end{aligned}
	\end{equation}
	\begin{proof}
		Since $v \in \mathcal{A}_{M}$, $\int_{0}^{T}\|v(s)\|^{2}_{0}ds \le  M$ a.s.,  $\widetilde{W}(\cdot):= W(\cdot)+ \frac{1}{\sqrt{\varepsilon}} \int_{0}^{\cdot} v(s)ds$ is a Wiener process  with covariance form $Q$ under the probability measure $\widetilde{P}^{\varepsilon}_{v}$:
		\begin{equation}
			\label{eq:3.40}
			\begin{aligned}
				\frac{d\widetilde{P}^{\varepsilon}_{v}}{dP} = \exp\left\{ - \dfrac{1}{\sqrt{\varepsilon}} \int_{0}^{T}v(s)dW(s) - \dfrac{1}{2\varepsilon}\int_{0}^{T} \|v(s)\|^{2}_{0} ds\right\}.
			\end{aligned}
		\end{equation}
		By Girsanov's theorem, we know that $u^{\varepsilon}_{v}$ is the unique  strong solution of \eqref{eq:1.1} with $W$ replaced by $\widetilde{W}$, and hence $u^{\varepsilon}_{v}$ is the unique solution of \eqref{eq:3.39} with initial condition $u^{\varepsilon}_{v}(0) = u_{0} \in H_{g}$. 
		\par 
		Next, we prove \eqref{eq:2.32}. By \eqref{eq:3.93} and It\^o formula, taking supremum and expectation we have for $t \in [0,T]$,
	\begin{equation}
		\label{eq:2.35}
		\begin{aligned}
		\textbf{E} \underset{0 \le t\le T}{\sup} \|u^{\varepsilon}_{v}(t)\|^{2}_{g} &+ 2\nu \textbf{E} \underset{0 \le t\le T}{\sup} \int_{0}^{t} \|u^{\varepsilon}_{v}(s)\|^{2}_{V_{g}}ds 
		 +  2\nu \textbf{E} \underset{0 \le t\le T}{\sup} \int_{0}^{t} (R_{g}u^{\varepsilon}_{v}(s), u^{\varepsilon}_{v}(s) )_{g}ds 
		\\& = \|u_{0}\|^{2}_{g} + 2 \textbf{E} \underset{0 \le t\le T}{\sup} \int_{0}^{t} (f(s), u^{\varepsilon}_{v}(s))_{g}ds + 2 \textbf{E} \underset{0 \le t\le T}{\sup} \int_{0}^{t} (\sigma(s, u^{\varepsilon}_{v}(s))v(s), u^{\varepsilon}_{v}(s))_{g}ds
		\\& + \varepsilon \textbf{E} \underset{0 \le t\le T}{\sup} \int_{0}^{t}\|\sigma(s, u^{\varepsilon}_{v}(s))\|^{2}_{L_{Q}}ds + 2\sqrt{\varepsilon}\textbf{E} \underset{0 \le t\le T}{\sup} \int_{0}^{t} (\sigma(s, u^{\varepsilon}_{v}(s))dW(s), u^{\varepsilon}_{v}(s))_{g}ds.
		\end{aligned}
	\end{equation}
	For the last term on the left-hand side of \eqref{eq:2.35}, by \eqref{eq:2.10} we have
	\begin{equation}
		\label{eq:4.81}
		\begin{aligned}
		2\nu \textbf{E} \underset{0 \le t\le T}{\sup}\left|  \int_{0}^{t} (R_{g}u^{\varepsilon}_{v}(s), u^{\varepsilon}_{v}(s) )_{g}ds \right| \le \frac{\nu }{8} \textbf{E} \underset{0 \le t\le T}{\sup} \int_{0}^{t} \|u^{\varepsilon}_{v}(s)\|^{2}_{V_{g}}ds.
		\end{aligned}
	\end{equation}
	For the second term on the right-hand side of \eqref{eq:2.35}, by Young's inequality we have
	\begin{equation}
			\label{eq:2.36}
			\begin{aligned}
			 2 \textbf{E} \underset{0 \le t\le T}{\sup} \left| \int_{0}^{t} (f(s), u^{\varepsilon}_{v}(s))_{g}ds \right| & \le 2 \textbf{E} \underset{0 \le t\le T}{\sup} \int_{0}^{t}\|f(s)\|_{V^{*}}\|u^{\varepsilon}_{v}(s)\|_{g} ds 	
			 \\& \le \frac{1}{\nu}\textbf{E} \underset{0 \le t\le T}{\sup} \int_{0}^{t} \|f(s)\|^{2}_{V^{*}} ds + \nu  \textbf{E} \underset{0 \le t\le T}{\sup} \int_{0}^{t} \|u^{\varepsilon}_{v}(s)\|^{2}_{V_{g}}ds. 
			\end{aligned}
		\end{equation}
	For the third term on the right-hand side of \eqref{eq:2.35}, by condition (\textbf{A2}) and Young's inequality we have
	\begin{equation}
		\label{eq:2.37}
		\begin{aligned}
	& 2 \textbf{E} \underset{0 \le t\le T}{\sup} \left| \int_{0}^{t} (\sigma(s, u^{\varepsilon}_{v}(s))v(s), u^{\varepsilon}_{v}(s))_{g}ds \right|  \le 2 \textbf{E} \underset{0 \le t\le T}{\sup} \int_{0}^{t} |\sigma(s, u^{\varepsilon}_{v}(s))||v(s)||u^{\varepsilon}_{v}(s)|gdxds
	\\& \le 	\frac{2K}{\nu }\textbf{E} \underset{0 \le t\le T}{\sup} \int_{0}^{t} \|u^{\varepsilon}_{v}(s)\|^{2}_{g}\|v(s)\|^{2}_{L^{2}(0,T; H_{0,g})}ds  +  	\frac{\nu }{2K}\textbf{E} \underset{0 \le t\le T}{\sup} \int_{0}^{t} \|\sigma(s, u^{\varepsilon}_{v}(s))\|^{2}_{L_{Q}}ds
	\\& \le \frac{2K}{\nu} \textbf{E} \underset{0 \le t\le T}{\sup} \int_{0}^{t} \|u^{\varepsilon}_{v}(s)\|^{2}_{g}\|v(s)\|^{2}_{L^{2}(0,T; H_{0,g})}ds  +  	\frac{\nu }{2K}\textbf{E} \underset{0 \le t\le T}{\sup} \int_{0}^{t} K( 1+ \|u^{\varepsilon}_{v}(s)\|^{2}_{V_{g}}) ds. 
		\end{aligned}
	\end{equation}	
	For the fourth term on the right-hand side of \eqref{eq:2.35}, by condition (\textbf{A2}) we have	
	\begin{equation}
		\label{eq:2.38}
		\begin{aligned}
		\varepsilon \textbf{E} \underset{0 \le t\le T}{\sup} \int_{0}^{t}\|\sigma(s, u^{\varepsilon}_{v}(s))\|^{2}_{L_{Q}}ds \le \varepsilon \textbf{E} \underset{0 \le t\le T}{\sup} \int_{0}^{t}K(1+ \|u^{\varepsilon}_{v}(s)\|^{2}_{V_{g}}) ds.
		\end{aligned}
	\end{equation}	
	For the last term on the right-hand side of \eqref{eq:2.35}, by condition (\textbf{A2}) and BDG inequality we have	
		\begin{equation}
			\label{eq:2.39}
			\begin{aligned}
		&2\sqrt{\varepsilon}\textbf{E} \underset{0 \le t\le T}{\sup} \left| \int_{0}^{t} (\sigma(s, u^{\varepsilon}_{v}(s))dW(s), u^{\varepsilon}_{v}(s))_{g}ds\right| 
		\\& \le 2 \sqrt{\varepsilon}c_{2}\textbf{E} \left( \int_{0}^{T} \|\sigma(s, u^{\varepsilon}_{v}(s))\|^{2}_{L_{Q}} \|u^{\varepsilon}_{v}(s)\|^{2}_{g}ds \right)^{1/2} 
		\\& \le \frac{1}{2}\textbf{E}  \underset{0 \le t\le T}{\sup}  \|u^{\varepsilon}_{v}(t)\|^{2}_{g} + 2\varepsilon c^{2}_{2} \textbf{E} \int_{0}^{T} K(1+ \|u^{\varepsilon}_{v}(s)\|^{2}_{V_{g}})ds. 
			\end{aligned}
		\end{equation}
	By \eqref{eq:2.36}-\eqref{eq:2.39} we obtain that for all $t \in [0, T ]$ with $\varepsilon < \frac{\nu }{8K(1+2c^{2}_{2})}$, 
	\begin{equation}
		\label{eq:2.40}
		\begin{aligned}
 & \textbf{E}\left( \frac{1}{2} \underset{0 \le t\le T}{\sup}  \left( \|u^{\varepsilon}_{v}(t)\|^{2}_{g} + \nu  \int_{0}^{t} \|u^{\varepsilon}_{v}(s)\|^{2}_{V_{g}}ds \right) \right)
 \\& \le \|u_{0}\|^{2}_{g} + \frac{1}{\nu}\textbf{E} \underset{0 \le t\le T}{\sup} \int_{0}^{t} \|f(s)\|^{2}_{V^{*}} ds + \frac{2K}{\nu} \textbf{E} \underset{0 \le t\le T}{\sup} \int_{0}^{t} \|u^{\varepsilon}_{v}(s)\|^{2}_{g}\|v(s)\|^{2}_{L^{2}(0,T; H_{0,g})}ds + \frac{\nu T}{8}.
		\end{aligned}
	\end{equation}	
	Since $u_{0}\in H_{g}$ with $\| u_{0}\| \le R_{4}$ and $v \in \mathcal{A}_{M}$, then by Gronwall’s inequality implies that for all $t \in [0, T ]$, 	
		\begin{equation}
			\label{eq:2.49}
			\begin{aligned}
		&\textbf{E}\left(   \underset{0 \le t\le T}{\sup}\left( \|u^{\varepsilon}_{v}(t)\|^{2}_{g} + \nu  \int_{0}^{t} \|u^{\varepsilon}_{v}(s)\|^{2}_{V_{g}}ds \right)  \right) 
		\\& \le \left(2R^{2}_{4} + \frac{2}{\nu}\textbf{E}\int_{0}^{T} \|f(s)\|^{2}_{V^{*}} ds + \frac{\nu T}{4} \right) e^{\frac{4K}{\nu }\int_{0}^{T}\|v(s)\|^{2}_{L^{2}(0,T; H_{0,g})} ds}, 
			\end{aligned}
		\end{equation}
	which implies \eqref{eq:2.32} and thus completes the proof.	
	\end{proof}
	Next, the following lemma is concerned with the convergence of $\mathcal{G}^{\varepsilon}$ as $\varepsilon \to 0$, which is necessary to prove (\textbf{H2}).
	\begin{lem}
		\label{lem:1.20}
		Suppose (\textbf{A1})--(\textbf{A3}) hold, $T> 0$ and $ \{v^{\varepsilon} \} \subseteq  \mathcal{A}_{M}$ with $ 0 < M < \infty$. If $\{v^{\varepsilon} \}$ converges in distribution to $v$ as $\mathcal{S}_{M}$-valued random variables, then $\mathcal{G}^{\varepsilon}\left( W + \varepsilon^{-\frac{1}{2}}\int_{0}^{\cdot} v^{\varepsilon}(t)dt \right)$ converges to $\mathcal{G}^{0}\left( \varepsilon^{-\frac{1}{2}}\int_{0}^{\cdot} v^{\varepsilon}(t)dt \right)$  in $C([0,T]; H_{g}) \cap L^{2}(0,T; V_{g})$ in distribution. 
		\end{lem}	
	\begin{proof}
		Let	 $u^{\varepsilon}_{v^{\varepsilon}}  = \mathcal{G}^{\varepsilon}\left( W + \varepsilon^{-\frac{1}{2}}\int_{0}^{\cdot} v^{\varepsilon}(t)dt \right)$. By Lemma  \ref{lem:3.6}, $u^{\varepsilon}_{v^{\varepsilon}}$ satisfies 
		\begin{equation}
			\label{eq:3.22}
			\begin{aligned}
	du^{\varepsilon}_{v^{\varepsilon}}(t) + (\nu A_{g}u^{\varepsilon}_{v^{\varepsilon}}(t) + B_{g}(u^{\varepsilon}_{v^{\varepsilon}}(t), u^{\varepsilon}_{v^{\varepsilon}}(t)) + \nu R_{g}u^{\varepsilon}_{v^{\varepsilon}}(t))dt &= f(t)dt + \sigma(t, u^{\varepsilon}_{v^{\varepsilon}}(t))v^{\varepsilon}(t)dt
	\\& +\sqrt{\varepsilon}\sigma(t, u^{\varepsilon}_{v^{\varepsilon}}(t))dW(t),
			\end{aligned}
		\end{equation}
	with initial condition $u^{\varepsilon}_{v^{\varepsilon}}(0) = u_{0} \in H_{g}$. Let $u_{v^{\varepsilon}} = \mathcal{G}^{0}\left(\int_{0}^{\cdot} v^{\varepsilon}(t)dt \right)$. Then $u_{v^{\varepsilon}}$ is the solution of 
	\begin{equation}
		\label{eq:3.25}
		\begin{aligned}
			du_{v^{\varepsilon}}(t) + (\nu A_{g}u_{v^{\varepsilon}}(t) + B_{g}(u_{v^{\varepsilon}}(t), u_{v^{\varepsilon}}(t)) + \nu R_{g}u_{v^{\varepsilon}}(t))dt = f(t)dt + \sigma(t, u_{v^{\varepsilon}}(t))v^{\varepsilon}(t)dt,
		\end{aligned}
	\end{equation}
	with initial condition $u_{v^{\varepsilon}}(0) = u_{0} \in H_{g}$. We will now show that $ u^{\varepsilon}_{v^{\varepsilon}}(t)$ converges to $u_{v}$ in $C([0, T ]; H_{g}) \cap L^{2}(0, T ; V_{g} )$ in distribution as $\varepsilon \to 0$. By \eqref{eq:3.22} and \eqref{eq:3.25}, we have 
	\begin{equation}
		\label{eq:4.15}
		\begin{aligned}
		&d(u^{\varepsilon}_{v^{\varepsilon}}(t) - u_{v^{\varepsilon}}(t)) + \nu A_{g}(u^{\varepsilon}_{v^{\varepsilon}}(t) - u_{v^{\varepsilon}}(t))dt + (B_{g}(u^{\varepsilon}_{v^{\varepsilon}}(t), u^{\varepsilon}_{v^{\varepsilon}}(t)) - B_{g}(u_{v^{\varepsilon}}(t), u_{v^{\varepsilon}}(t)))dt 
		\\& +\nu R_{g}(u^{\varepsilon}_{v^{\varepsilon}}(t) - u_{v^{\varepsilon}}(t))dt 
		\\&= (\sigma(t, u^{\varepsilon}_{v^{\varepsilon}}(t))v^{\varepsilon}(t) - \sigma(t, u_{v^{\varepsilon}}(t))v^{\varepsilon}(t))dt + \sqrt{\varepsilon}\sigma(t, u^{\varepsilon}_{v^{\varepsilon}})dW(t), 
		\end{aligned}
	\end{equation}
	with initial value $u^{\varepsilon}_{v^{\varepsilon}}(0) - u_{v^{\varepsilon}}(0) = 0$. By \eqref{eq:4.15} and It\^o formula, we obtain
	\begin{equation}
		\label{eq:3.27}
		\begin{aligned}
			&\| u^{\varepsilon}_{v^{\varepsilon}}(t) - u_{v^{\varepsilon}}(t)\|^{2}_{g} + 2\nu  \int_{0}^{t}\| u^{\varepsilon}_{v^{\varepsilon}}(s) - u_{v^{\varepsilon}}(s) \|^{2}_{V_{g}} ds
			 + 2 \nu \int_{0}^{t}(R_{g}(u^{\varepsilon}_{v^{\varepsilon}}(s) - u_{v^{\varepsilon}}(s)), u^{\varepsilon}_{v^{\varepsilon}}(s) - u_{v^{\varepsilon}}(s))_{g}ds 
			 \\& + 2\int_{0}^{t}(B_{g}(u^{\varepsilon}_{v^{\varepsilon}}(s), u^{\varepsilon}_{v^{\varepsilon}}(s)) -B_{g}(u_{v^{\varepsilon}}(s), u_{v^{\varepsilon}}(s)) , u^{\varepsilon}_{v^{\varepsilon}}(s) - u_{v^{\varepsilon}}(s))_{g} ds
			 \\& = 2 \int_{0}^{t} (\sigma(s,u^{\varepsilon}_{v^{\varepsilon}}(s))v^{\varepsilon}(s) - \sigma(s,u_{v^{\varepsilon}}(s))v^{\varepsilon}(s), u^{\varepsilon}_{v^{\varepsilon}}(s) - u_{v^{\varepsilon}}(s))_{g}ds + \varepsilon \int_{0}^{t} tr\sigma(s,u^{\varepsilon}_{v^{\varepsilon}}(s))Q\sigma^{*}(s,u^{\varepsilon}_{v^{\varepsilon}}(s))ds
			\\& + 2\sqrt{\varepsilon}\int_{0}^{t} (\sigma(s,u^{\varepsilon}_{v^{\varepsilon}}(s))dW(s),u^{\varepsilon}_{v^{\varepsilon}}(s) - u_{v^{\varepsilon}}(s))_{g}.
		\end{aligned}	
	\end{equation}
	For the third term on the left-hand side of  \eqref{eq:3.27}, by Lemma \ref{lem:3.65}, we have 
	\begin{equation}
		\label{eq:4.16}
		\begin{aligned}
	2 \nu \left|\int_{0}^{t}(R_{g}(u^{\varepsilon}_{v^{\varepsilon}}(s) - u_{v^{\varepsilon}}(s)), u^{\varepsilon}_{v^{\varepsilon}}(s) - u_{v^{\varepsilon}}(s))_{g}ds\right| 
	 \le \frac{\nu }{8} \int_{0}^{t}\|u^{\varepsilon}_{v^{\varepsilon}}(s) - u_{v^{\varepsilon}}(s)\|^{2}_{V_{g}}ds. 
		\end{aligned}
	\end{equation}
	For the fourth term on the left-hand side of  \eqref{eq:3.27}, by \eqref{eq:2.32}, \eqref{eq:1.16} and Young's inequality, we have 
	\begin{equation}
		\label{eq:4.17}
		\begin{aligned}
		& 2\left| \int_{0}^{t}(B_{g}(u^{\varepsilon}_{v^{\varepsilon}}(s), u^{\varepsilon}_{v^{\varepsilon}}(s)) -B_{g}(u_{v^{\varepsilon}}(s), u_{v^{\varepsilon}}(s)) , u^{\varepsilon}_{v^{\varepsilon}}(s) - u_{v^{\varepsilon}}(s))_{g} \right|
		\\& \le 2\left|\int_{0}^{t} b_{g}(u^{\varepsilon}_{v^{\varepsilon}}(s), u^{\varepsilon}_{v^{\varepsilon}}(s)- u_{v^{\varepsilon}}(s), u^{\varepsilon}_{v^{\varepsilon}}(s) - u_{v^{\varepsilon}}(s))ds \right| 
		\\& +2\left|\int_{0}^{t} b_{g}(u^{\varepsilon}_{v^{\varepsilon}}(s) - u_{v^{\varepsilon}}(s), u_{v^{\varepsilon}}(s), u^{\varepsilon}_{v^{\varepsilon}}(s) - u_{v^{\varepsilon}}(s))ds \right|
		\\& \le 2\int_{0}^{t}\int_{\mathcal{O}}^{} |u^{\varepsilon}_{v^{\varepsilon}}(s) - u_{v^{\varepsilon}}(s)| |\nabla u_{v^{\varepsilon}}(s) ||u^{\varepsilon}_{v^{\varepsilon}}(s) - u_{v^{\varepsilon}}(s)| g dxds
		\\& \le 2\int_{0}^{t} \|u^{\varepsilon}_{v^{\varepsilon}}(s) - u_{v^{\varepsilon}}(s)\|^{2}_{L^{4}} \|\nabla u_{v^{\varepsilon}}(s)\|_{g}ds
		\\& \le \frac{2\sqrt{2M_{0}}}{m_{0}}\int_{0}^{t} \|u^{\varepsilon}_{v^{\varepsilon}}(s) - u_{v^{\varepsilon}}(s)\|_{g}\|u^{\varepsilon}_{v^{\varepsilon}}(s) - u_{v^{\varepsilon}}(s)\|_{V_{g}} \| u_{v^{\varepsilon}}(s)\|_{V_{g}}ds
		\\& \le \frac{7\nu}{8}\int_{0}^{t} \|u^{\varepsilon}_{v^{\varepsilon}}(s) - u_{v^{\varepsilon}}(s)\|^{2}_{V_{g}}ds + \frac{16M_{0}}{7\nu m^{2}_{0}}\int_{0}^{t} \|u^{\varepsilon}_{v^{\varepsilon}}(s) - u_{v^{\varepsilon}}(s)\|^{2}_{g}\| u_{v^{\varepsilon}}(s)\|^{2}_{V_{g}}ds. 
		\end{aligned}
	\end{equation}
	For the first term on the right-hand side of  \eqref{eq:3.27}, by \eqref{eq:2.32}, \eqref{eq:1.16} and H\"older inequality, we have
	\begin{equation}
		\label{eq:4.18}
		\begin{aligned}
		&2 \left| \int_{0}^{t} (\sigma(s,u^{\varepsilon}_{v^{\varepsilon}}(s)v^{\varepsilon}(s) - \sigma(s,u_{v^{\varepsilon}}(s))v^{\varepsilon}(s), u^{\varepsilon}_{v^{\varepsilon}}(s) - u_{v^{\varepsilon}}(s))_{g}ds  \right|
		\\& \le 2\int_{0}^{t}\|\sigma(s,u^{\varepsilon}_{v^{\varepsilon}}(s))-\sigma(s,u_{v^{\varepsilon}}(s)) \|_{L_{Q}} \|v^{\varepsilon}(s)\|_{L^{2}(0,T; H_{0,g})} \|u^{\varepsilon}_{v^{\varepsilon}}(s) - u_{v^{\varepsilon}}(s)\|_{g}ds
		\\& \le 2\sqrt{L} \int_{0}^{t}\|v^{\varepsilon}(s)\|_{L^{2}(0,T; H_{0,g})} \|u^{\varepsilon}_{v^{\varepsilon}}(s) - u_{v^{\varepsilon}}(s)\|^{2}_{g}ds. 
		\end{aligned}
	\end{equation}
	It follows from \eqref{eq:4.16} - \eqref{eq:4.18} that for all $t \in [0,T]$,  
	\begin{equation}
		\label{eq:4.19}
		\begin{aligned}
		& \| u^{\varepsilon}_{v^{\varepsilon}}(t) - u_{v^{\varepsilon}}(t)\|^{2}_{g} + \nu  \int_{0}^{t}\| u^{\varepsilon}_{v^{\varepsilon}}(s) - u_{v^{\varepsilon}}(s) \|^{2}_{V_{g}} ds
		\\&  \le \frac{16M_{0}}{7\nu m^{2}_{0}}\int_{0}^{t} \|u^{\varepsilon}_{v^{\varepsilon}}(s) - u_{v^{\varepsilon}}(s)\|^{2}_{g}\| u_{v^{\varepsilon}}(s)\|^{2}_{V_{g}}ds + 2\sqrt{L} \int_{0}^{t}\|v^{\varepsilon}(s)\|_{L^{2}(0,T; H_{0,g})} \|u^{\varepsilon}_{v^{\varepsilon}}(s) - u_{v^{\varepsilon}}(s)\|^{2}_{g}ds 
		\\& +  \varepsilon \int_{0}^{t} \|\sigma(s,u^{\varepsilon}_{v^{\varepsilon}}(s))\|^{2}_{L_{Q}}ds + 2\sqrt{\varepsilon}\int_{0}^{t} (\sigma(s,u^{\varepsilon}(s)_{v^{\varepsilon}})dW(s),u^{\varepsilon}_{v^{\varepsilon}}(s) - u_{v^{\varepsilon}}(s))_{g}.
		\end{aligned}
	\end{equation}
		Given $ N >0$ and $\varepsilon \in (0,1)$, define a stopping time 
	\begin{equation}
		\label{eq:3.29}
		\begin{aligned}
			\tau_{N,\varepsilon} := inf\left\{t \ge 0:  \|u^{\varepsilon}_{v^{\varepsilon}}(t)\|^{2}_{g}  \ge  N\right\} \wedge T.
		\end{aligned}
	\end{equation}
	By \eqref{eq:4.19} we have for all $t \in [0,T]$,  
	\begin{equation}
		\label{eq:4.20}
		\begin{aligned}
		& \underset{0\le r\le t}{\sup} \left( \| u^{\varepsilon}_{v^{\varepsilon}}(r \wedge \tau_{N,\varepsilon}) - u_{v^{\varepsilon}}(r \wedge \tau_{N,\varepsilon})\|^{2}_{g} + \nu  \int_{0}^{r \wedge \tau_{N,\varepsilon}}\| u^{\varepsilon}_{v^{\varepsilon}}(s) - u_{v^{\varepsilon}}(s) \|^{2}_{V_{g}} ds \right)
		\\&  \le \int_{0}^{r \wedge \tau_{N,\varepsilon}} \left( \frac{16M_{0}}{7\nu m^{2}_{0}}\|u_{v^{\varepsilon}}(s)\|^{2}_{V_{g}} + 2\sqrt{L}\|v^{\varepsilon}(s)\|_{L^{2}(0,T; H_{0,g})}  \right) \underset{0 \le r \le s}{\sup}\|u^{\varepsilon}_{v^{\varepsilon}}(r \wedge \tau_{N,\varepsilon}) - u_{v^{\varepsilon}}(r \wedge \tau_{N,\varepsilon})\|^{2}_{g}ds 
		\\& +  \varepsilon \int_{0}^{r \wedge \tau_{N,\varepsilon}}\|\sigma(s,u^{\varepsilon}_{v^{\varepsilon}}(s))\|^{2}_{L_{Q}}ds + 2\sqrt{\varepsilon}\underset{0\le r\le t}{\sup}\int_{0}^{r \wedge \tau_{N,\varepsilon}} (\sigma(s,u^{\varepsilon}(s)_{v^{\varepsilon}})dW(s),u^{\varepsilon}_{v^{\varepsilon}}(s) - u_{v^{\varepsilon}}(s))_{g}.
		\end{aligned}
	\end{equation}
	By \eqref{eq:4.20} and Gronwall's inequality, we have that for any $t \in [0,T]$, 
	\begin{equation}
		\label{eq:4.21}
		\begin{aligned}
		& \underset{0\le r\le T}{\sup} \left( \| u^{\varepsilon}_{v^{\varepsilon}}(r \wedge \tau_{N,\varepsilon}) - u_{v^{\varepsilon}}(r \wedge \tau_{N,\varepsilon})\|^{2}_{g} + \nu  \int_{0}^{r \wedge \tau_{N,\varepsilon}}\| u^{\varepsilon}_{v^{\varepsilon}}(s) - u_{v^{\varepsilon}}(s) \|^{2}_{V_{g}} ds \right)
	    \\& \le  \varepsilon e^{c_{8}}  \int_{0}^{r \wedge \tau_{N,\varepsilon}} \|\sigma(s,u^{\varepsilon}_{v^{\varepsilon}}(s) )\|^{2}_{L_{Q}} ds
		+ 2\sqrt{\varepsilon}e^{c_{8}}  \underset{0\le r\le T}{\sup}\int_{0}^{T \wedge \tau_{N,\varepsilon}} (\sigma(s,u^{\varepsilon}_{v^{\varepsilon}}(s))dW(s),u^{\varepsilon}_{v^{\varepsilon}}(s) - u_{v^{\varepsilon}}(s))_{g}, 
		\end{aligned}
	\end{equation}
	where $c_{8} = (\frac{16M_{0}}{7m^{2}_{0}}C_{1} + 2\sqrt{LM})T > 0$.  For the first term on the right-hand side \eqref{eq:4.21}, by condition (\textbf{A2}) and \eqref{eq:4.38}, we have 
	\begin{equation}
		\label{eq:4.22}
		\begin{aligned}
		  \underset{\varepsilon \to 0}{\rm{lim}}\,\varepsilon e^{c_{8}}  \int_{0}^{T \wedge \tau_{N,\varepsilon}} \|\sigma(s,u^{\varepsilon}_{v^{\varepsilon}}(s) )\|^{2}_{L_{Q}} ds
		  \le \underset{\varepsilon \to 0}{\rm{lim}}\,\varepsilon e^{c_{8}} K\int_{0}^{T \wedge \tau_{N,\varepsilon}} (1+ \|u^{\varepsilon}_{v^{\varepsilon}}(s)\|^{2}_{V_{g}}) ds
		 = 0, \mathbb{P} - almost\,surely. 
		\end{aligned}
	\end{equation}
	In addition, by Lemma  \ref{lem:1.13}, we have that there exists $ C'_{1} = C'_{1}(M, T) > 0$ such that 
	\begin{equation}
		\label{eq:4.23}
		\begin{aligned}
			\underset{\varepsilon \in (0, 1)}{\sup} \|u_{v^{\varepsilon}} \|_{g} \le C'_{1}. 
		\end{aligned}
	\end{equation}
	For the last term on the right-hand side of \eqref{eq:4.21}, by \eqref{eq:4.22}, \eqref{eq:4.23} and Doob's maximal inequality, we obtain
	\begin{equation}
		\label{eq:4.24}
		\begin{aligned}
	& \underset{\varepsilon \to 0}{\rm{lim}}   \textbf{E} \left[ \underset{0\le r\le T}{\sup} \left|  \int_{0}^{r \wedge \tau_{N,\varepsilon}}2\sqrt{\varepsilon}\, e^{c_{8}}\, (\sigma(s,u^{\varepsilon}_{v^{\varepsilon}}(s))dW(s),u^{\varepsilon}_{v^{\varepsilon}}(s) - u_{v^{\varepsilon}}(s))_{g} \right|^{2} \right]
	\\& \le  \underset{\varepsilon \to 0}{\rm{lim}} 16 \varepsilon e^{2c_{8}}  \textbf{E} \left[\int_{0}^{T \wedge \tau_{N,\varepsilon}} \|u^{\varepsilon}_{v^{\varepsilon}}(s) - u_{v^{\varepsilon}}(s)\|^{2}_{g} \| \sigma(s,u^{\varepsilon}(s))\|^{2}_{L_{Q}}ds\right]
	\\& \le 16e^{2c_{8}} (N + C'_{1})^{2} \underset{\varepsilon \to 0}{\rm{lim}} \textbf{E} \left[\varepsilon \int_{0}^{ T \wedge \tau_{N,\varepsilon}} \| \sigma(s,u^{\varepsilon}(s))\|^{2}_{L_{Q}} ds\right]
	\\& =0.
		\end{aligned}
	\end{equation}
	It follows from \eqref{eq:4.21}, \eqref{eq:4.22} and \eqref{eq:4.24} that 
	\begin{equation}
		\label{eq:4.25}
		\begin{aligned}
		\underset{\varepsilon \to 0}{\rm{lim}}\,\,  \underset{0\le r\le T}{\sup} \|u^{\varepsilon}_{v^{\varepsilon}}(r \wedge \tau_{N,\varepsilon}) - u_{v^{\varepsilon}}(r \wedge \tau_{N,\varepsilon})\|^{2}_{g} = 0 \,\,in\, \, probability, 
		\end{aligned}
	\end{equation}
	and 
	\begin{equation}
		\label{eq:4.26}
		\begin{aligned}
			\underset{\varepsilon \to 0}{\rm{lim}}  \int_{0}^{T\wedge \tau_{N,\varepsilon}} \|u^{\varepsilon}_{v^{\varepsilon}}(s) - u_{v^{\varepsilon}}(s)\|^{2}_{V_{g}} ds = 0 \,\,in\, \, probability.
		\end{aligned}
	\end{equation}
	On the other hand by  \eqref{eq:2.32} and  \eqref{eq:3.29} we have for all $\varepsilon \in (0,1)$,  
	\begin{equation}
		\label{eq:4.27}
		\begin{aligned}
		\mathbb{P}(\tau_{N,\varepsilon} < T  ) \le \mathbb{P}\left(\underset{t \in [0,T]}{\sup} \|u^{\varepsilon}_{v^{\varepsilon}}(t)\|_{g} \ge N \right) \le \dfrac{1}{N^{2}}\textbf{E}\left( \underset{t \in [0,T]}{\sup} \|u^{\varepsilon}_{v^{\varepsilon}}(t)\|^{2}_{g}\right) \le \frac{M_{1}}{N^{2}},
		\end{aligned}
	\end{equation}
	where $M_{1} = M_{1}( T, M) >0$. It follows from \eqref{eq:4.27} that for every $\eta > 0$,
	\begin{align*}
		& \mathbb{P}\left(\underset{t \in [0,T]}{\sup}\|u^{\varepsilon}_{v^{\varepsilon}}(t) - u_{v^{\varepsilon}}(t) \|_{g}  > \eta \right) 
		\\& \le \mathbb{P}\left(\underset{t \in [0,T]}{\sup}\|u^{\varepsilon}_{v^{\varepsilon}}(t) - u_{v^{\varepsilon}}(t) \|_{g}  > \eta, \tau_{N,\varepsilon} =T \right) + \mathbb{P}\left(\underset{t \in [0,T]}{\sup}\|u^{\varepsilon}_{v^{\varepsilon}}(t) - u_{v^{\varepsilon}}(t) \|_{g}  > \eta, \tau_{N,\varepsilon} < T \right) 
		\\& \le \mathbb{P}\left(\underset{t \in [0,T]}{\sup}\|u^{\varepsilon}_{v^{\varepsilon}}(t \wedge \tau_{N,\varepsilon}) - u_{v^{\varepsilon}}(t \wedge \tau_{N,\varepsilon}) \|_{g}  > \eta \right) + \frac{M_{1}}{N^{2}}.
	\end{align*}
	First letting $\varepsilon \to 0$ and  then $N \to \infty$, we get from \eqref{eq:4.25} that 
	\begin{equation}
		\label{eq:4.29}
		\begin{aligned}
		\underset{\varepsilon \to 0}{\rm{lim}} \left( u^{\varepsilon}_{v^{\varepsilon}} - u_{v^{\varepsilon}} \right) = 0 \,\, in\,\, C([0,T]; H_{g}) \,\, in\, probability.
		\end{aligned}
	\end{equation}
	Similarly, by \eqref{eq:4.27} we have, for every $\eta > 0$, 
	\begin{align*}
		& \mathbb{P}\left(\int_{0}^{T} \|u^{\varepsilon}_{v^{\varepsilon}}(t) - u_{v^{\varepsilon}}(t) \|^{2}_{V_{g}}dt  > \eta^{2} \right) 
		\\& \le \mathbb{P}\left(\int_{0}^{T} \|u^{\varepsilon}_{v^{\varepsilon}}(t) - u_{v^{\varepsilon}}(t) \|^{2}_{V_{g}} dt > \eta^{2}, \tau_{N,\varepsilon} =T \right) + \mathbb{P}\left(\int_{0}^{T} \|u^{\varepsilon}_{v^{\varepsilon}}(t) - u_{v^{\varepsilon}}(t) \|^{2}_{V_{g}}dt  > \eta^{2}, \tau_{N,\varepsilon} < T \right) 
		\\& \le \mathbb{P}\left(\int_{0}^{T \wedge \tau_{N,\varepsilon}} \|u^{\varepsilon}_{v^{\varepsilon}}(t) - u_{v^{\varepsilon}}(t) \|^{2}_{V_{g}}dt  > \eta^{2} \right) + \frac{M_{1}}{N^{2}}.
		\end{align*}
		First letting $\varepsilon \to 0$ and then $N \to \infty$, we get from \eqref{eq:4.26} that 
		\begin{equation}
			\label{eq:4.31}
			\begin{aligned}
				\underset{\varepsilon \to 0}{\rm{lim}} \left( u^{\varepsilon}_{v^{\varepsilon}} - u_{v^{\varepsilon}} \right) = 0 \,\, in\,\, L^{2}(0,T; V_{g}) \,\, in\, probability.
			\end{aligned}
		\end{equation}
		It follows from \eqref{eq:4.29} and \eqref{eq:4.31} that 
		\begin{equation}
			\label{eq:4.32}
			\begin{aligned}
				\underset{\varepsilon \to 0}{\rm{lim}} \left( u^{\varepsilon}_{v^{\varepsilon}} - u_{v^{\varepsilon}} \right) = 0 \,\, in\,\, C([0,T]; H_{g}) \cap  L^{2}(0,T; V_{g}) \,\, in\, probability.
			\end{aligned}
		\end{equation}
		By assumption, we know $\{ v^{\varepsilon} \}$ converges in distribution to $v$ as $\mathcal{S}_{M}$-valued random element, and hence by Skorokhod's Theorem, there exist a probability space $(\tilde{\Omega}, \tilde{\mathcal{F}}, \tilde{\mathbb{P}}  )$ and $\mathcal{S}_{M}$-valued random variables $\tilde{v}^{\varepsilon}$ and $\tilde{v}$ on $(\tilde{\Omega}, \tilde{\mathcal{F}}, \tilde{\mathbb{P}}  )$ such that the law of $\tilde{v}^{\varepsilon}$ and $\tilde{v}$ are equal to that of $v^{\varepsilon}$ and $v$ respectively. Furthermore, $\{\tilde{v}^{\varepsilon}\}$ converges to $\tilde{v}$ almost surely in $\mathcal{S}_{M}$, it follows from Lemma \ref{lem:3.6} that
		\begin{equation}
			\label{eq:4.33}
			\begin{aligned}
			u_{\tilde{v}^{\varepsilon}} \to  u_{\tilde{v}}	\,\, in\,\, C([0,T]; H_{g}) \cap  L^{2}(0,T; V_{g}) \,\, almost\,\, surely,
			\end{aligned}
		\end{equation}
		and hence it is also convergent in distribution, which implies that 
\begin{equation}
	\label{eq:4.34}
	\begin{aligned}
		u_{v^{\varepsilon}} \to  u_{v}	\,\, in\,\, C([0,T]; H_{g}) \cap  L^{2}(0,T; V_{g}) \,\, in \,\, distribution.
	\end{aligned}
\end{equation}
	By \eqref{eq:4.33} and \eqref{eq:4.34} we infer that 
	\begin{equation}
		\label{eq:4.35}
		\begin{aligned}
			u^{\varepsilon}_{v^{\varepsilon}} \to  u_{v}	\,\, in\,\, C([0,T]; H_{g}) \cap  L^{2}(0,T; V_{g}) \,\, in \,\, distribution, 
		\end{aligned}
	\end{equation}
	as desired. The proof is thus completed.
	
	\end{proof}
	\par 
	\textbf{Proof of Theorem \ref{thm:1.12}}. Lemmas \ref{lem:3.6} and \ref{lem:1.20} verify conditions (\textbf{H1}) and (\textbf{H2}), respectively. Proposition \ref{prop:1.11} therefore yields Theorem \ref{thm:1.12}.


\section*{Acknowledgement}
\indent
\par
We thank the anonymous reviewers for their constructive comments. This work was 
partially supported by the National Natural Science Foundation of China(No.12326414), Sichuan 
Science and Technology Program(No.2023NSFSC0076).

\end{document}